\documentclass[10pt,reqno,a4paper]{amsart}

\usepackage{fullpage}
\usepackage{amsmath}
\usepackage{amsfonts}
\usepackage{amssymb}
\usepackage{amsxtra}
\usepackage{amsthm}
\usepackage{bm}
\usepackage{graphicx}
\usepackage{mathrsfs}
\usepackage[colorlinks=true]{hyperref}
\hypersetup{urlcolor=blue, citecolor=red, linkcolor=blue}
\usepackage[usenames,dvipsnames]{color}
\usepackage[misc]{ifsym}

\def\eqdefa{\buildrel\hbox{\footnotesize def}\over =}

\def\S{\mathbb{S}}
\def\T{\mathbb{T}}

\newcommand{\R}{\mathbb{R}}

\newcommand{\E}{\mathcal{E}}
\newcommand{\D}{\mathcal{D}}
\renewcommand{\L}{\mathcal{L}}

\newcommand{\ve}{\varepsilon}

\newcommand{\pa}{\partial}
\newcommand{\na}{\nabla}

\newcommand{\al}{\alpha}
\newcommand{\lam}{\lambda}

\allowdisplaybreaks

\renewcommand{\S}{\mathbb{S}}
\newcommand{\<}{\langle}
\renewcommand{\>}{\rangle}
\newcommand{\I}{\mathbf{I}}
\newcommand{\II}{\mathbf{I}_{\pm}}
\renewcommand{\P}{\mathbf{P}}
\newcommand{\PP}{\mathbf{P}_{\pm}}
\numberwithin{equation}{section}

\usepackage{cite}

\def\eqdefa{\buildrel\hbox{\footnotesize def}\over =}

\newcommand{\vertiii}[1]{{\left\vert\kern-0.25ex\left\vert\kern-0.25ex\left\vert #1 \right\vert\kern-0.25ex\right\vert\kern-0.25ex\right\vert}}

\newtheorem{thm}{Theorem}[section]
\newtheorem{cor}[thm]{Corollary}
\newtheorem{lem}[thm]{Lemma}
\newtheorem{rmk}[thm]{Remark}

{ \theoremstyle{remark} }

\begin{document}

\title[VPB system with polynomial perturbation near Maxwellian]
{The Vlasov-Poisson-Boltzmann/Landau system with polynomial perturbation near Maxwellian}

\author[C. Cao, D. Deng, L. Li]{Chuqi Cao, Dingqun Deng and Xingyu Li}
\address[Chuqi Cao]{Yau Mathematical Science Center and Beijing Institute of Mathematical Sciences and Applications, Tsinghua University\\
Beijing 100084,  P. R.  China.} \email{chuqicao@gmail.com}
\address[Dingqun Deng]{Department of Mathematics, Pohang University of Science and Technology, Pohang, Republic of Korea (South), ORCID: \href{https://orcid.org/0000-0001-9678-314X}{0000-0001-9678-314X}} \email{dingqun.deng@postech.ac.kr} \email{dingqun.deng@gmail.com}
\address[Xingyu Li] {Institut de Mathematiques, Universite de Toulouse III Paul Sabatier     118 Rte de Narbonne, 31400 Toulouse, France.} \email{xingyuli92@gmail.com}


\begin{abstract} 
	In this work, we consider the Vlasov-Poisson-Boltzmann system without angular cutoff and the Vlasov-Poisson-Landau system with Coulomb potential near a global Maxwellian $\mu$ in torus or union of cubes. We establish the global existence, uniqueness, and large-time behavior for solutions in a polynomial-weighted Sobolev space $H^2_{x, v}( \langle v \rangle^k)$ for some constant $k >0$. 
 For the domain union of cubes, We will consider the specular-reflection boundary condition and its high-order compatible specular boundary condition. The proof is based on an extra dissipation term generated from an improved semigroup method including the electrostatic field with the help of macroscopic estimates. 
\end{abstract}
\maketitle
\tableofcontents

\section{Introduction}

\subsection{Models and equations}
We consider the Vlasov-Poisson-Boltzmann (VPB) and Vlasov-Poisson-Landau (VPL) systems describing the motion of plasma particles of two species in 
bounded domain $\Omega\subseteq\mathbb R^3$ (cf. \cite{G6,G8}):
\begin{equation}
\label{vpb1}
\begin{aligned}
	&\pa_tF_+ + v\cdot \na_xF_+ - \na_x\phi\cdot\na_vF_+ = Q(F_+,F_+)+Q(F_-,F_+),\\
	&\pa_tF_- + v\cdot \na_xF_- + \na_x\phi\cdot\na_vF_- = Q(F_+,F_-)+Q(F_-,F_-),\\
	&F(0)=F_0,\quad E(0)=E_0,
\end{aligned}
\end{equation}
here $F(t, x, v) \ge 0$ is a distribution function of particles at time $t>0$ with position $x \in \Omega$ and velocity $v \in \R^3$, $E = -\nabla_x \phi$ is the self-consistent electrostatic field satisfying 
\begin{equation}
\label{vpb2}
-\Delta_x\phi = \int_{\R^3}(F_+-F_-)\,dv, \quad\int_{\Omega}\phi(t, x) dx =0.
\end{equation}
For the VPB system, the bilinear collision operator $Q$ acts only on the velocity variable $v$ given by 
\begin{align*}
 Q(G,F)(v)=\int_{\R^3}\int_{\mathbb{S}^2}B(v-v_*,\sigma )(G'_*F'-G_*F)d\sigma dv_*.
\end{align*}
Here we use the standard notation $F=F(v),$ $G_*=G(v_*)$, $F'=F(v')$, $G'_*=G(v'_*)$, where $v',v'_*$ are velocities of two particles after collision given by
\begin{align*}
v'=\frac{v+v_*}{2}+\frac{|v-v_*|}{2}\sigma, \quad v_*'=\frac{v+v_*}{2}-\frac{|v-v_*|}{2}\sigma, \quad \sigma\in\mathbb{S}^2.
\end{align*}
 This representation follows from the physical law of elastic collision:
\begin{align*}
  v+v_*=v'+v'_*,  \quad
  |v|^2+|v_*|^2=|v'|^2+|v'_*|^2.
\end{align*}
The nonnegative function $B(v-v_*,\sigma)$ is called the Boltzmann collision kernel. It depends only on relative velocity $|v-v_*|$ and the deviation angle $\theta$ through 
$\cos\theta\eqdefa \frac{v-v_*}{|v-v_*|}\cdot\sigma$.

In the present work, we consider the non-cutoff kernel $B$ 
  as the following. 
\begin{itemize}
  \item[$\mathbf{(A1).}$] 
  The Boltzmann kernel $B$ takes the form 
  $B(v-v_*,\sigma)=|v-v_*|^\gamma b(\frac{v-v_*}{|v-v_*|}\cdot\sigma )$,
where   $b$ is a nonnegative angular function.

  \item[$\mathbf{(A2).}$] $b(\cos \theta)$ is not locally integrable and there exists some constant $\mathcal{K}>0$, such that
\begin{align*}
  \mathcal{K}\theta^{-1-2s}\leq \sin\theta b(\cos\theta) \le \mathcal{K}^{-1}\theta^{-1-2s},~\mbox{with}~0<s<1.
\end{align*}

  \item[$\mathbf{(A3).}$]
  The parameter $\gamma$ and $s$ satisfy the condition $-3 < \gamma  \le 1, \quad 1/2 \le s \le 1, \quad \gamma+2s> -1$.

\item[$\mathbf{(A4).}$]  Without loss of generality, we may assume that $B(v-v_*,\sigma )$ is supported in the set $0\leq \theta \le \pi/2$, i.e.$\frac{v-v_*}{|v-v_*|}   \cdot  \sigma  \ge 0$. Otherwise, $B$ can be replaced by its symmetrized form:
\begin{align*}
\overline{B}(v-v_*,\sigma )=|v-v_*|^\gamma\Big(b(\frac{v-v_*}{|v-v_*|}\cdot\sigma) + b(\frac{v-v_*}{|v-v_*|}\cdot(-\sigma ))\Big)  \mathrm{1}_{\frac{v-v_*}{|v-v_*|}\cdot\sigma \ge 0},
\end{align*}
where $\mathrm{1}_A$ is the indicator function of the set $A$.
\end{itemize}

\begin{rmk} For inverse repulsive potential, we have $\gamma = \frac {p-5} {p-1}$ and $s = \frac 1 {p-1}$ with $p > 2$. It is worth noting that the condition $\gamma + 2s > - 1$ is met across the entire spectrum of the inverse power law model. Generally, the cases $\gamma > 0$,  $\gamma = 0$, and  $\gamma < 0$ correspond to so-called hard, Maxwellian, and soft potentials respectively.
\end{rmk}

For the case of the VPL system, $Q$ is the Landau collision operator given by 
	\begin{align*}
		Q(g, f)(v)
= \partial_i \int_{\R^3}  \phi^{ij} (v-v_*) (g_* \partial_j f - f  \partial_j g_*  )    dv_*,\quad\phi^{i j} (v) = |v|^{\gamma+2} \left( \delta_{ij} - \frac {v_i v_j}  {|v|^2}\right), \quad -3 \le \gamma \le 1.
	\end{align*}
We use the convention of summation for repeated indices, and the derivatives are in the velocity variable $\partial_i = \partial_{v_i}$. 
Define the notations $g_* = g(v_*), f = f(v), \partial_j g_* = \partial_{v_{*j}}g(v_*),  \partial_j f = \partial_{v_j} f(v)$.
Notice that
\begin{equation*}
 \partial_j \phi^{ij} (v) = -2|v|^\gamma v_i, \quad\partial_{ij} \phi^{ij}(v) = \left\{\begin{aligned}& -2(\gamma + 3) |v|^\gamma,\ \ \text{ if }  -3 < \gamma \le 1,\\
 	 & - 8 \pi \delta_0,   \qquad\qquad\text{ if }\gamma=-3,
 \end{aligned}\right.
\end{equation*}
where $\delta_0$ is the Dirac measure.
We call it {\it hard potential} if $\gamma\ge 0$, and {\it soft potential} if $\gamma\in[-3,0)$.

\subsection{Reformulation}
 We reformulate the VPB/VPL system near a global Maxwellian.
 For simplicity, we assume the initial data $F_0$ is normalized such that the equilibrium associated with \eqref{vpb1} will be the standard Gaussian function, i.e. $\mu(v)\eqdefa (2\pi)^{-3/2} e^{-|v|^2/2}$, 
 which enjoys the same mass, momentum, and energy as $F_0$.
 For the perturbation framework, we denote $F=[F_+,F_-]$ and let $f=[f_+,f_-]$ satisfies $F_\pm=\mu+f_\pm$. Then system  \eqref{vpb1} and \eqref{vpb2} become
\begin{equation}
	\left\{
	\begin{aligned}
		\label{vplocal1}
		&\partial_t f_\pm + v\cdot\nabla_x f_\pm \mp\nabla_x\phi\cdot\nabla_v f_\pm \pm \nabla_x\phi\cdot v\mu = Q(f_\pm+f_\mp,\mu)+ Q(2\mu+f_\pm+f_\mp, f_\pm),\\
		&-\Delta_x\phi=\int_{\mathbb R^3}{(f_+-f_-)dv},\quad\int_{\Omega}{\phi(x) dx}=0, f(0) = f_0,\quad \phi(0) = \phi_0.
	\end{aligned}\right.
\end{equation}
Note that $Q(\mu,\mu)=0$. 
We also denote linear operator $L=[L_+,L_-]$ and $\L=[\L_+,\L_-]$ by 
\begin{equation}\label{L}
	L_\pm f=2Q(\mu,f_\pm)+Q(f_\pm+f_\mp,\mu),\quad \L f = -v\cdot\nabla_x f \mp\na_x\phi\cdot v\mu+L f,
\end{equation}
where $\phi(x)$ is solved by the second equation of \eqref{vplocal1}. 
%
%
The kernel of $L$ on $L^2_v\times L^2_v$ is the span of $\big\{[1,0]\mu,[0,1]\mu,[1,1]v\mu,[1,1]|v|^2\mu\big\}$ (cf. \cite{G6}) and we define the projection of $L^2_v\times L^2_v$ onto $\ker L$ by 
\begin{equation}\label{P}
	\P f = \Big(a_+(t,x)[1,0]+a_-(t,x)[0,1]+v\cdot b(t,x)[1,1]+(|v|^2-3)c(t,x)[1,1]\Big)\mu,
\end{equation}
or equivalently by 
\[
\PP f = \Big(a_\pm(t,x)+v\cdot b(t,x)+(|v|^2-3)c(t,x)\Big)\mu,
\]
where function $a_\pm,b,c$ are given by 
\begin{equation}
	\begin{aligned}\label{abc}
		a_\pm &= \int_{\R^3} f_\pm dv,\quad 
		b_j= \frac{1}{2}\int_{\R^3} v_i (f_++f_-) dv , \quad 
		c=\int_{\R^3} \frac {|v|^2-3}{12}(f_++f_-) dv. 
	\end{aligned}
\end{equation}
Taking inner produce of the first equation of \eqref{vplocal1} with $1$ over $\R^3_v$, we have the continuity equation 
\begin{align}\label{19}
	\frac {\partial} {\partial t} \int_{\R^3}  f_\pm ( v) dv + \nabla_x \cdot \int_{\R^3} v f_\pm(v)  dv =0.
\end{align}

\subsection{Domains}
In this paper, we suppose the bounded domain $\Omega$ is either a torus or the union of finite cubes. 

\paragraph{\textbf{Torus}}
For the case of a torus, we set $\Omega = \T^3 = [-\pi,\pi]^3$. In this case, the solution $(F,\phi)$ to the VPB/VPL system \eqref{vpb1} with the initial data $F_0$ enjoys the conservation of mass, momentum, and the energy, i.e. 
\begin{equation}
	\begin{aligned}
		\label{conse1}
		&\frac d {dt} \int_{\Omega} \int_{  \R^3} F_+ dv dx =\frac d {dt} \int_{\Omega} \int_{  \R^3} F_- dv dx =0, \quad \frac d {dt}\int_{\Omega} \int_{  \R^3} v (F_++F_-) dv dx =0,\\
		&\frac d {dt} \int_{\Omega} \int_{  \R^3} |v|^2 (F_++F_-) dv dx+\frac {d} {dt}  \int_{\Omega} |\nabla_x \phi  (t, x)|^2 dx=0.
	\end{aligned}
\end{equation}
Suppose $f_0$ have the same mass, momentum and energy as $\mu$, then conservation laws  \eqref{conse1} yield
\begin{equation}
	\begin{aligned}
		\label{conse2.torus}
		&\int_{\Omega} \int_{ \R^3}  f_+ dv dx=\int_{\Omega} \int_{ \R^3}  f_- dv dx  =0,\quad\int_{\Omega} \int_{ \R^3} v (f_++f_-) dv dx = 0,\\
		&\int_{\Omega} \int_{ \R^3} |v|^2(f_++f_-) dv dx  +  \int_{\Omega} |\nabla_x \phi  (t, x)|^2 dx=0.
	\end{aligned}
\end{equation}

\paragraph{\textbf{Union of cubes}}
The second type of bounded domain we consider is a union of finitely many cubes, denoted as:
\begin{align}\label{Omega}
	\Omega=\cup_{i=1}^N\Omega_i,
\end{align}  
where each $\Omega_i$ is defined as a rectangular region: $\Omega_i = (a_{i,1},b_{i,1})\times(a_{i,2},b_{i,2})\times(a_{i,3},b_{i,3})$, with $a_{i,j},b_{i,j}\in\R$ satisfying $a_{i,j} < b_{i,j}$. The boundary of this domain, denoted as $\partial\Omega$, is composed of three types of boundaries, $\Gamma_i$ $(i=1,2,3)$, each of which is orthogonal to one of the coordinate axes $x_i$. These boundaries are further divided into connected sets. It is important to note that we assume that each $\Gamma_i$ has a non-zero spherical measure, while the boundaries of these sets themselves have zero spherical measure. Therefore, we do not distinguish between $\Gamma_i$ and their interiors.

On the interior of $\Gamma_i$ $(i=1,2,3)$, we have a unit normal outer vector $n(x)$ defined almost everywhere with respect to spherical measure. This vector takes the form $n(x) = e_i$ or $-e_i$, where $e_i$ is a unit vector with its $i$th component equal to 1. Additionally, we denote vectors $\tau_1(x)$ and $\tau_2(x)$ on the boundary $\partial\Omega$ such that $(n(x),\tau_1(x), \tau_2(x))$ forms a unit orthonormal basis for $\R^3$. Furthermore, for $j=1,2$, $\tau_j$ takes the values $e_k$ or $-e_k$ for some $k$.


The boundary of the phase space is denoted as $\gamma:=\{(x,v)\in\partial\Omega\times\R^3\}$.
We define $n=n(x)$ to represent the outward normal direction at $x\in\partial\Omega$. This allows us to partition $\gamma$ into three distinct sets:
\begin{align*}
	\gamma_- &= \{(x,v)\in\partial\Omega\times\R^3 : n(x)\cdot v<0\},\quad\text{(the incoming set),}\\
	\gamma_+ &= \{(x,v)\in\partial\Omega\times\R^3 : n(x)\cdot v>0\},\quad\text{(the outgoing set),}\\
	\gamma_0 &= \{(x,v)\in\partial\Omega\times\R^3 : n(x)\cdot v=0\},\quad\text{(the grazing set).}
\end{align*}
For these sets, we assume that $F(t,x,v)$ satisfies a {\it specular-reflection} boundary condition on $\gamma_-$, where $(x,v)\in\gamma$ is subject to the following reflection operation:
\begin{equation}\label{Rx}
	R_xv = v - 2n(x)(n(x)\cdot v). 
\end{equation}
This specular reflection condition for $f$ can be expressed as:
\begin{equation}\label{specular}\begin{aligned}
		f(t,x,R_xv) = f(t,x,v), \ \text{ on } \gamma_-. 
	\end{aligned}
\end{equation}Regarding the boundary condition for the electric potential $\phi$, we further assume that it satisfies a Neumann boundary condition:
\begin{align}\label{Neumann}
	\partial_{n}\phi=0, \ \text{ on }x\in \partial\Omega.  
\end{align}
In particular, the Poisson equation for potential $\phi$ becomes a pure Neumann boundary problem. To ensure the existence of a solution, we impose a zero-mean condition:
\begin{align*}
	\int_\Omega\int_{\R^3}(f_+-f_-)\,dvdx=0,\quad \text{ for } t\ge 0,
\end{align*}
which follows from the conservation laws \eqref{conservation_bounded}. Similar to \eqref{conse2.torus}, it's also well-known that the solution to \eqref{vplocal1} in the bounded domain $\Omega$ given by \eqref{Omega} satisfies the conservation laws on mass and energy. Specifically, the solution $f$ to \eqref{vplocal1} satisfies the following identities when the initial data $f_0$ satisfies them:
\begin{equation}\label{conservation_bounded}
	\begin{aligned}
		&\int_{\Omega\times\R^3}f_+(t)\,dvdx = \int_{\Omega\times\R^3}f_-(t)\,dvdx =0,
		\\
		&\int_{\Omega\times\R^3}(f_+(t)+f_-(t))|v|^2\,dvdx +\int_{\Omega}|\nabla_x \phi(t, x)|^2\,dx= 0.
	\end{aligned}
\end{equation}

\subsection{Notations}  Let us first introduce the function spaces and notations.
We let the multi-indices $\alpha$ and $\beta$ be $\alpha=[\alpha_1,\alpha_2,\alpha_3]$, $\beta=[\beta_1,\beta_2,\beta_3]$ and define $\partial^\alpha_{\beta}:=\partial^{\alpha_1}_{x_1}\partial^{\alpha_2}_{x_2}\partial^{\alpha_3}_{x_3}\partial^{\beta_1}_{v_1}\partial^{\beta_2}_{v_2}\partial^{\beta_3}_{v_3}$.

%
If each component of $\al'$ is not greater than that of the $\al$'s, we denote by $\al'\le\al$. 
$\al'<\al$ means $\al'\le\al$, and $|\al'|<|\al|$.
%
%
We write $a\lesssim b \ (a\gtrsim b)$ to indicate that there is a uniform constant $C$, which may be different on different lines, such that $a\le Cb \ (a\ge Cb)$. We use the notation $a\sim b$ if $a\lesssim b$ and $b\lesssim a$. We denote $a\gg b$ if $a,b$ are two constants such that $a>b$ and $a$ is sufficiently large. 

We denote $C_{a_1,a_2,\cdots, a_n}$ by a constant depending on parameters $a_1,a_2,\cdots, a_n$. Moreover, we use the parameter $\ve$ to represent different positive numbers much less than 1 and determined in different cases.
%
 We use $(f, g)=(f, g)_{L^2_v}$ to denote the inner product of $f, g$ over velocity variable for short and, $(f, g)_{L^2_{x, v}}$ to represent the inner product over both spatial and velocity variables. Also, we write 
\begin{align*}
(f, g)_{L^2_k}=(\langle v \rangle^{2k}f, g  )_{L^2_v},   \quad 	\|f(\theta)\|_{L^1_\theta} = \int_{\S^2}f(\theta) d\sigma = 2\pi\int_0^{\pi}   f(\theta)  \sin\theta  d\theta,
\end{align*}
and for any $p\in[1,\infty]$,  $\|f(t)\|_{L^p_T} = \|f(t)\|_{L^p([0,T])}$. Define the \textbf{Japanese brackets} by $\langle v\rangle :=(1+|v|^2)^{1/2}$.

%
%
 For linear operator $\L$, we write $S_\L$ as the semigroup generated by $\L$. For real numbers $m, l$, we define the weighted Sobolev norm $\|\cdot\|_{H_l^m}$ by $|f|_{ H^m_l}=|\langle v\rangle^l\langle D_v \rangle^m  f(v)|_{L^2_v}$, and $L^2_l=H^0_l$ for $m=0$.
here $a(D)$ is a pseudo-differential operator with the symbol $a(\xi)$, and is defined by 
\begin{align*}
(a(D)f)(v):= \frac{1} {(2\pi)^3}  \int_{\R^3}\int_{\R^3}e^{i(v-u) \xi }  a(\xi) f(u) du d\xi.
\end{align*}
%
The mixed norm $\|\cdot\|_{H^n_xH^m_l}$ is defined as
$\|f\|_{H^n_xH^m_l}:=\Big(\int_{\Omega}\|\langle D_x\rangle ^n f(x,\cdot)  \|^2_{H^m_l}dx\Big)^{1/2}$, and $H^0_xH^m_l=L^2_xH^m_l$. 
%
 The entropy $L\log L$ space is defined as $L\log L:=\Big\{f(v):\|f\|_{L\log L}=\int_{\R^3}|f|\log(1+|f|)dv\Big\}$.\\
For any $ k \in  \R$, $\gamma \in (-3, 1]$, we define
\[
\Vert f \Vert_{L^2_{k+\gamma/2, *}}^2 :=  \int_{\R^3} \int_{\R^3} \mu (v_*) |v-v_*|^\gamma |f(v)|^2 \langle v \rangle^{2k} dv dv_* \sim\Vert f \Vert_{L^2_{k+\gamma/2}}^2.
\]
For Boltzmann case, we denote the dissipation norm $L^2_{D,k}$ by $$\|f\|_{L^2_{D,k}} = \|f\|_{H^s_{k+\gamma/2}},\,\, \|f\|_{L^2_{D}}=\|f\|_{L^2_{D,0}}.$$
For Landau case, we denote the anisotropic norm $L^2_{D}(m)$ by  
$$\Vert f \Vert_{L^2_{D}(m) }  : = \Vert f \Vert_{L^2(m\langle v \rangle^{\gamma/2} )} +  \Vert \widetilde{\nabla}_v (mf )\Vert_{L^2(\langle v \rangle^{\gamma/2} )}$$
and for brevity, we let $L^2_{D, k} := L^2_{D}(\langle v \rangle^k)$ and $L^2_{D} := L^2_{D, 0}$.
Here $\widetilde{\nabla}_v$ is the anisotropic gradient given by 
\begin{equation*}
	\widetilde{\nabla}_v f :=  P_v \nabla_v f + \langle v \rangle (I-P_v) \nabla_v f, \quad P_v \xi := \left( \xi \cdot \frac {v} {|v|}          \right)\frac {v} {|v|} , \quad \forall~ \xi \in \mathbb R^3.
\end{equation*}

With multi-indices $(\al,\beta)$, we come to define our weight function with some constant $k$. For $-3\le\gamma\le 1$, we choose the weight function $w(\alpha, \beta) $ as
\begin{equation}
\label{functionw}
w(\alpha, \beta)  =  \left\{
\begin{aligned}
&\langle v \rangle^{k-  p|\alpha|  - q|\beta| + r}, 
\quad q = 6s -3(\gamma{{-1}}), \quad p =q + \gamma{{-1}}, r = 2q 
\quad\mbox{for Boltzmann case},
\\
&\langle v \rangle^{k - p|\alpha|  - q |\beta|+r} , \quad q =  3-(\gamma{{-1}}), \quad p =  3, r = 2q , \quad\mbox{for Landau case}.
\end{aligned}
\right. 
\end{equation}
Note that $w(\al,\beta)\ge 1$ for any $|\al|+|\beta|\le 2$ and any $k\ge 0$.
For brevity, we write  $w(|\alpha|, |\beta|)=w(\alpha, \beta) $ throughout the paper. 
Noticing $w^2(\alpha, \beta)$ is still of the form $\langle v \rangle^k$ for some $k>0$, we have 
\begin{align}\label{Cab1}
	\nabla_v (w^2(\alpha, \beta))  = A_{\al,\beta} \frac v {\langle v \rangle^2} w^2(\alpha, \beta),
\end{align}
here $A_{\al,\beta}$ depends on $\al,\beta, k,\gamma, s $. In this work, we will apply a useful space-velocity weight 
\begin{align}\label{eA}
	e^{\frac{\pm A_{\al,\beta}\phi}{\<v\>^2}}
\end{align} to eliminate the dissipation loss term.
Next, we define some useful norms in our analysis. For this, we denote constants by
\begin{align}\label{Cal}
	C_{|\al|,|\beta| } \gg C_{|\al|, |\beta_1|}, \  \hbox{ if }  |\beta_1| < |\beta|, \quad C_{|\al| +1,|\beta|-1 } \gg C_{|\al|, |\beta|},
\end{align}
 for any multi-indices $\al,\beta$.
Then for both the Boltzmann case and Landau case, we denote energy norms as
\begin{equation}\label{E}
	\Vert f \Vert_{X_k}^2 = \sum_{|\alpha| + |\beta| \le  2}C_{|\alpha|, |\beta|}\Vert e^{\frac{\pm A_{\al,\beta}\phi}{\<v\>^2}} w (\alpha, \beta)\partial^{\alpha}_{\beta} f  \Vert_{L^2_xL^2_v}^2,
\end{equation}
\begin{equation}\label{D}
	\Vert f \Vert_{Y_k}^2  = \sum_{|\alpha| + |\beta| \le  2}C_{|\alpha|, |\beta|} \Vert w (\alpha, \beta) \partial^{\alpha}_{\beta} f \Vert_{L^2_xL^2_{D}}^2.  
\end{equation}
We further define the ``instant energy functional" $\E_k(t)$ and ``dissipation energy functional" $\D_k(t)$ as
\begin{equation}
	\label{DefE}
	\E_k(t)\approx \|f\|^2_{X_k}+ \|\na_x\phi\|^2_{H^2_x}, 
\end{equation}
\begin{equation}
	\label{DefD}
	\D_k(t) := \|f\|^2_{Y_k} + \|\na_x\phi\|^2_{H^2_x},  
\end{equation}
where the explicit definition of $\E_k(t)$ is given in \eqref{realE}. \\

We fix the weight index $k_0$ by assuming $k_0\ge 14$ for the Boltzmann case, and $k_0\ge 7$ for the Landau case.

\medskip

\subsection{Main results} We may now state our main results.

\begin{thm}[Global existence, uniqueness and large-time decay]\label{globaldecay}
Consider the  Cauchy problem \eqref{vplocal1} for Vlasov-Poisson-Boltzmann/Landau system. Suppose $\gamma \in (-3, 1]$, $s \in [\frac 1 2, 1)$, $\gamma+2s>-1$ for the Boltzmann case and $\gamma \in [-3, 1]$ for the Landau case. 
Let $l=0$ for the hard potential case and $l>\frac{|\gamma|}{2}$ for the soft potential case. 
Then 
   there exist constants $\lam>0$, $M>0$ (small) such that,  for any 
   \begin{align*}
   k\ge k_0+4+l\  \text{ for  Landau case and }\ \ k\ge k_0+4+l \  \text{ for Boltzmann case.} 
   \end{align*}
  Remind $\E_k(t)$ is given in \eqref{DefE}. If the initial data $f_0$ satisfies $F_0(x,v)=\mu+f_0(x,v)\ge 0$, conservation laws \eqref{conse2.torus} for the case of torus (or \eqref{conservation_bounded} for the case of union of cubes) and 
   \begin{align} \label{small}
   	\E_{k_0+l}(0)\leq M \text{ and } \E_k(0)< + \infty,
   \end{align}
then there exists a unique solution $f(t,x,v)$ to \eqref{vplocal1} satisfying $F(t,x,v)=\mu+f(t,x,v)\ge 0$ such that $\sup_{0\le t\le T}\E_{k}(t)\le 2\E_k(0)$ for any $T>0$. Moreover, we have the following large-time asymptotic behavior: 
\begin{align*}
	\E_{k}(t)\le  e^{-\lam t}\E_k(0)\quad\text{for}\quad\gamma\ge0,  \quad\E_{k-l}(t) \lesssim (1+t)^{-\frac{2l}{|\gamma|}}\E_k(0)\quad\text{for}\quad\gamma<0.
\end{align*}

%

\end{thm}


\begin{rmk}              
  (1) It is crucial to emphasize that for global solutions with polynomial weight, the requirement for smallness of initial data is only in relation to a specific norm. In this context, we solely need $X_{k_0 + l }$ to be sufficiently small, rather than necessitating the general $X_k$ to be small.

(2) Notably, the cases involving the torus and the union of cubes exhibit similarities. In the case of the union of cubes, as defined in \eqref{Omega}, we exercise meticulous care when performing integration by parts. Conversely, integration by parts in the torus case is a straightforward process.

(3) In the scenario of a soft potential with $\gamma < 0$ and considering the non-cutoff Boltzmann equation, the work presented in \cite{CHJ} demonstrates that the optimal decay rate for the non-cutoff Boltzmann equation cannot exceed $k/|\gamma|$. Consequently, when the value of $k$ is sufficiently large, the decay rate articulated in Theorem \ref{globaldecay} can be regarded as approximately optimal.

(4) Our proof methodology can also be effectively employed to establish results for the high regularity space $H^{N}_{x, v} (\langle v \rangle^k)$.
\end{rmk}

\subsection{Strategies and ideas of the proof}  In this subsection, we provide a brief overview of relevant literature and outline the key strategies employed in proving our results. 

We begin by revisiting established findings in the context of the Landau and Boltzmann equations, with a particular focus on the aspects central to this paper—namely, the global existence and long-term behavior of solutions to spatially inhomogeneous equations within the perturbation framework. To shed light on the global solutions of the renormalized equation, particularly in the presence of substantial initial data, we draw attention to classic works such as \cite{DL, DL2, L, V2, V3, DV, AV}.

For insights into the smoothing effects in the Boltzmann equation without cut-off, we refer to \cite{CH, CH2}. Additionally, when considering the stability of vacuum states, we consult references \cite{L2, G5, C2}, which cover scenarios involving the Landau equation, the cutoff Boltzmann equation, and the non-cutoff Boltzmann equation, with a focus on moderate soft potentials.

These references provide the foundation upon which we build our exploration of spatially inhomogeneous equations and perturbation techniques to address the questions at hand.

\smallskip
In the context of the non-cutoff Boltzmann equation, prior research efforts like \cite{IMS, IMS2, IS, IS2, IS3, S} have achieved global regularity and elucidated long-term behavior under remarkably broad assumptions. They have demonstrated that global regularity and long-term behavior can be established by merely assuming uniform bounds in both time ($t$) and space ($x$), characterized by:
\begin{align*}
0<m_0 \le M(t, x) \le M_0, \quad E(t, x) \le E_0, \quad  H(t, x) \le H_0, 
\end{align*}
where $m_0$, $M_0$, $E_0$, and $H_0$ are positive constants, and $M(t, x)$, $E(t, x)$, and $H(t, x)$ are defined as:
\begin{align*}
M(t, x) =\int_{\R^3} f(t, x, v) dv, \, E(t, x) = \int_{\R^3} f(t, x, v) |v|^2 dv, \,  H(t, x) = \int_{\R^3} f(t, x, v) \ln f(t, x ,v) dv.
\end{align*}
Furthermore, in the case of the Landau equation, local Hölder estimates are established in \cite{GIMV}, while higher regularity properties of solutions are explored in \cite{HS} through the application of kinetic variants of Schauder estimates.

\smallskip
Subsequently, we turn our attention to the outcomes achieved within the perturbation framework. In the vicinity of the Maxwellian distribution, we find substantial progress in demonstrating the global existence and long-term behavior of solutions to spatially inhomogeneous equations. These achievements are well-documented, with \cite{G2, G3, SG, SG2} presenting results for the cutoff Boltzmann equation and \cite{G} addressing the Landau equation within this context.

For the non-cutoff Boltzmann equation, substantial contributions can be found in \cite{GS, GS2, AMUXY, AMUXY2, AMUXY3, AMUXY4}. A more recent perspective on this topic can be found in \cite{DLSS}, where recent developments are explored. Additionally, for insights into the theory of existence in bounded domains, references \cite{Guo2009, Guo2016, Guo2020} are valuable resources.

Furthermore, the non-cutoff case within the union of cubes is examined in \cite{Deng2021}. Importantly, all these works are founded on the following decomposition:
\begin{align*}
\partial_t f  + v \cdot \nabla_x f =L_{\mu} f + \Gamma(f, f), \quad L_\mu f =\frac 1 {\sqrt{\mu}}  Q(\sqrt{\mu} f, \mu  ) + \frac 1 {\sqrt{\mu}}  Q(\mu, \sqrt{\mu}  f  ) , 
\end{align*}
where
\begin{align*}
\quad \Gamma(g, f) = \frac 1 2 \frac 1 {\sqrt{\mu}} Q(\sqrt{\mu} g, \sqrt{\mu} f) + \frac 1 2\frac 1 {\sqrt{\mu}}  Q(\sqrt{\mu} f, \sqrt{\mu} g), 
\end{align*}
signifying that the solution is constructed within a $\mu^{-1/2}$ weighted space.

\smallskip
In the realm of inhomogeneous Boltzmann/Landau equations with polynomial weight perturbations near the Maxwellian distribution, Gualdani-Mischler-Mouhot (GMM) laid the foundation by first proving global existence and delineating the large-time behavior of solutions featuring polynomial velocity weights. Their groundbreaking work began with the cutoff Boltzmann equation involving a hard potential, as reported in \cite{GMM}. This method was subsequently extended to encompass the Landau equation, with significant contributions from \cite{CTW, CM}. For the non-cutoff Boltzmann equation, the achievements in the hard potential scenario were documented in \cite{HTT, AMSY}, while the soft potential case was addressed in \cite{CHJ}.

\smallskip
Moreover, it is noteworthy to acknowledge prior research on the Vlasov-Poisson/Maxwell-Boltzmann/Landau systems in proximity to Maxwellian distributions. For the cutoff hard sphere case, the Vlasov-Poisson/Maxwell-Boltzmann system has been proven in \cite{G6, G7, S2}, with \cite{DS} further exploring the optimal convergence rate. Various other cases under the cutoff assumption have been studied and can be found in \cite{DYZ, DYZ2, DLYZ}. For the Landau equation, pioneering results were established in \cite{G8} for the torus case and in \cite{SZ, W} for the entire space scenario, with additional insights provided by \cite{D}.

In the non-cutoff setting, research endeavors have been directed towards the Vlasov-Poisson-Boltzmann system \cite{DLiu, XXZ}, the Vlasov-Maxwell-Boltzmann system \cite{DLYZ2}, and investigations into the regularizing effect \cite{Deng2021b}. The implications of bounded domains are discussed in \cite{Cao2019, Dong2020, Deng2021e}. Importantly, all these endeavors are rooted in the following decomposition:
\begin{align*}
\partial_t f_\pm  + v \cdot \nabla_x f_\pm  \mp \nabla_x \phi \cdot \nabla_v f_\pm \pm \nabla_x \phi \cdot v \sqrt{\mu} \pm \frac 1 2 \nabla_x \phi \cdot v f =L_{\mu,\pm} f + \Gamma_\pm(f, f), 
\end{align*}
where
\begin{align}\label{L mu}
L_{\mu,_\pm} f = \frac 1 {\sqrt{\mu}}  Q(\sqrt{\mu} (f_\pm+f_\mp), \mu  ) + 2\frac 1 {\sqrt{\mu}} Q(\mu, \sqrt{\mu}f_\pm  ),  \,\, 
\end{align}
and
\[
\Gamma_\pm(f, f) = \frac 1 {\sqrt{\mu}}   Q(\sqrt{\mu} f_\pm, \sqrt{\mu} f_\pm) +\frac 1 {\sqrt{\mu}}  Q(\sqrt{\mu} f_\mp, \sqrt{\mu}f_\pm),
\]
demonstrating that these studies are conducted within a $\mu^{-1/2}$ weighted space. To the best of our knowledge, our research represents the inaugural effort to address these questions within the framework of polynomial-weighted spaces.

  \smallskip
The foundation of our results lies in the utilization of the semigroup method. This method, originally introduced by Gualdani-Mischler-Mouhot in \cite{GMM}, forms the basis of our work and has been further developed and expanded upon, particularly with respect to the macroscopic component $\P f$. The central concept of this method can be succinctly summarized as follows: Consider the case where $\gamma =0$ and the function space $H^2_xL^2_k$. Drawing from \cite{MS}, we arrive at the following inequality:
\begin{align*}
\sum_{\pm}(\L_\pm f, f_\pm)_{H^2_xL^2 (\mu^{-1/2})} \le -\lambda \Vert f \Vert_{H^2_x L^2 (\mu^{-1/2})}, \quad\text{ if } \P f =0, 
\end{align*}
where $\lambda>0$ serves as a crucial constant. Together with the macroscopic estimates from \cite{Deng2021e, GS} for $\P f$, we can deduce that  $\Vert S_{\L}(t) f_0 \Vert_{H^2_x L^2 (\mu^{-1/2})} \le e^{-\lambda t} \Vert f_0 \Vert_{H^2_x L^2 (\mu^{-1/2})}^2$. Next, for some $M, R>0$, define 
\begin{align*}
	&A_\pm = -v\cdot\na_x+L_\pm  -M\chi_R,\quad
	K_1  = M\chi_R, \quad B_\pm = -v\cdot\na_x+L_\pm,\quad
	K_2 = \pm\mu v\cdot\na_x\phi,
\end{align*}
where $\chi \in D(\R)$ is the truncation function  satisfying $1_{[-1,1]} \le \chi \le 1_{[-2,2]}$ and we denote $\chi_R(\cdot) := \chi(\cdot/R)$ for $R > 0$.
By the results in \cite{CHJ} we have
\begin{align*}
\sum_{\pm}(L_\pm f, f_\pm)_{H^2_x L^2_k} \le -C \Vert f \Vert_{H^2_x H^s_k}^2 +  C_k \Vert f \Vert_{H^2_x L^2_v}^2,
\end{align*}
Take $M, R>0$ large, we have $(Af, f)_{H^2_x L^2_k} \le -C \Vert f \Vert_{H^2_x L^2_k}^2$, which implies 
$\Vert S_A(t) f \Vert_{H^2_x L^2_k} \le e^{-\lambda t} \Vert f_0 \Vert_{H^2_x L^2_k }^2$. By Duhamel's formula $S_B =S_A +S_B * K_1 S_A$ and $S_\L =S_B +S_\L * K_2 S_B$, roughly speaking, we have
\begin{multline*}
\Vert S_{B}(t) \Vert_{H^2_x L^2_{k} \to H^2_x L^2_k }   \le \Vert S_A(t) \Vert_{H^2_x L^2_k \to H^2_x L^2_k} + \int_0^t \Vert S_B(s)  \Vert_{H^2_x L^2 (\mu^{-1/2}) \to H^2_x L^2 (\mu^{-1/2})} \\\times\Vert K_1 \Vert_{H^2_x L^2_k \to H^2_x L^2(\mu^{-1/2} ) }\Vert S_A(t-s)\Vert_{H^2_x L^2_k \to H^2_x L^2_k}ds \le C e^{-\lambda t}. 
\end{multline*}
and hence
\begin{multline*}
\Vert S_{\L}(t) \Vert_{H^2_x L^2_{k} \to H^2_x  L^2_k }   \le \Vert S_B(t) \Vert_{H^2_x L^2_k \to H^2_x L^2_k} + \int_0^t \Vert S_L(s)  \Vert_{H^2_x L^2 (\mu^{-1/2}) \to H^2_x L^2 (\mu^{-1/2})} \\\times\Vert K_2 \Vert_{H^2_x L^2_k \to H^2_x L^2(\mu^{-1/2} ) }\Vert S_B(t-s)\Vert_{H^2_x L^2_k \to H^2_x L^2_k}ds \le C e^{-\lambda t}. 
\end{multline*}
Thus, the rate of convergence for the linear operator $\L$ is established. To estimate the nonlinear part, we need to define a scalar product by
\[
((f, g))_k := (f, g)_{L^2_k} + \eta\int_0^{+\infty} (S_\L(\tau )f, S_\L(\tau) g )_{L^2_v} d\tau.
\] 
Due to the fact that
\begin{align*}
\int_0^{+\infty} (S_{\L}(\tau) \L f, S_{\L}(\tau) f ) d\tau = \int_0^{+\infty} \frac d {d\tau} \Vert S_{\L}(\tau) f \Vert_{L^2}^2 d\tau 
= - \Vert f \Vert_{L^2}^2, 
\end{align*}
so after choosing  proper $\eta$, we deduce that
\[
((\L f, f))_k = ( \L f, f)_{L^2_k} + \eta\int_0^\infty (S_\L (\tau ) L f, S_\L(\tau) f)_{L^2_v} d\tau \sim \|f\|_{H^s_{k}}^2.
\]
Therefore, it is important to note that the linear operator $\L$ can exhibit non-negativity properties within an appropriate function space. The estimate of $\L$ within this specific function space empowers us to establish global well-posedness by combining the nonlinear estimates.


The semigroup method serves as a valuable tool in addressing lower order terms. However, it introduces new challenging terms that require estimation, and this approach needs an assumption of $s \ge 1/2$. Precisely, when applying the semigroup method to analyze the Vlasov-Poisson term $\nabla_x \phi(x) \cdot \nabla_v f$, we encounter a departure from the usual case, where integration by parts yields the expected result
\begin{align*}
(\nabla_x \phi(x) \cdot \nabla_v f, f )=0.
\end{align*}
For the semigroup term, we have
\begin{align*}
	\int_0^\infty (S_{\L}(\tau)\nabla_x \phi(x) \cdot \nabla_v f , S_{\L}(\tau) f ) d\tau.
\end{align*} 
Since the operators $S_{\L}(t) $ and $\nabla_v $ are non-commutative,  integration by parts with respect to $\nabla_v$ is not allowed and we must use upper bound to constrain this term. By  \cite{CHJ}  and Duhamel's formula, we have
\begin{align*}
\Vert S_{B}(t) f \Vert_{H^2_xL^2_v} \le t^{-1/2}e^{-\lambda t}\Vert f \Vert_{H^2_xH^{-s}_{12}},\quad\text{and}\quad \Vert S_{\L}(t) f \Vert_{H^2_x L^2_v} \le t^{-1/2}e^{-\lambda t}\Vert f \Vert_{H^2_xH^{-s}_{12}}.
\end{align*}
Due to the extra $\nabla_v $ term, we need to use
\begin{align*}
\Vert S_{\L}(t) \nabla_x \phi(x) \cdot \nabla_v f \Vert_{H^2_xL^2_v} &\lesssim t^{-1/2}e^{-\lambda t} \Vert\nabla_x \phi(x) \Vert_{H^2_x} \Vert \nabla_v f \Vert_{H^2_x H^{-s}_{12}} \lesssim t^{-1/2}e^{-\lambda t} \Vert\nabla_x \phi(x) \Vert_{H^2_x}  \Vert f \Vert_{H^2_x H^s_{12}},
\end{align*}
which requires $s \ge 1/2$.


\smallskip
\subsubsection{Comments on the weight \eqref{eA}. }
The polynomial perturbation can be regarded as a change of variable $F=\mu+\langle  v \rangle^{-k} f_\pm$ for \eqref{vpb1} for some constant $k >0$. Then we have  
\begin{align*}
\partial_t f_\pm  + v \cdot \nabla_x f_\pm  \mp \nabla_x \phi \cdot \nabla_v f_\pm \pm  \nabla_x \phi \cdot v \mu \langle v \rangle^{2k}  \pm k  \nabla_x \phi \cdot \frac  v {\langle v \rangle^2 }f =L_{k,\pm} f + \Gamma_{k,\pm} (f, f), 
\end{align*}
with
\begin{align*}
L_{k,\pm} f = \langle v \rangle^k Q(\langle v \rangle^{-k} (f_\pm+f_\mp), \mu ) + 2\langle v \rangle^{k}  Q(\mu, \langle v \rangle^{-k} f_\pm), \quad\Gamma_{k,\pm}(f, f) = \langle v \rangle^{k}  Q(\langle v \rangle^{-k}(f_\pm+f_\mp), \langle v \rangle^{-k} f_\pm). 
\end{align*}
In order to deal with the extra term $\pm k  \nabla_x \phi \cdot \frac  v {\langle v \rangle^2 }f$, we follow the idea of \cite{G8} and introduce the corresponding weight function 
\begin{equation}
\label{weight function exponential k}
	\exp\Big\{\frac{\pm k\phi}{\<v\>^2}\Big\},
\end{equation}
which satisfies 
\[
v \cdot \nabla_x (  e^{\frac{\pm k\phi}{\<v\>^2}} f   ) =e^{\frac{\pm k\phi}{\<v\>^2}} (v \cdot \nabla_x  f  \pm k  \nabla_x \phi \cdot\frac{v}{\<v\>^2} f).
\]
Such choice of weight function allows us to absorb the extra term by taking integration by parts as  
\begin{align*}
	\Big(\pm k  \nabla_x \phi \cdot \frac  v {\langle v \rangle^2 }f, e^{\frac{\pm k\phi}{\<v\>^2}}f\Big)_{L^2_x L^2_v}
	+ \Big(v \cdot \nabla_x f_\pm, e^{\frac{\pm k\phi}{\<v\>^2}}f\Big)_{L^2_x L^2_v} = 0.
\end{align*}
In contrast to the approach in \cite{G8}, where the weight function was defined solely as $e^{-\phi}$, our weight function potential \eqref{weight function exponential k} exhibits dependence on the velocity variable as well. Fortunately, the newly introduced error contributions take the form of $ (\nabla_x \phi \cdot v \mu, (e^{\frac{\pm k\phi}{\<v\>^2}}-1) \langle v \rangle^{2k} f )$. These contributions can be effectively managed when $\phi$ is sufficiently  small, and the velocity derivative of $\phi$ remains under control, particularly if $\phi$ is bounded.

\subsubsection{Difference between the polynomial  case and the exponential weight case.} 
The primary distinction lies in the symmetry properties of  $L_\mu f$ and $\Gamma(f, f)$  as defined in \eqref{L mu}. In the exponential case, these operators exhibit symmetry, whereas in the polynomial case, they do not. This lack of symmetry gives rise to two significant challenges.
The first issue arises due to the introduction of an additional $\P f$ term in the linearized equation. To elaborate further, if we denote
\begin{align*}
\P_\mu f = a_\pm(t, x) \sqrt{\mu} +  b(t, x) \cdot v \sqrt{\mu}   +c(t, x) |v|^2  \sqrt{\mu}.
\end{align*}
Then by $L_\mu\P_{\mu}f = 0$ and $L_\mu$ is symmetric, we have
\begin{align*}
(L_\mu f , f ) = (L_{\mu} (\P_{\mu} f +\{\I-\P_{\mu}\} f ) , \P_{\mu} f + \{\I-\P_{\mu}\} f) = (L_{\mu}( \{\I-\P_{\mu}\}f) , \{\I-\P_{\mu}\} f). 
\end{align*}
Only the term with $\{\I-\P_{\mu}\} f$ remains. While for the polynomial weight case, since $L$ is not symmetric, one merely has 
\begin{align*}
(L f, f)   = (L \{\I-\P\} f, f) =  (L \{\I-\P\} f, \P f) +  (L \{\I-\P\} f, \{\I-\P\} f).
\end{align*}
An extra term $(L \{\I-\P\} f, \P f)$ occurs. Fortunately, we can bound it by
\begin{align*}
(L \{\I-\P\} f, \P f)  \le \epsilon \Vert \P f \Vert_{L^2_{x, v}}^2 +  \epsilon \Vert \{\I-\P\} f \Vert_{L^2_xL^2_k}^2 +     C_\epsilon  \Vert \{\I-\P\} f \Vert_{L^2_{x, v}} .
\end{align*}
Estimating the $\P f$ term can be accomplished through macroscopic estimates, while the last term can be assessed using the semigroup method. 

However, the second challenge arises when we attempt to extend our results to $x \in \R^3$ using existing techniques. In previous works such as \cite{S2, SZ}, applied to the entire space $x \in \R^3$, the following fact is relied upon: $(\Gamma(g, f), \P_\mu h) = 0$ due to the symmetry of $\Gamma$. Consequently, in \cite{S2, SZ}, the dissipation rate functional $D(t)$ is defined as follows:
\[
D(t) =   \sum_{ |\alpha| \le K} \Vert \nabla_x \partial^{\alpha} \phi \Vert_{L^2_x}^2  +  \sum_{1 \le |\alpha| \le K} \Vert  \P_\mu \partial^\alpha f\Vert_{L^2_{x, v}}^2 +  \sum_{0 \le |\alpha| +|\beta| \le K} \Vert w(\alpha, \beta)  \{\I-\P_\mu \} \partial^\alpha_\beta f  \partial^\alpha f\Vert_{Y_k}^2, 
\]
In the exponential case, the estimation of $\P_{\mu} f$ conveniently avoids complications. However, in the polynomial case, a departure from this simplicity occurs. This is because, in the polynomial case, the relationship $(Q(f, g), \mu \phi(v)) \neq 0$ holds for functions $\phi(v)$ such as $\phi(v) = 1, v, |v|^2$. Consequently, the methods employed in \cite{S2, SZ} are no longer applicable in the polynomial scenario. To address this, we necessitate the assumption that $\Omega$ represents a bounded domain and  Poincar\'e inequality to estimate $\P f$ via $\nabla_x (\P f)$. However, extending our results to the entire space case remains an open challenge.

\subsection{Organization of the paper}
In Section \ref{sec2}, we first recall some basic properties of the non-cutoff Boltzmann/Landau collision operator and then compute the basic estimates for the VPB/VPL system. We obtain corresponding global results and decay of solutions in Section \ref{sec4}.

\section{Basic estimates }
\label{sec2}

In this section, we recall some critical results on the non-cutoff Boltzmann equation and Landau equation and prove several weighted estimates for the VPB/VPL equation. 

\subsection{Preliminaries  on the Boltzmann/Landau equation} 

Before introducing the upper and lower bound for the non-cutoff  Boltzmann and Landau collision operator, we first state the following lemmas. 

\begin{lem}(\cite{C}, Lemma 2.5) \label{L53}
	  Let $\gamma \in (-2, 0)$. Then for any function $f$, we have
\begin{align}\label{21}
\sup_{v \in \R^3}\int_{\R^3} |v-v_*|^{\gamma} |f(v_*)|  dv_* \lesssim \Vert f \Vert_{L^1}^{1 + \frac { 2 \gamma} 3} \Vert f \Vert_{L^2}^{-\frac  { 2 \gamma} 3}\quad
\mbox{when}\quad \gamma \in (- \frac 3 2, 0),
\end{align}
and 
\begin{align}\label{22}
\sup_{v \in \R^3}\int_{\R^3} |v-v_*|^{\gamma} |f(v_*)| dv_* \lesssim \Vert f \Vert_{L^1}^{1 + \frac { \gamma} 2} \Vert f \Vert_{L^3}^{-\frac  {\gamma} 2} \quad
\mbox{when}\quad \gamma \in (-  2, 0).
\end{align}
\end{lem}

\begin{lem}\label{L54}
For any $\gamma \in [-3, 1]$ and any functions $f, g, h$, we have
\begin{align}\label{23}
\int_{\R^3}\int_{\R^3} |v-v_*|^{\gamma+2} f (v_*) g(v) h(v)dv_* dv\lesssim  \Vert f \Vert_{L^2_5}  \Vert g \Vert_{L^2}  \Vert h \Vert_{L^2_3},
\end{align}
\begin{align}\label{24}
\int_{\R^3} \int_{\R^3} |v-v_*|^{\gamma+1} f (v_*) g(v) h(v)dv_* dv\lesssim  \Vert f \Vert_{L^2_5}  \Vert g \Vert_{L^2}  \Vert h \Vert_{H^1_3},
\end{align}
and 
\begin{align}\label{25}
\int_{\R^3} \int_{\R^3} |v-v_*|^{\gamma+1} f (v_*) g(v) h(v)dv_* dv   \lesssim  \Vert f \Vert_{L^2_5}  \Vert g \Vert_{H^1_{\gamma/2 +1}}  \Vert h \Vert_{L^2_{\gamma/2}}.
\end{align}
If $\gamma \in (-3, 1]$, then we have
\begin{align}\label{26}
\int_{\R^3} \int_{\R^3} |v-v_*|^{\gamma} f (v_*) g^2 (v)dv_* dv   \lesssim  \Vert f \Vert_{L^2_5}  \Vert g \Vert_{H^1_{\gamma/2 }}^2.
\end{align}
\end{lem}

\begin{proof}
If $-2\le\gamma\le 1$, we deduce that
\begin{align*}
\int_{\R^3} |v-v_*|^{\gamma+2} f (v_*)  g (v) h(v)dv_* dv\lesssim  \Vert f \Vert_{L^1_3}  \Vert g \Vert_{L^2}  \Vert h \Vert_{L^2_3} \lesssim  \Vert f \Vert_{L^2_5}  \Vert g \Vert_{L^2}  \Vert h \Vert_{L^2_3}.
\end{align*}
If $-3\le\gamma<-2$, by \eqref{21}, we have $\sup_{v \in \R^3} \int_{\R^3} |v-v_*|^{\gamma+2} f (v_*) dv_*\lesssim  \Vert f \Vert_{L^2_2}$, which implies that 
\begin{align*}
\int_{\R^3} |v-v_*|^{\gamma+2} f (v_*) g (v) h(v)dv_* dv\lesssim   \Vert f \Vert_{L^2_2}   \int_{\R^3} |g ( v) h ( v)| dv  \lesssim   \Vert f \Vert_{L^2_2}  \Vert g \Vert_{L^2}  \Vert h \Vert_{L^2}. 
\end{align*}
This proves \eqref{23}.
 For \eqref{24},
 if $\gamma\ge -2$, then it reduce to the \eqref{23}. So we only need to consider the case $\gamma \in [-3, -2)$, in which case we have $- 3< \gamma + \frac 1 2 <0$ and $-2<\gamma + \frac 3 2  < 0$. Then we have 
 \begin{align*}
 	\int_{\R^3} \int_{\R^3} |v - v_*|^{\gamma +\frac 1 2} \frac {1} { \langle v_* \rangle^4}  |g(v)|^2  dv dv_*  \lesssim \Vert g \Vert_{L^2}^2.
 \end{align*}
Applying \eqref{22} to $h$, we have 
\begin{multline*}
	\int_{\R^3} \int_{\R^3} |v - v_*|^{\gamma+\frac 3 2}  \langle v_* \rangle^4     |f(v_*)|^2 | h(v)|^2 dv dv_*  
	\lesssim \Vert f \Vert_{L^2_2}^2  \||h|^2\|_{L^1}^{1-\frac{\gamma+3/2}{2}}\Vert |h|^2 \Vert^{\frac{\gamma+3/2}{2}}_{L^3 }\lesssim \Vert f\Vert_{L^2_2}^2  \Vert h \Vert_{H^1 }^2. 
\end{multline*}
Here we apply Sobolev embedding $\|\cdot\|_{L^6}\lesssim \|\cdot\|_{H^1}$. 
Thus, 
\begin{align*}
 &\quad\,   \int_{\R^3} \int_{\R^3} |v-v_*|^{\gamma+1} | f(v_*) | |g(v)| |h(v) |   dv dv_* 
\\
&\le  \Big(\int_{\R^3} \int_{\R^3} |v - v_*|^{\gamma+\frac 3 2}  \langle v_* \rangle^4     |f(v_*)|^2 | h(v)|^2 dv dv_*   \Big)^{\frac{1}{2}}\Big(\int_{\R^3} \int_{\R^3}   \frac {|v - v_*|^{\gamma + \frac 1 2}} { \langle v_* \rangle^4}  |g(v)|^2  dv dv_* \Big)^{\frac{1}{2}}\\
&\le \Vert f \Vert_{L^2_2}  \Vert g \Vert_{L^2}  \Vert h \Vert_{H^1}. 
\end{align*}
The proof of \eqref{24} is then finished. Now we come to \eqref{25}. If $-1\le\gamma \le 1$, we deduce that 
\begin{align*}
\int_{\R^3}\int_{\R^3} |v-v_*|^{\gamma + 1} f (v_*)  g (v) h(v)dv_* dv&\lesssim  \Vert f \Vert_{L^1_{\gamma+1}} \int_{\R^3} \langle v \rangle^ {\gamma + 1} | g (v) h(v)| dv\lesssim  \Vert f \Vert_{L^2_{4}}   \Vert g \Vert_{L^2_{\gamma/2+1}}  \Vert h \Vert_{L^2_{\gamma/2}}.
\end{align*}
If $-3\le\gamma<-1$, we have
$
\langle v \rangle^{ - (\gamma + 1)} \lesssim \langle v_* \rangle^{ - (\gamma + 1)} \langle v - v_* \rangle^{ - (\gamma + 1)} ,
$
which implies
\begin{equation*}\label{2t}
\begin{aligned}
&\quad\,\int_{\R^3} \int_{\R^3} |v-v_*|^{\gamma+1} f (v_*) g(v) h(v)dv_* dv   
\\
&\lesssim   \int_{\R^3} \int_{\R^3} \langle v-v_*\rangle^{-(\gamma + 1)}  |v-v_*|^{\gamma+1} f (v_*)  \langle v_*\rangle^{- (\gamma + 1 ) } |g(v) h(v)|    \langle v \rangle^{\gamma + 1  } dv_* dv 
\\
&\lesssim   \int_{\R^3} \int_{\R^3} ( 1+   |v-v_*|^{\gamma+1} ) f (v_*)  \langle v_*\rangle^{- (\gamma + 1 ) } |g(v) h(v) |   \langle v \rangle^{\gamma + 1  } dv_* dv  := T_1 +  T_2.
\end{aligned}
\end{equation*}
It is easily seen that $T_1 \lesssim  \Vert f \Vert_{ L^1_{-(\gamma+1)} }  \Vert g \Vert_{L^2_{\gamma/2 +1}}  \Vert h \Vert_{L^2_{\gamma/2}}  \lesssim  \Vert f \Vert_{ L^2_5   }  \Vert g \Vert_{L^2_{\gamma/2 +1}}  \Vert h \Vert_{L^2_{\gamma/2}}$. If $\gamma \in [-2, -1)$, we have $\gamma +1 \in(-\frac 3 2,0)$ and by \eqref{21}, 
\begin{equation*}
\begin{aligned}
T_2 &\lesssim \sup_{v \in \R^d} \int_{\R^3}   |v-v_*|^{\gamma+1}| f (v_*) | \langle v_*  \rangle^{ -(\gamma + 1 )} dv_*   \int_{\R^3}|g(v) h(v)|    \langle v \rangle^{\gamma + 1  } dv \lesssim  \Vert f \Vert_{L^2_5}  \Vert g \Vert_{H^1_{\gamma/2 +1}}  \Vert h \Vert_{L^2_{\gamma/2}}.
\end{aligned}
\end{equation*}
If $\gamma \in [-3, -2)$, by Hardy-Littlewood-Sobolev (HLS) inequality (cf. \cite[Theorem 1.1, pp. 119]{Stein1971}) and H\"{o}lder's inequality, we have
\begin{multline*}
T_2 \lesssim \Big\|\int_{\R^3}   |v-v_*|^{\gamma+1} f (v_*)  \langle v_*  \rangle^{ -(\gamma + 1 )} dv_*\Big\|_{L^6_v}\|h g \langle  v \rangle^{\gamma+1} \Vert_{L^{\frac 6 5}}\\
\lesssim  \Vert f \Vert_{L^p_{- (\gamma + 1 ) }}  \Vert g \Vert_{L^3_{\gamma/2 +1}}  \Vert h \Vert_{L^2_{\gamma/2}}\lesssim \Vert f \Vert_{L^2_5}  \Vert g \Vert_{H^1_{\gamma/2 +1}}  \Vert h \Vert_{L^2_{\gamma/2}},
\end{multline*}
with  $p  = \frac {6} {9+2\gamma} \in [1, 2]$, and \eqref{25} is proved. Finally, we come to prove \eqref{26}. For $\gamma  \in [-2, 1]$, \eqref{26} can be seen as a special case of \eqref{25}. We focus on the case $\gamma \in (-3, -2)$. Similarly to \eqref{2t}, we have
\begin{equation*}
\begin{aligned}
&\quad\,\int_{\R^3} \int_{\R^3} |v-v_*|^{\gamma} |f (v_*)| |g(v)|^2 dv_* dv   
\lesssim   \int_{\R^3} \int_{\R^3} ( 1+   |v-v_*|^{\gamma} ) |f (v_*)|  \langle v_*\rangle^{- \gamma   } |g(v)|^2   \langle v \rangle^{\gamma  } dv_* dv  := I_1 +  I_2.
\end{aligned}
\end{equation*}
It is seen that $I_1 \lesssim  \Vert f \Vert_{ L^1_{-\gamma} }  \Vert g \Vert_{L^2_{\gamma/2 }}^2   \lesssim  \Vert f \Vert_{ L^2_5   }  \Vert g \Vert_{L^2_{\gamma/2 }}^2$. For $I_2$, by HLS inequality as above, we have
\begin{equation*}
\begin{aligned}
I_2 &\lesssim \Vert f \Vert_{L^2_{- \gamma   }}\Big\|\int_{\R^3}  |v-v_*|^{\gamma}  |g(v)|^2 \langle v \rangle^{\gamma }  dv \Big\|_{L^2_{v_*}}
\lesssim \Vert f \Vert_{L^2_{- \gamma   }}  \Vert |g|^2 \langle  v \rangle^{\gamma} \Vert_{L^{p}}
\lesssim  \Vert f \Vert_{L^2_5}  \Vert g \Vert_{L^{2p}_{\gamma/2 }}^2 \lesssim \Vert f \Vert_{L^2_5}  \Vert g \Vert_{H^1_{\gamma/2}}^2,
\end{aligned}
\end{equation*}
where $p = \frac {6} {9+2\gamma}$ satisfies $2p \in [2, 6]$. This gives \eqref{26} and the proof of Theorem \ref{L54} is finished. 
\end{proof}

We also recall some basic interpolation on $x$. 
\begin{lem}
For any function $f$ and constant $ k \in \R$, we have $
\Vert f g \Vert_{H^2_x} \lesssim \min_{c + d = 2} \{ \Vert f \Vert_{H^c_x} \Vert g \Vert_{H^d_x}  \}.
$ More precisely, 
\begin{align}
\label{Linfty}
\Vert f \Vert_{L^\infty_x L^2_v} \lesssim \Vert f \langle v \rangle^{k} \Vert_{H^1_{x}L^2_v}+ \Vert f \langle v \rangle^{-k}  \Vert_{H^2_xL^2_v},
\end{align}
\begin{equation}
\label{L3}
\Vert f  \Vert_{L^3_xL^2_v}  \lesssim\Vert f \langle v \rangle^{k} \Vert_{L^2_{x}L^2_v}+ \Vert f \langle v \rangle^{-k} \Vert_{H^1_xL^2_v},
\end{equation}
and for any constant $s \in (0, 1)$
\begin{equation}
\label{Ls}
\Vert f   \Vert_{H^1_v}
 \lesssim  \Vert f \langle v \rangle^{k}  \Vert_{H^s_v} + \Vert f \langle v \rangle^{ -k s /(1-s)} \Vert_{H^{1+s}_v}. 
\end{equation}
\end{lem}
\begin{proof}
	Firstly, we use extension theorem \cite[Thoerem VI.5, pp. 181]{Stein1971} to extend function $f$ in domain $\Omega$ with Lipschitz boundary to a function $Ef$ in $\R^3$ such that $Ef=f$ in $\Omega$ and $\|Ef\|_{H^k(\R^3)}\lesssim \|f\|_{H^k(\Omega)}$
	for any $k\ge 0$. 
	By Gagliardo-Nirenberg interpolation inequality on $\R$ and $\R^2$
	(cf. \cite[Theorem 12.83]{Leoni2017} and \cite[Page 125]{Nirenberg1959}), we obtain 
	\begin{align*}
		\|Ef\|_{L^\infty_{x_1}(\R)}\lesssim \|\pa_{x_1}Ef\|_{L^2_{x_1}(\R)}^{1/2}\|Ef\|_{L^2_{x_1}(\R)}^{1/2},\quad\text{and}\quad \|Ef\|_{L^\infty_{x_2,x_3}(\R^2)}\lesssim \|\na_{x_2,x_3}Ef\|_{L^2_{x_2,x_3}(\R^2)}. 
	\end{align*}
	Combining the above estimates, we have 
		$\|f\|_{L^\infty_x(\Omega)}\lesssim \big\|\|Ef\|_{L^\infty_{x_1}(\R)}\big\|_{L^\infty_{x_2,x_3}(\R^2)}\lesssim \|f\|_{H^2_{x}(\Omega)}^{1/2}\|f\|_{H^1_{x}(\Omega)}^{1/2}$,  
which implies \eqref{Linfty}. Also, by Gagliardo-Nirenberg interpolation inequality on $\R^3$, 
we have $$\|Ef\|_{L^3_x(\R^3)}\lesssim \|Ef\|_{L^2_x(\R^3)}^{1/2}\|\na_xEf\|_{L^2_x(\R^3)}^{1/2}\lesssim\|f\|_{L^2_x(\Omega)}^{1/2}\|f\|_{H^1_x(\Omega)}^{1/2},$$  which gives \eqref{L3}. For \eqref{Ls}, by Young's inequality, $\<\eta\>\lesssim \<\eta\>^s\<v\>^k + \<\eta\>^{1+s}\<v\>^{-\frac{ks}{1-s}} $. Hence $\<\eta\>$ is a symbol in $S(\<\eta\>^s\<v\>^k + \<\eta\>^{1+s}\<v\>^{-\frac{ks}{1-s}})$ (cf. \cite{Deng2020a}), where $\eta$ is the Fourier variable of $v$. Then by \cite[Lemma 2.3 and Corollary 2.5]{Deng2020a}, we have  $\Vert f   \Vert_{H^1_v}\lesssim  \Vert f \langle v \rangle^{k}  \Vert_{H^s} + \Vert f \langle v \rangle^{ -k s /(1-s)} \Vert_{H^{1+s}}$. This completes the Lemma. 
\end{proof}

\subsubsection{Upper and lower bound for the non-cutoff Boltzmann operator.} 

In this subsection, we consider the Boltzmann collision operator $Q(f, g)$. 
\begin{lem}[\cite{H}, Theorem 1.1]\label{L21} Suppose $\gamma \in (-3, 1], s \in (0, 1), \gamma+2s >-1$. 
Let $w_1, w_2 \in \R$, $a, b \in[0, 2s]$ with $w_1+w_2 =\gamma+2s$  and $a+b =2s$. Then for any functions $f, g, h$ we have \\
(1) if $\gamma + 2s >0$, then $|(Q(g, h), f)_{L^2_v}| \lesssim (\Vert g \Vert_{L^1_{\gamma+2s +(-w_1)^++(-w_2)^+}}  +\Vert g \Vert_{L^2} ) \Vert h \Vert_{H^a_{w_1}}   \Vert f \Vert_{H^b_{w_2}} $,\\
(2) if $\gamma + 2s = 0$, then $|(Q(g, h), f)_{L^2_v}|\lesssim (\Vert g \Vert_{L^1_{w_3}} + \Vert g \Vert_{L^2}) \Vert h \Vert_{H^a_{w_1}}\Vert f \Vert_{H^b_{w_2}} $,
where $w_3 = \max\{\delta,(-w_1)^+ +(-w_2)^+ \}$, with $\delta>0$ sufficiently small.\\
(3) if $-1< \gamma + 2s < 0$, then $|(Q(g, h), f)_{L^2_v}|\lesssim  (\Vert g \Vert_{L^1_{w_4}} + \Vert g \Vert_{L^2_{-(\gamma+2s)}}) \Vert h \Vert_{H^a_{w_1}}   \Vert f \Vert_{H^b_{w_2}} $, where $w_4= \max\{-(\gamma+2s), \gamma+2s +(-w_1)^+ +(-w_2)^+ \}$.
\end{lem}
As a first application, we have the following corollaries. 
\begin{cor}\label{C22}
Suppose $\gamma \in (-3, 1], s \in (0, 1), \gamma+2s >-1$.  For any multi-indices $|\beta | \le 2$,  any constant $k \ge 0$ and any functions $f, g$,  there exists some constant $C_k >0$, such that
\begin{align*}
(Q(\partial_\beta \mu, f  ), g \langle  v \rangle^{2k} )\le C_k \Vert f \Vert_{H^s_{k + \gamma/2 + 2s}}  \Vert g \Vert_{H^s_{k+\gamma/2}}.
\end{align*}
\end{cor}

\begin{lem} [\cite{CHJ}, Lemma 3.3]\label{L23}
Suppose that $-3 <\gamma \le 1$. For any $k \ge 14$, and functions $g, h$, we have
\begin{equation}\label{25a}
\begin{aligned}
\quad\,|(Q (h, \mu), g \langle v \rangle^{2k}) |
 &\le   \Vert b(\cos\theta) \sin^{k-\frac {3+\gamma} 2} \frac \theta 2 \Vert_{L^1_\theta}    \Vert h \Vert_{L^2_{k+\gamma/2, *}}\Vert g \Vert_{L^2_{k+\gamma/2, *}} + C_k \Vert h \Vert_{L^2_{k+\gamma/2-1/2}}\Vert g \Vert_{L^2_{k+\gamma/2-1/2}}
\\
&\le   \Vert b(\cos\theta) \sin^{k- 2} \frac \theta 2 \Vert_{L^1_\theta}    \Vert h \Vert_{L^2_{k+\gamma/2, *}}\Vert g \Vert_{L^2_{k+\gamma/2, *}}  + C_k \Vert h \Vert_{L^2_{k+\gamma/2-1/2}}\Vert g \Vert_{L^2_{k+\gamma/2-1/2}},
\end{aligned}
\end{equation}
for some constant $C_k>0$.  Moreover, for any $|\beta| \le 2$ we have
\begin{align}\label{25b}
|(Q (h, \partial_\beta\mu), g \langle v \rangle^{2k}) | \le  C_k \Vert h \Vert_{L^2_{k+\gamma/2}}\Vert g \Vert_{L^2_{k+\gamma/2}} .
\end{align}
\end{lem}

\begin{rmk}
In \cite[Lemma 3.3]{CHJ},  the authors only prove the first statement of Lemma \ref{L23}, but the second statement can be proved in the same way by just replacing $\mu$ by $\partial_\beta \mu$.
\end{rmk}
\begin{thm}[\cite{CHJ}, Theorem 3.1] \label{T25}
Suppose that $-3<\gamma\le 1, s \in (0, 1), \gamma+2s>-1$, $k\ge 14$ and $G= \mu +g \ge 0$. If there exists $A_1,A_2>0$ such that 
\begin{align*}
G \ge 0,\quad  \Vert G \Vert_{L^1} \ge A_1, \quad \Vert G \Vert_{L^1_2} +\Vert G \Vert_{L \log L} \le A_2,
\end{align*}
then there exist some constants $\gamma_1, C_k>0$, such that
\begin{equation}
\label{ineq1}
\begin{aligned}
(Q(G, f), f \langle v \rangle^{2k} )
&\le   - \frac {1} {8} \Vert  b(\cos \theta) \sin^2 \frac \theta 2\Vert_{L^1_\theta}\Vert f \Vert_{L^2_{k+\gamma/2, *}}^2  - \gamma_1 \Vert f \Vert_{H^s_{k+\gamma/2}}^2 + C_k  \Vert f \Vert_{L^2}^2
\\
&+C_k\Vert f \Vert_{L^2_{14}} \Vert g \Vert_{H^s_{ k+\gamma/2 }}\Vert f \Vert_{H^s_{ k + \gamma/2}} +C_k\Vert g \Vert_{L^2_{14} } \Vert f \Vert_{H^s_{ k + \gamma/2}}^2.
\end{aligned}
\end{equation}
\end{thm}
\begin{rmk}In \cite[Theorem 3.1]{CHJ}, the authors explored both exponential weight perturbation and polynomial weight perturbation cases. In the exponential weight scenario, the estimation involves an infinite sum, necessitating $k\ge 22$ for convergence. However, for the polynomial case, it has been demonstrated that $k\ge 14$ is sufficient, employing the same technique. Note that whenever $\|g\|_{L^1_v}$ is small, we have $\|\mu+g\|_{L^1_v}\ge 1-\|g\|_{L^1_v}>A_1$. Moreover, we couldn't obtain a better estimate for the term $(Q(f,g),g\<v\>^{2k})$ as the Landau case (see \eqref{28}) and thus we should put the term $\mu+g\ge 0$ together.  
\end{rmk}
\begin{lem}[\cite{CHJ}, Lemma 2.4 + Lemma \ref{L21}] \label{L26}
Suppose $\gamma \in (-3, 1], s \in (0, 1), \gamma+2s >-1$. For any functions $f, g, h$ and $k \ge 14$, we have
\begin{multline}\label{213a}
(Q(f, g), h \langle v \rangle^{2k} ) \lesssim  \Vert f \Vert_{L^2_{14}} \min\{\|g\|_{H^s_{k+\gamma/2}}\|h\|_{H^s_{k+\gamma/2+2s}},\,\Vert g \Vert_{H^s_{k+\gamma/2+2s }}  \Vert h \Vert_{H^s_{k+\gamma/2}}\}\\ +  \Vert g \Vert_{L^2_{14}}  \Vert f \Vert_{H^s_{k+\gamma/2}}  \Vert h \Vert_{H^s_{k+\gamma/2}} .
\end{multline}
In particular, by duality we have
\begin{align}\label{2134}
\Vert Q(f, g) \Vert_{H^{-s}_{k-\gamma/2}} \lesssim  \Vert f \Vert_{L^2_{14}} \Vert g \Vert_{H^s_{k+\gamma/2+2s }}  +  \Vert g \Vert_{L^2_{14}}  \Vert f \Vert_{H^s_{k+\gamma/2}} .
\end{align}
\end{lem}

\subsubsection{Upper and lower bound for Landau operator.}
Next, we give some results about the Landau collision operator $Q(f,g)$.  
\begin{lem}\label{L51}
  Denote $m = \langle v \rangle^k$ with $k \ge 7$. For any $-3 \le \gamma \le 1, |\beta| \le 2$ and functions $f, g$, we have
\begin{align}\label{26a}
|(Q(f, \partial_\beta \mu),  g \langle v \rangle^{2k} )| \le C_k \Vert f  \Vert_{L^2_5} \Vert g \Vert_{L^2_5},
\end{align}
and 
\begin{align}\label{26b}
(Q(\mu , f) , f \langle v \rangle^{2k}) \le -\gamma_1 \Vert f \Vert_{L^2_{D}(m)}^2 + C_k \Vert f \Vert_{L^2}^2,
\end{align}
for some constants $\gamma_1, C_k>0$.
\end{lem}
\begin{proof}
The first estimate follows from \cite[Lemma 2.5]{CM} and  in \cite[Lemma 2.12]{CTW} with $\mu$ replacing by $\partial_\beta \mu$ if necessary. The second estimate follows from \cite[Lemma 2.3]{CM} and \cite[Lemma 2.7]{CTW}. Note that the work \cite{CTW} is for the case $\gamma\in[-2,1]$ and \cite{CM} is for $\gamma\in[-3,-2)$. 
\end{proof}

\begin{lem}[\cite{CTW}, Lemma 3.5 and \cite{CM}, Lemma 4.3]\label{L52}
For any $-3 \le \gamma \le 1$, $k \ge 7$ and any functions $f, g, h$, we have 
\begin{align}\label{27}
|(Q(f, g) , h \langle v \rangle^{2k} )| \le C \Vert f \Vert_{L^2_7} \min\{ \Vert g \Vert_{L^2_{D, k+1} }\Vert h \Vert_{L^2_{D,k} },  \, \Vert g \Vert_{L^2_{D,k} }\Vert h \Vert_{L^2_{D, k+1} }\}.
\end{align}
By duality, we have 
\begin{align}
	\label{27a}
	\|Q(f,g)\|_{H^{-1}_{k-\gamma/2}}\lesssim \|f\|_{L^2_7}\|g\|_{L^2_{D,k+1}}. 
\end{align}
When $h=g$, we have a better estimate: 
\begin{align}\label{28}
(Q(f, g) , g \langle v \rangle^{2k} ) \le C \Vert f \Vert_{L^2_7} \Vert g \Vert_{L^2_{D,k} }^2.
\end{align}
\end{lem}

\begin{proof} The proof is similar to \cite[Lemma 3.5]{CTW}  and \cite[Lemma 4.3]{CM}  with a little modification. Denoting $m  =\langle v \rangle^{k}$, we have
\begin{equation*}
\begin{aligned} 
(Q(f, g), h m^2) &= \int_{\R^3} \partial_j  \{(\phi^{ij} * f) \partial_i g - (\partial_{i}\phi^{ij} *f) g \} h m^2 dv
\\
&=  - \int_{\R^3} (\phi^{ij} *f) \partial_i g \partial_j h m^2 dv - \int_{\R^3}  (\phi^{ij} *f )  \partial_i g  h \partial_j m^2 dv      
\\
&+\int_{\R^3}  (\partial_{i}\phi^{ij} *f) g \partial_j h m^2 dv   +    \int_{\R^3} (\partial_{i}\phi^{ij} *f) g h \partial_j m^2  dv : =T_1 +T_2 +T_3 +T_4.
\end{aligned}
\end{equation*}
For the $T_1, T_2$ term, by \cite[Lemma 3.5]{CTW}  and \cite[Lemma 4.3]{CM},  we have
$|T_1| +|T_2| \le C \Vert f \Vert_{L^2_7} \Vert g \Vert_{L^2_{D,k} }\Vert h \Vert_{L^2_{D,k} }$.
Now we give a better estimate for the $T_3, T_4$ term. Since $\pa_{i}\phi^{ij} (v-v_*) \lesssim |v-v_*|^{\gamma+1}$,  $\partial_j  m^2 \le C\langle v \rangle^{-1} m^2$, for the term $T_3, T_4$, by \eqref{25}, we have 
\begin{align*}
|T_3|& \le C \Vert f \Vert_{L^2_7}\min\{ \Vert g \Vert_{L^2_{D, k+1 } }\Vert h \Vert_{L^2_{D,k} },\, \Vert g \Vert_{L^2_{D,k} }\Vert h \Vert_{L^2_{D, k+1} }\}, \quad|T_4| \le C \Vert f \Vert_{L^2_7} \Vert g \Vert_{L^2_{D,k} }\Vert h \Vert_{L^2_{D,k} },
\end{align*}
and \eqref{27} is thus proved. For the case $g =h$, we give a better estimate for $T_3$. In fact, 
\begin{multline*}
T_3 = \int_{\R^3}  (\partial_{i}\phi^{ij} *f) g \partial_j g m^2 dv = -\frac 1 2 \int_{\R^3} (\partial_{ij}\phi^{ij} *f) g^2m^2 dv - \frac 1 2 \int_{\R^3} (\partial_{i}\phi^{ij} * f )  \partial_j m^2 g^2 dv :=T_{31} +T_{32}. 
\end{multline*}
For the $T_{32}$ term, similar to the $T_4$ term, by \eqref{25}, we have $T_{32} \le C \Vert f \Vert_{L^2_7} \Vert g \Vert_{L^2_{D,k} }^2$.
For the $T_{31}$ term, if $\gamma =-3$, then $\partial_{ij}\phi^{ij} =  -8\pi \delta_0$. It is easily seen that 
\begin{align*}
T_{31} \le C \int_{\R^3} |f| g^2 m^2 dv \le C  \Vert f \Vert_{L^2_{-\gamma}} \Vert g^2 \Vert_{L^2_{2k+\gamma}}  \le C  \Vert f \Vert_{L^2_{5}} \Vert g \Vert_{L^4_{k+\gamma/2}}^2\le C  \Vert f \Vert_{L^2_{5}} \Vert g \Vert_{L^2_{D,k}}^2.
\end{align*}
For the case $\gamma \in [0, 1]$,  we have
\begin{align*}
T_{31} \le C \int_{\R^3} |v-v_*|^\gamma |f_*| g^2 m^2 dv dv_* \le C  \Vert f \Vert_{L^1_{\gamma}} \Vert g \Vert_{L^2_{k+\gamma/2}}^2\le C  \Vert f \Vert_{L^2_{5}} \Vert g \Vert_{L^2_{D,k}}^2.
\end{align*}
The case $\gamma \in (-3, 1 ]$ follows from \eqref{26}. Collecting the above estimates on $T_{31}$ and $T_{32}$, we obtain \eqref{28} and complete the proof of Lemma \ref{L52}. 
\end{proof}
For the Landau operator, we need another upper bound which writes

\begin{lem}\label{L55}
For the Landau operator $Q$, for any $\gamma \in [-3, 1]$ and any function $f, g, h$, we have
\begin{align*}
|(Q(f, g), h)| \le \Vert f \Vert_{L^2_5}\Vert g \Vert_{L^2} \Vert h \Vert_{H^2_{5}}, \quad |(Q(f, g), h)| \le \Vert f \Vert_{L^2_5}\Vert g \Vert_{L^2_{10}} \Vert h \Vert_{H^2_{-5}},
\end{align*}
\end{lem}

\begin{proof}
By using integration by parts, we easily have
\begin{equation*}
\begin{aligned} 
(Q(f, g), h)  &= \int_{\R^3} \Big(- (\phi^{ij} *f )\partial_j g   \partial_i  h  +  (\pa_j\phi^{ij} * f)  g    \partial_i  h\Big)   dv= \int_{\R^3} \Big((  \phi^{ij} *f  )    g \partial_i  \partial_j h + 2 (\pa_j\phi^{ij} * f)  g    \partial_i  h  \Big) dv,
\end{aligned}
\end{equation*}
then by the homogeneity of $\phi^{ij}$ and $\pa_j\phi^{ij}$, we have $|\phi^{ij} (v-v_*)| \lesssim |v-v_*|^{\gamma+2}, \quad |\pa_j\phi^{ij}(v-v_*)| \lesssim |v-v_*|^{\gamma+1}$. So the theorem follows directly from \eqref{23} and \eqref{24}. 
\end{proof}

\subsection{High-order specular boundary conditions}
When $\Omega$ is the union of cubes given by \eqref{Omega}, we give the high-order compatible specular boundary condition from \cite[Lemma 3.1 and 3.2]{Deng2021e}. Although the proof in \cite{Deng2021e} is given for exponential decay perturbations $F=\mu+\mu^{1/2}f$, similar calculations can be applied to polynomial perturbations $F=\mu+f$. 
\begin{lem}[{\cite[Lemma 3.1]{Deng2021e}}] \label{highspecular}
	Let $(f,\phi)$ be the solution to \eqref{vplocal1} with boundary conditions \eqref{specular} and \eqref{Neumann}. Fix $i\in\{1,2,3\}$, $x\in \Gamma_i$ and $v\cdot n(x)\neq 0$. Then we have the following identities on boundary $\Gamma_i$:
	\begin{equation}\label{33a}
		\pa^\al f(x,v)=(-1)^{\al_i}\pa^\al f(x,R_xv),
	\end{equation}
for any $\al=(\al_1,\al_2,\al_3)\in\mathbb{N}^3$. 
\end{lem}

As a corollary, by definition \eqref{abc}, we have the following boundary values for $[a_\pm,b,c]$.
\begin{lem}[{\cite[Lemma 3.2]{Deng2021e}}]
	Let $(f,\phi)$ be the solution to \eqref{vplocal1} with boundary conditions \eqref{specular} and \eqref{Neumann}. Define $[a_\pm,b,c]$ by \eqref{abc}. 
	For $i=1,2,3$, $j\neq i$ and any $x\in\Gamma_i$, we have 
	\begin{align}\label{boundaryabc}
		\partial_{x_i}c(x) =\partial_{x_i}a_\pm(x)=\partial_{x_i}b_j(x)=\pa_{x_ix_i}b_i(x)=	b_i(x) = 0.
	\end{align}
\end{lem}

\subsection{Estimates on the weight}
 Recall the weight function $w=w(|\al|,|\beta|)=w(\al,\beta)$ is defined in \eqref{functionw}. We have the following properties of $w$.


\begin{lem}\label{L27}
	Assume $-3<\gamma\le 1$, $1/2\le s<1, \gamma+2s >-1$ for Boltzmann case and $-3\le\gamma\le 1$ for Landau case. 
For any multi-indices $\al,\beta$ and $k\in\mathbb N$, 
   $w(\alpha, \beta) $ satisfies the following properties.

\noindent  $\bullet$ For $ |\alpha_1| < |\alpha|,~ |\beta_1 | < |\beta|$, we have 
\begin{align}\label{29}
w(\alpha, \beta) \langle v \rangle^{6s}\le  w(\alpha, \beta_1), \quad w(\alpha, \beta) \langle v \rangle^{6s}  \le  w(\alpha_1, \beta) \quad\text{for the Boltzmann case}, 
\end{align}
\begin{align}\label{210}
w( \alpha, \beta)\langle v \rangle^3 \le  w( \alpha, \beta_1), \quad w( \alpha, \beta)\langle v \rangle^{3}\le  w( \alpha_1, \beta)\quad \text{for the Landau case}. 
\end{align}

\noindent $\bullet$ For both Boltzmann and Landau case, and any  $|\alpha| \ge 0$,  $|\beta| \ge 1$, we have 
\begin{align}\label{211}
w(\alpha, \beta) \le \langle v \rangle^{\gamma{{-1}} } w (|\alpha|+1, |\beta|-1 ).
\end{align}

\noindent $\bullet$ For the Boltzmann case and any $|\alpha|\ge 1$, $|\beta|\ge 0$, we have 
\begin{align}\label{213}
w(\alpha, \beta) \le  w(|\alpha|-1, |\beta|)^s w(|\alpha|-1,|\beta|+1)^{1-s} \langle v \rangle^{\gamma}.
\end{align}
\end{lem}

\begin{proof}
%
	Recall that $q = 6s -3(\gamma{{-1}}) $, $  p =q + \gamma{{-1}}$, $r = 2q + 6$, {for Boltzmann case}, and $q =  3-(\gamma{{-1}})$, $ p = 3$, $r = 2q + 6$ for the Landau case. 
\eqref{29} and \eqref{210} are just from the fact that $-3\le\gamma\le 1$ and 
\begin{align*}
p = -3(\gamma{{-1}})+ 6s + \gamma{{-1}} \ge 6s  , \quad q \ge 6s,
\end{align*}
for the Boltzmann case and $p, q\ge 3$ for the Landau case.
The estimate \eqref{211} follows from the fact that
\begin{multline*}
\quad\,k- p|\alpha|  - q|\beta| + r \le k - p(|\alpha| + 1)  - q ( |\beta| - 1)  +  r  +  \gamma{{-1}}   \quad
 \Longleftrightarrow  \quad  q -p +\gamma{{-1}} \ge 0,
\end{multline*}
and we conclude \eqref{211} from the definition of $p, q$. \eqref{213} is equivalent to 
\begin{multline*}
k- p|\alpha|  - q|\beta| + r\le (k - p(|\alpha| -  1)  - q  |\beta|   +  r )s  + (k - p(|\alpha| - 1)  - q ( |\beta| + 1)  +  r )(1-s) +\gamma,
\end{multline*}
which is equivalent to 
\begin{multline*}
0 \le p s + p (1-s) - q (1-s) +\gamma  \ \Longleftrightarrow \     0 \le q s + (p-q)  + \gamma  = (6s -3(\gamma{{-1}})) s  +   \gamma{{-1}}  + \gamma.
\end{multline*}
We thus conclude \eqref{213} from the definition of $q$, $1/2\le s<1$ and $-3<\gamma\le 1$. 
\end{proof}

We can directly deduce the following Corollary from Lemma \ref{L27}.
\begin{cor}
Assume the same conditions as in Lemma \ref{L27}. Then we have
\begin{equation}
\label{l31}
\max_{|\alpha|+|\beta| =1}w^2 (\alpha, \beta) \langle v \rangle^{\kappa} \le w(0, 0) w(1, 0) ,
\end{equation}
\begin{equation}
\label{l32}
\max_{|\alpha|+|\beta| =2}w^2 (\alpha, \beta) \langle v \rangle^{\kappa} \le w(1, 0) w(2, 0) , \quad \max_{\substack{|\alpha|+|\beta| =2\\|\beta|\ge 1}}w^2 (\alpha, \beta) \langle v \rangle^{\kappa} \le w( 0, 1) w(1, 1),
\end{equation}
\begin{equation}
\label{linfty2}
\max_{|\alpha|+|\beta| =2}w^2 (\alpha, \beta) \langle v \rangle^{\kappa}\le  w^2(2, 0) \langle v \rangle^{\kappa} \le  w(1, 0)^{4/5} w(2, 0)^{6/5},
\end{equation}
where $\kappa=4s$ for Boltzmann case,  $\kappa=2$ for Landau case. 
\end{cor}

\subsection{Weighted estimates}\label{secpri}
Next, we come to the nonlinear term for the Vlasov-Poisson term. Noticing $w^2(\alpha, \beta)$ is still of the form $\langle v \rangle^k$ for some $k>0$, we have 
\begin{align}\label{Cab}
	\nabla_v (w^2(\alpha, \beta))  = A_{\al,\beta} \frac v {\langle v \rangle^2} w^2(\alpha, \beta),
\end{align}
for some constant $A_{\al,\beta}$ depending on $\al,\beta,k,\gamma$. 
Then we can apply the space-velocity weight $e^{\frac{\pm A_{\al,\beta}\phi}{\<v\>^2}}$.  To the end of this section, we will assume 
\begin{align}
	\label{214c}
	\|\phi\|_{L^\infty_x}\le C< + \infty, \quad \| \psi \|_{L^\infty_x}\le C< + \infty. 
\end{align}
and hence, 
\begin{align}\label{214d}
	\Big|e^{\frac{\pm A_{\al,\beta}\phi}{\<v\>^2}}\Big|\le C<\infty, \quad \Big|e^{\frac{\pm A_{\al,\beta}\psi}{\<v\>^2}}\Big|\le C<\infty. 
\end{align}

\begin{lem}\label{L212}
	Suppose that $-3 \le\gamma\le 1$ for Landau case and $-3<\gamma\le 1$ for Boltzmann case. For $|\alpha| + |\beta| \le 2$ and any functions $f_\pm$, for any function $\phi$ satisfies \eqref{214c}, 
	we have 
	\begin{equation}
		\label{214a}
		\Big|\big(v\cdot\na_x\partial_{\beta}^{\alpha}f_\pm,e^{\frac{\pm A_{\al,\beta}\phi}{\<v\>^2}} \partial^\alpha_{\beta}f_\pm w^2(\alpha, \beta) \big)_{L^2_{x, v}} 
		\mp( \nabla_x \phi \cdot \nabla_v \partial_{\beta}^{\alpha}f_\pm,e^{\frac{\pm A_{\al,\beta}\phi}{\<v\>^2}} \partial^\alpha_{\beta}f_\pm w^2(\alpha, \beta) )_{L^2_{x, v}}\Big|\
		\lesssim \|\na_x\phi\|_{H^2_x}\|f\|^2_{Y_k},
	\end{equation}
	and for $|\beta|>0$, 
	\begin{equation}
		\label{estiv3}
		\sum_{|\beta_1|=1}|( \pa_{\beta_1}v \cdot \nabla_x\partial^\alpha_{\beta-\beta_1} f  ,e^{\frac{\pm A_{\al,\beta}\phi}{\<v\>^2}}\partial^\alpha_{\beta}f  w^2(\alpha, \beta)  )_{L^2_{x, v}}|\le C\sum_{\substack{|\beta_2|=|\beta|-1 \\ |\alpha_2| =|\alpha|+1}} \Vert \partial^{\alpha_2}_{\beta_2}  f w(\alpha_2, \beta_2)\Vert_{L^2_xL^2_{\gamma/2}} \Vert \partial^\alpha_{\beta} f w(\alpha, \beta) \Vert_{L^2_{x}L^2_{\gamma/2} }.
	\end{equation}
%
%
%
%
%
	For $|\al_1|\ge 1$, we have 
	\begin{align}
		\label{estiv5}
		| ( \nabla_x \partial^{\alpha_1} \phi \cdot \nabla_v \partial_{\beta}^{\alpha-\alpha_1}f,e^{\frac{\pm A_{\al,\beta}\phi}{\<v\>^2}} \partial^\alpha_{\beta}f w^2(\alpha, \beta) )_{L^2_{x, v}}|\lesssim \Vert \na_x\phi \Vert_{H^3_x }   \Vert f \Vert_{Y_k}^2,
	\end{align}
	where $A_{\al,\beta}$ is given by \eqref{Cab},
	for some constant $C_k>0$ for both Landau and Boltzmann case. 
\end{lem}

\begin{proof}
	For \eqref{214a}, taking integration by parts with respect to $\na_x$, we have
	\begin{multline*}
		\big(v\cdot\na_x\partial_{\beta}^{\alpha}f_\pm,e^{\frac{\pm A_{\al,\beta}\phi}{\<v\>^2}} \partial^\alpha_{\beta}f_\pm w^2(\alpha, \beta) \big)_{L^2_{x, v}} 
		\pm \frac{A_{\al,\beta}}{2}\big(\frac{v\cdot\na_x\phi}{\<v\>^2}\partial_{\beta}^{\alpha}f_\pm,e^{\frac{\pm A_{\al,\beta}\phi}{\<v\>^2}} \partial^\alpha_{\beta}f_\pm w^2(\alpha, \beta) \big)_{L^2_{x, v}}\\
		=\int_{\pa\Omega}\int_{\R^3}v\cdot n(x)|e^{\frac{\pm A_{\al,\beta}\phi}{2\<v\>^2}}\partial_{\beta}^{\alpha}f_\pm w(\alpha, \beta)|^2\,dvdS(x).  
	\end{multline*}
For the torus case, the right hand side is equal to zero. For the case of union of cubes, we apply Lemma \ref{highspecular} to obtain $|\partial_{\beta}^{\alpha}f_\pm(v)|=|\partial_{\beta}^{\alpha}f_\pm(R_xv)|$ for $x\in\pa\Omega$ and $R_x$ is given by \eqref{Rx}. Thus, by change of variable $v\mapsto R_xv$, 
\begin{align}\label{222}\notag
	&\int_{\pa\Omega}\int_{\R^3}v\cdot n(x)|e^{\frac{\pm A_{\al,\beta}\phi}{2\<v\>^2}}\partial_{\beta}^{\alpha}f_\pm(v) w(\alpha, \beta)|^2\,dvdS(x)\\ \nonumber
	=&\int_{\pa\Omega}\int_{\R^3}R_xv\cdot n(x)|e^{\frac{\pm A_{\al,\beta}\phi}{2\<R_xv\>^2}}\partial_{\beta}^{\alpha}f_\pm(R_xv) w(\alpha, \beta)|^2\,dvdS(x)\\
	=&-\int_{\pa\Omega}\int_{\R^3}v\cdot n(x)|e^{\frac{\pm A_{\al,\beta}\phi}{2\<v\>^2}}\partial_{\beta}^{\alpha}f_\pm(v) w(\alpha, \beta)|^2\,dvdS(x)=0,
\end{align}
where $R_xv\cdot n(x)=-v\cdot n(x)$. 
	Taking integration by parts with respect to $\na_v$ and using \eqref{Cab}, we have 
	\begin{align*}
		\quad\,\mp( \nabla_x \phi \cdot \nabla_v \partial_{\beta}^{\alpha}f_\pm,e^{\frac{\pm A_{\al,\beta}\phi}{\<v\>^2}} \partial^\alpha_{\beta}f_\pm w^2(\alpha, \beta) )_{L^2_{x, v}} 
		&= \mp( \frac{ A_{\al,\beta} v\cdot\na_x\phi\, \phi}{\<v\>^4}  \partial_{\beta}^{\alpha}f_\pm,e^{\frac{\pm A_{\al,\beta}\phi}{\<v\>^2}} \partial^\alpha_{\beta}f_\pm w^2(\alpha, \beta) )_{L^2_{x, v}}\\
		&\quad \pm \frac{1}{2} ( \frac{A_{\al,\beta} v\cdot\nabla_x \phi}{\<v\>^2}  \partial_{\beta}^{\alpha}f_\pm,e^{\frac{\pm A_{\al,\beta}\phi}{\<v\>^2}} \partial^\alpha_{\beta}f_\pm w^2(\alpha, \beta) )_{L^2_{x, v}}.
	\end{align*}
	Noticing $|\frac{v}{\<v\>^4}|\le \<v\>^{\gamma}$ for $-3\le\gamma\le 1$ and taking summation of the above two estimates, we have 
	\begin{align*}
		&\quad\,\Big|\big(v\cdot\na_x\partial_{\beta}^{\alpha}f_\pm,e^{\frac{\pm A_{\al,\beta}\phi}{\<v\>^2}} \partial^\alpha_{\beta}f_\pm w^2(\alpha, \beta) \big)_{L^2_{x, v}} 
		\mp( \nabla_x \phi \cdot \nabla_v \partial_{\beta}^{\alpha}f_\pm,e^{\frac{\pm A_{\al,\beta}\phi}{\<v\>^2}} \partial^\alpha_{\beta}f_\pm w^2(\alpha, \beta) )_{L^2_{x, v}}\Big|
		\\
		&\lesssim \|\nabla_x \phi\|_{L^\infty_x}\|\<v\>^{\frac{\gamma}{2}}  \partial_{\beta}^{\alpha}f_\pm w(\alpha, \beta)\|_{L^2_{x,v}}\|\<v\>^{\frac{\gamma}{2}}  \partial^\alpha_{\beta}f_\pm w(\alpha, \beta)\|_{L^2_{x, v}}\lesssim \|\na_x\phi\|_{H^2_x}\|f\|^2_{Y_k}. 
	\end{align*}
	This yields \eqref{214a}. 
	Here we used the fact \eqref{214c} and its consequence \eqref{214d}.   
	
	\smallskip
		Notice that 
$
		\pa_{i}v \cdot \nabla_x\partial^\alpha_{\beta-e_i} f =\partial^{\alpha+e_i}_{\beta-e_i}f, \quad \forall~ i =1, 2, 3,
$
	where $e_i$ is the unit vector with $i$-th component being $1$. 
	Then \eqref{estiv3} follows from  \eqref{211} and \eqref{214d}. 
	 \smallskip

	When $1\le |\alpha_1| \le 2$, we have
	\begin{align}\label{215a}\notag
		&\quad\,( \nabla_x \partial^{\alpha_1} \phi \cdot \nabla_v \partial_{\beta}^{\alpha-\alpha_1}f,e^{\frac{\pm A_{\al,\beta}\phi}{\<v\>^2}} \partial^\alpha_{\beta}f w^2(\alpha, \beta) )_{L^2_{x, v}} 
		\notag\lesssim \Big(\sum_{|\al_1|=1} \Vert \nabla_x \partial^{\alpha_1} \phi \Vert_{L^\infty_x} \Vert\<v\>^{-\frac{\gamma}{2}} \nabla_v \partial_{\beta}^{\alpha-\alpha_1}f w(\alpha, \beta )\Vert_{L^2_xL^2_v} \\
		&\notag\quad+\sum_{|\al_1|=2}\Vert \nabla_x \partial^{\alpha_1} \phi \Vert_{L^3_x} \Vert\<v\>^{-\frac{\gamma}{2}} \nabla_v \partial_{\beta}^{\alpha-\alpha_1}f w(\alpha, \beta )\Vert_{L^6_xL^2_v} \Big)
		\Vert\<v\>^{\frac{\gamma}{2}} \partial^\alpha_{\beta}f w (\alpha, \beta)\Vert_{L^2_xL^2_v}
		\\
		&\lesssim  \Vert \na_x\phi \Vert_{H^3_x }  \sum_{|\al_1|=1} \Vert\<v\>^{-\frac{\gamma}{2}} \nabla_v \partial_{\beta}^{\alpha-\alpha_1}f w(\alpha, \beta )\Vert_{L^2_xL^2_v} \Vert f \Vert_{Y_k}.
	\end{align}
	For the Boltzmann case, by \eqref{213} and \eqref{Ls}, we have for $|\al_1|=1$ that 
	\begin{align*}
		&\quad\,\Vert \<v\>^{-\frac{\gamma}{2}}\nabla_v \partial_{\beta}^{\alpha-\alpha_1}f w(\alpha, \beta )\Vert_{L^2_xL^2_v}\le  \Vert \<v\>^{-\frac{\gamma}{2}} \partial_{\beta}^{\alpha-\alpha_1}f w(\alpha, \beta )\Vert_{L^2_xH^1_v}
		\\
		&\lesssim \Vert   \partial_{\beta}^{\alpha-\alpha_1}f w(|\alpha|-1, \beta ) \Vert_{L^2_xH^s_{\gamma/2}}   + \Vert   \partial_{\beta}^{\alpha-\alpha_1}f w(|\alpha|-1, |\beta| +1) \Vert_{L^2_xH^{1+s}_{\gamma/2}} 
		\lesssim \Vert f \Vert_{Y_k}.
	\end{align*}
	For the Landau case, by \eqref{210} we have
	\[
		\Vert \<v\>^{-\frac{\gamma}{2}}\nabla_v \partial_{\beta}^{\alpha-\alpha_1}f w(\alpha, \beta )\Vert_{L^2_xL^2_v} \le  \Vert \<v\>^{-\frac{\gamma}{2}} \partial_{\beta}^{\alpha-\alpha_1}f w(\alpha, \beta )\Vert_{L^2_xH^1_v}
		\lesssim \Vert   \partial_{\beta}^{\alpha-\alpha_1}f w(|\alpha|-1, \beta ) \Vert_{L^2_xL^2_{D}}  
		\lesssim \Vert f \Vert_{Y_k}.
	\]
	The proof of \eqref{estiv5} is finished by plugging the above estimates into \eqref{215a}.
\end{proof}

 Next, we compute the linear part for the Poisson term. 
\begin{lem}\label{L211}
	For $|\alpha| + |\beta| \le 2$, any functions $f_\pm, \phi$ satisfying \eqref{vplocal1} and any $\psi$ satisfies \eqref{214c}, 
	there exists $C_k>0$  such that 
	\begin{align}
		\label{estiv4}
		|( \partial^\alpha_{\beta} (\nabla_x \phi \cdot v \mu),e^{\frac{\pm A_{\al,\beta}\psi}{\<v\>^2}} \partial^\alpha_{\beta}f_\pm w^2(\alpha, \beta) )_{L^2_{x, v}}| \le C_k \Vert \na_x\phi \Vert_{H^{2}_x }   \Vert f \Vert_{H^{2}_x L^2_{v}  }.
	\end{align}
When $|\beta|=0$, there exists some $C_\al>0$, such that
\begin{multline}
	\label{231}
	\sum_\pm\pm\big( \partial^\alpha \nabla_x \phi \cdot v \mu, e^{\frac{\pm A_{\al,\beta}\psi}{\<v\>^2}}\partial^\alpha f_\pm w^2(\alpha, \beta)\big)_{L^2_{x, v}}\ge
	C_\al\pa_t\|\partial^\alpha\na_x\phi\|_{L^2_x}^2\\- C\|\pa^\al\na_x\phi\|_{L^2_x}\|\pa^\al\{\I-\P\}f\|_{L^2_xL^2_5}-C\|\partial^\alpha \nabla_x \phi\|_{L^2_x}\|\na_x\psi\|_{H^1_x}\|f\|_{Y_k}.
\end{multline}

\end{lem}

\begin{proof}
	From integration by parts about $\pa_\beta$, we have
	\begin{align}\label{232a}\notag
			\quad\,( \partial^\alpha_{\beta} (\nabla_x \phi \cdot v \mu),e^{\frac{\pm A_{\al,\beta}\psi}{\<v\>^2}} \partial^\alpha_{\beta}f w^2(\alpha, \beta) )_{L^2_{x, v}}
			\notag  &= ( \nabla_x \partial^\alpha \phi \cdot \partial_\beta( v \mu) w^2(\alpha, \beta) ,e^{\frac{\pm A_{\al,\beta}\psi}{\<v\>^2}} \partial^\alpha_{\beta}f )_{L^2_{x, v}}
			\\
			&=\big((-1)^{|\beta|} \nabla_x \partial^\alpha \phi \cdot  \partial_\beta \big(\partial_\beta( v \mu) e^{\frac{\pm A_{\al,\beta}\psi}{\<v\>^2}}w^2(\alpha, \beta) \big), \partial^\alpha f \big)_{L^2_{x, v}}.
	\end{align}
	By Fourier transform, Cauchy-Schwarz inequality, and using the exponential decay of $\mu$, it is direct to show that \eqref{232a} is bounded by 
	\begin{align*}
		\Vert \nabla_x \partial^\alpha \phi \Vert_{L^{2}_x}  \Big \Vert \int_{\R^d}   \partial^\alpha f  \partial_\beta \big(\partial_\beta( v \mu)e^{\frac{\pm A_{\al,\beta}\psi}{\<v\>^2}} w^2(\alpha, \beta) \big) dv  \Big \Vert_{L^{2}_x}
		\lesssim \|\na_x\phi\|_{H^2_x}\Vert f \Vert_{H^{2}_x L^2_v}.
	\end{align*}
When $\beta=0$, we temporarily define the projection 
\begin{align*}
	P f = \Big(\int_{\R^3}f(u)\,du+v\cdot\int_{\R^3} {u}f(u) du
	+(|v|^2-3)\int_{\R^3} \frac {|u|^2-3}{6}f(u) du\Big)\mu. 
\end{align*} 
Then we split $e^{\frac{\pm A_{\al,\beta}\psi}{\<v\>^2}}=1+e^{\frac{\pm A_{\al,\beta}\psi}{\<v\>^2}}-1$ and $f_\pm=P f + \{I-P\}f$ to obtain  
\begin{multline}\label{220a}
	\pm\big( \partial^\alpha \nabla_x \phi \cdot v \mu, e^{\frac{\pm A_{\al,\beta}\psi}{\<v\>^2}} w^2(\alpha, 0)\partial^\alpha f_\pm \big)_{L^2_{x, v}}
	= \pm\big( \partial^\alpha \nabla_x \phi \cdot v \mu, w^2(\alpha, 0)P \partial^\alpha  f_\pm \big)_{L^2_{x, v}}\\
	\pm\big( \partial^\alpha \nabla_x \phi \cdot v \mu, w^2(\alpha, 0)   \{I-P\}\partial^\alpha f_\pm \big)_{L^2_{x, v}} \pm\big( \partial^\alpha \nabla_x \phi \cdot v \mu, (e^{\frac{\pm A_{\al,\beta}\psi}{\<v\>^2}}-1) w^2(\alpha, 0)\partial^\alpha f_\pm \big)_{L^2_{x, v}}. 
\end{multline}
Then one can check that 
\begin{align}\label{240}
	\{I-P\}f_\pm=\{I-P\}\{\II-\PP\}f.
\end{align} 
For the first right hand term of \eqref{220a}, we have from \eqref{19} and \eqref{vplocal1}$_2$ that 
\begin{align}\label{242}\notag
	&\sum_\pm\pm\big( \partial^\alpha \nabla_x \phi \cdot v \mu, w^2(\alpha, 0) P \partial^\alpha  f_\pm \big)_{L^2_{x, v}}
	 \notag= \sum_{i=1}^3\Big( \partial^\alpha \nabla_x \phi \cdot v \mu,   w^2(\alpha, 0)v_i\mu\int_{\R^3}\frac{1}{2} v'_i\partial^\alpha (f_+-f_-)(v')\,dv' \Big)_{L^2_{x,v}}\notag\\
	 &= 2C_{\al}\int_{\Omega}  \int_{\R^3} \na_x\partial^\alpha \phi\cdot v\partial^\alpha  (f_+-f_-)(v)\,dvdx\notag= -2C_{\al}\int_{\Omega}   \partial^\alpha \phi\, \na_x\cdot\int_{\R^3} v\partial^\alpha  (f_+-f_-)(v)\,dvdx\\
		&= -2C_{\al}\int_{\Omega}   \partial^\alpha \phi\, \pa_t\Delta_x\phi dx
		= C_{\al}\,\pa_t\|\partial^\alpha\na_x\phi\|_{L^2_x}^2,
\end{align}
for some constant $C_{\al,\beta}$ depending only on $\al,\beta$. 
Here we can directly take integration by parts for the case of torus. When $\Omega$ is the union of cubes, we need to verify the zero boundary values as the following. 
Fix $i=1,2,3$. If $\al_i=1$, then $\pa_{x_i}\phi=0$ on $\Gamma_i$. If $\al_i=0,2$, then by Lemma \ref{highspecular} and change of variable $v\mapsto R_xv$, we have on $\Gamma_i$ that 
\begin{align*}
	\int_{\R^3} v_i\partial^\alpha  (f_+-f_-)(v)\,dv = \int_{\R^3} (R_xv)_i\partial^\alpha  (f_+-f_-)(R_xv)\,dv
	= -\int_{\R^3} v_i\partial^\alpha  (f_+-f_-)(v)\,dv =0. 
\end{align*}
Note that tangent derivatives don't affect the zero boundary values. This completes the integration by parts in \eqref{242}. 
For the second right hand term of \eqref{220a}, noticing exponentially velocity decay in $\mu$ and using \eqref{240}, we estimate it by 
\begin{align*}
	\big|\big( \partial^\alpha \nabla_x \phi \cdot v \mu, w^2(\alpha, \beta)\{I-P\}\partial^\alpha f_\pm \big)_{L^2_{x, v}}\big|&\lesssim \|\pa^\al\na_x\phi\|_{L^2_x}\|\pa^\al\{\II-\PP\}f\|_{L^2_xL^2_5}. 
\end{align*}
For the third term in \eqref{220a}, noticing 
$\big|e^{\frac{\pm A_{\al,\beta}\psi}{\<v\>^2}}-1\big|\lesssim A_{\al,\beta}\|\psi\|_{L^\infty}\lesssim A_{\al,\beta}\|\na_x\psi\|_{H^1_x}$, we have 
\[\big|\big( \partial^\alpha \nabla_x \phi \cdot v \mu, (e^{\frac{\pm A_{\al,\beta}\psi}{\<v\>^2}}-1) w^2(\alpha, \beta)\partial^\alpha f_\pm \big)_{L^2_{x, v}}\big|\lesssim \|\partial^\alpha \nabla_x \phi\|_{L^2_x}\|\na_x\psi\|_{H^1_x}\|f\|_{Y_k}.
\]
The above four estimates imply \eqref{231} and we conclude Lemma \ref{L211}.  

\end{proof}

In \cite{CHJ}, one only needs to compute $x$ derivative term, but for our Vlasov-Poisson system we also need estimates  about $v$ derivative term.

\begin{lem}\label{L29} Suppose that $\gamma \in (-3, 1], s \in (0, 1), \gamma+2s >-1$ for Boltzmann case and $\gamma\in[-3,1]$ for Landau case. For any $|\alpha| \ge 0$, $|\beta| \le 2$, $k \ge 7$, there exist constants $C_k >0$ such that, for any functions $f, g,\psi $, we have for the Boltzmann case, 
%
\begin{equation*}
\begin{aligned} 
&\quad\,|(\partial^\alpha_\beta Q (f, \mu),e^{\frac{\pm A_{\al,\beta}\psi}{\<v\>^2}}  \partial^\alpha_\beta g w^2(\alpha, \beta)  )_{L^2_{x, v} }   |\le  \Vert b(\cos\theta) \sin^{k-2 } \frac \theta 2 \Vert_{L^1_\theta}   \Vert \partial^\alpha_\beta f w(\alpha, \beta) \Vert_{L^2_x L^2_{\gamma/2, *}}\Vert \partial^\alpha_\beta g  w(\alpha, \beta) \Vert_{L^2_x L^2_{\gamma/2, *}} 
\\
&\quad+ C_k \Vert \partial^\alpha_\beta f w(\alpha, \beta) \Vert_{L^2_{x} L^2_{\gamma/2-1/2}}   \Vert \partial^\alpha_\beta g w(\alpha, \beta)  \Vert_{L^2_{x}L^2_{\gamma/2-1/2}} 
+ C_k \sum_{\beta_1 < \beta} \Vert \partial^{\alpha}_{\beta_1} f w(\alpha, \beta_1) \Vert_{L^2_x L^2_{\gamma/2}}\Vert \partial^\alpha_\beta g w(\alpha, \beta) \Vert_{L^2_x L^2_{\gamma/2  }},
\end{aligned}
\end{equation*}
and for the Landau case,
\begin{equation*}
\begin{aligned} 
|(\partial^\alpha_\beta Q (f, \mu),e^{\frac{\pm A_{\al,\beta}\psi}{\<v\>^2}}  \partial^\alpha_\beta g w^2(\alpha, \beta)  )_{L^2_{x, v} }   | \le C_k \sum_{\beta_1 \le \beta} \Vert \partial^{\alpha}_{\beta_1} f\Vert_{L^2_x L^2_{5}}\Vert \partial^\alpha_\beta g \Vert_{L^2_x L^2_{5}}.
\end{aligned}
\end{equation*}

\end{lem}
\begin{proof}
For both cases, it's easily seen that
\begin{align*}
(\partial^\alpha_\beta Q(f, \mu), e^{\frac{\pm A_{\al,\beta}\psi}{\<v\>^2}}\partial^\alpha_\beta g w^2(\alpha, \beta) )_{L^2_{x, v} } =\sum_{\alpha_1 \le \alpha} ( Q(\partial^\alpha_{\beta_1} f, \partial_{\beta - \beta_1} \mu), e^{\frac{\pm A_{\al,\beta}\psi}{\<v\>^2}}\partial^\alpha_\beta g   w^2(\alpha, \beta))_{L^2_{x, v} }  .
\end{align*}
For the Boltzmann case, we split it into two cases: $\beta_1 = \beta$ and $\beta_1 < \beta$. For the case $\beta_1 = \beta$, by  \eqref{25a} and integration about $x$,  we have
\begin{equation*}
\begin{aligned} 
 |( Q(\partial^\alpha_\beta f,  \mu),\partial^\alpha_\beta g w^2(\alpha, \beta)  )_{L^2_{x, v} }    | \le  & \Vert b(\cos \theta) \sin^{k-2} \frac \theta 2 \Vert_{L^1_\theta}   \Vert \partial^\alpha_\beta f  w(\alpha, \beta)\Vert_{L^2_x L^2_{\gamma/2, *}}\Vert \partial^\alpha_\beta g w(\alpha, \beta) \Vert_{L^2_x L^2_{  \gamma/2, *  }} 
\\
&+ C_k \Vert \partial^\alpha_\beta f  w(\alpha, \beta)\Vert_{L^2_x L^2_{\gamma/2-1/2}}\Vert \partial^\alpha_\beta g w(\alpha, \beta) \Vert_{L^2_x L^2_{  \gamma/2 -1/2 }} .
\end{aligned}
\end{equation*}
Notice that  
 \begin{align}\label{psi}
 	\big|e^{\frac{\pm A_{\al,\beta}\psi}{\<v\>^2}}-1\big|\le A_{\al,\beta}\|\psi\|_{L^\infty_x}\<v\>^{-2}\lesssim A_{\al,\beta}\|\na_x\psi\|_{H^1_x}\<v\>^{-2}. 
 \end{align}
By    \eqref{25b} we have
\begin{equation*}
\begin{aligned} 
&|( Q(\partial^\alpha_\beta f,  \mu), (e^{\frac{\pm A_{\al,\beta}\psi}{\<v\>^2}}  -1)\partial^\alpha_\beta g w^2(\alpha, \beta)  )_{L^2_{x, v} }    |
\\
\le& C_k \Vert \partial^\alpha_\beta f  w(\alpha, \beta)  \langle v \rangle^{-1}  \Vert_{L^2_x L^2_{\gamma/2}}\Vert (e^{\frac{\pm A_{\al,\beta}\psi}{\<v\>^2}}  -1)  \partial^\alpha_\beta g w(\alpha, \beta) \langle v \rangle \Vert_{L^2_x L^2_{  \gamma/2 }}
\\
\le& C_k \Vert \partial^\alpha_\beta f  w(\alpha, \beta)\Vert_{L^2_x L^2_{\gamma/2-1/2}}\Vert \partial^\alpha_\beta g w(\alpha, \beta) \Vert_{L^2_x L^2_{  \gamma/2 -1/2 }}.
\end{aligned}
\end{equation*}
For $\beta_1 < \beta$, by    \eqref{25b}, \eqref{214d} and integration about $x$,  we have
\[
|( Q(\partial^\alpha_{\beta_1} f, \partial_{\beta - \beta_1} \mu), e^{\frac{\pm A_{\al,\beta}\psi}{\<v\>^2}}\partial^\alpha_\beta g w^2 (\alpha, \beta))_{L^2_{x, v} }    |  \le  C_k \sum_{\beta_1 < \beta} \Vert \partial^\alpha_{\beta_1} f w(\alpha, \beta) \Vert_{L^2_x L^2_{\gamma/2}}   \Vert \partial^\alpha_\beta g w(\alpha, \beta) \Vert_{L^2_x L^2_{\gamma/2}}.
\]
So the proof for the Boltzmann case is completed by gathering the three terms. The Landau case can be proved similarly by taking integration in  \eqref{26a}.
\end{proof}
The next two coercive estimates play a key role. For the Boltzmann case, we have the following. 
\begin{lem}\label{L210}
Remind $\|\cdot\|_{Y_k}$ is given in \eqref{D}. Suppose that $-3<\gamma\le 1, s \in (0, 1), \gamma+2s >-1$ and $k\ge 14$. For any function $f, g$, let $G= \mu +g \ge 0$ satisfies  $\Vert G \Vert_{L^1} \ge A_1, \quad \Vert G\Vert_{L^1_2} + \Vert G\Vert_{L \log L} \le A_2$ for some generic constants $A_1,A_2$. 
Then for $Q$ as the Boltzmann collision operator, and any $|\alpha| + |\beta| \le 2$, 
\begin{equation}\label{236}
\begin{aligned} 
(\partial^\alpha_\beta Q ( \mu+g, f),  \partial^\alpha_\beta f w^2(\alpha, \beta) )_{L^2_{x, v}} 
&\le  - \frac {1} {8} \Vert  b(\cos \theta) \sin^2 \frac \theta 2  \Vert_{L^1_\theta}\Vert \partial^\alpha_\beta  f w(\alpha, \beta)    \Vert_{L^2_x L^2_{\gamma/2, * }}^2 -\gamma_1 \Vert  \partial^\alpha_\beta  f w(\alpha, \beta) \Vert_{L^2_xH^s_{\gamma/2}}^2
\\
&+C_k\Vert \langle v \rangle^{14} f \Vert_{H^2_{x, v} } 
\Vert g \Vert_{Y_k}\Vert f \Vert_{Y_k} 
+C_k\Vert \langle v \rangle^{14} g \Vert_{H^2_{x, v }} \|f\|^2_{Y_k}+ C_{k}  \Vert \partial^\alpha_\beta f \Vert_{L^2_{x, v}}^2\\&+ C_k \sum_{\beta_1 < \beta} \Vert \partial^\alpha_{\beta_1}  f w(\alpha, \beta_1) \Vert_{H^s_{\gamma/2}}\Vert \partial^\alpha_\beta g w(\alpha, \beta)\Vert_{H^s_{\gamma/2}}, 
\end{aligned}
\end{equation}
for some constants $\gamma_1, C_k \ge 0$. 
Moreover, suppose $\psi$ satisfies \eqref{214c}, we have 
\begin{multline}\label{236a}
	\big| \big( \partial^\alpha_\beta Q (g,  f),\big(e^{\frac{\pm A_{\al,\beta}\psi}{\<v\>^2}}-1\big)\partial^{\alpha}_{\beta} h w^2 (\alpha, \beta)\big)_{L^2_{x, v}} \big| \le C_k \|\<v\>^{14}g\|_{H^2_{x,v}}\|\na_x\psi\|_{H^1_x}\|f\|_{Y_k}\|h\|_{Y_k}
	\\+C_k\|\<v\>^{14}f\|_{H^2_{x,v}}\|\na_x\psi\|_{H^1_x}\|g\|_{Y_k}\|h\|_{Y_k},
\end{multline}
and 
\begin{multline}\label{236c}
	\big| \big( \partial^\alpha_\beta Q (g,  f),\partial^{\alpha}_{\beta} h w^2 (\alpha, \beta)\big)_{L^2_{x, v}} \big|+\big| \big( \partial^\alpha_\beta Q (g,  f),e^{\frac{\pm A_{\al,\beta}\psi}{\<v\>^2}}\partial^{\alpha}_{\beta} h w^2 (\alpha, \beta)\big)_{L^2_{x, v}} \big| \\\le C_k \|\<v\>^{14}g\|_{H^2_{x,v}}\min\{\|\<v\>^{2s}f\|_{Y_k}\|h\|_{Y_k}, \|f\|_{Y_k}\|\<v\>^{2s}h\|_{Y_k}\}
	+C_k\|\<v\>^{14}f\|_{H^2_{x,v}}\|g\|_{Y_k}\|h\|_{Y_k}.
\end{multline}
\end{lem}
\begin{rmk}
By assuming $\|g\|_{L^1_v}$ small, we can check $\|G\|_{L^1_v}=\|\mu+g\|_{L^1_v}\ge A_1$ to apply Lemma \ref{L210}. 
\end{rmk}
\begin{proof}
First we have $(\partial^\alpha_\beta Q(G, f), \partial^\alpha f w^2(\alpha, \beta))_{L^2_{x, v}}= \sum_{\alpha_1 \le \alpha, \beta_1 \le \beta} ( Q(\partial^{\alpha_1}_{\beta_1} G, \partial^{\alpha - \alpha_1}_{\beta- \beta_1} f), \partial^{\alpha}_\beta f w^2(\alpha, \beta))_{L^2_{x, v}}$. We split it into several cases. For $\alpha_1 =\beta_1 =0$, after integrating about $x$ in Theorem \ref{T25}, we have
\begin{equation*}
\begin{aligned} 
&\quad\,( Q( G, \partial^{\alpha }_\beta f), \partial^{\alpha}_\beta f w^2(\alpha, \beta))_{L^2_{x, v}} 
\\
&\le - \frac {1} {8} \Vert  b(\cos \theta) \sin^2 \frac \theta 2  \Vert_{L^1_\theta}\Vert \partial^\alpha_\beta f w(\alpha, \beta)\Vert_{L^2_x L^2_{\gamma/2, *}}^2 - \gamma_1  \Vert \partial^\alpha_\beta f w(\alpha, \beta) \Vert_{L^2_x H^s_{\gamma/2}}^2+ C_{k}  \Vert \partial^\alpha_\beta f \Vert_{L^2_{x, v}}^2
\\
&+\int_{\Omega} C_k\Vert \partial^\alpha_\beta f   \Vert_{L^2_{14}} \Vert g w(\alpha, \beta) \Vert_{ H^s_{ \gamma/2 }}\Vert  \partial^\alpha_\beta f w(\alpha, \beta) \Vert_{ H^s_{  \gamma/2}}+ C_k\Vert   g \Vert_{   L^2_{14}} \Vert \partial^\alpha_\beta f w(\alpha, \beta) \Vert_{  H^s_{  \gamma/2}}^2 dx.
\end{aligned}
\end{equation*}
For the case $|\beta_1| > 0, \alpha_1=0$, using $\partial_{\beta_1} G = \partial_{\beta_1} \mu + \partial_{\beta_1} g$, we split it into two parts. By \eqref{29},  if $|\beta_1| > 0$, we have
$
w(\alpha, \beta) \langle v \rangle^{2s} \le w(\alpha, \beta - \beta_1).
$
Then by Corollary  \ref{C22}, we have
\begin{equation*}
\begin{aligned} 
\quad\,|( Q(\partial_{\beta_1} \mu, \partial^\alpha_{\beta - \beta_1} f), \partial^{\alpha}_{\beta} f w^2 (\alpha, \beta))_{L^2_{x, v}} |
&\le  C_k \Vert \partial^\alpha_{\beta - \beta_1}  f w(\alpha, \beta) \Vert_{L^2_x H^s_{  \gamma/2+2s }}\Vert \partial^\alpha_\beta f w(\alpha, \beta) \Vert_{L^2_x H^s_{\gamma/2}} 
\\
&\le   C_k \Vert \partial^\alpha_{\beta - \beta_1}  f w(\alpha, \beta-  \beta_1 ) \Vert_{L^2_x H^s_{  \gamma/2 }}\Vert \partial^\alpha_\beta f w(\alpha, \beta) \Vert_{L^2_x H^s_{\gamma/2}} .
\end{aligned}
\end{equation*}
For the $\partial_{\beta_1} g$ term, by Lemma \ref{L26}, we have
\begin{equation*}
\begin{aligned} 
\quad\,|( Q(\partial_{\beta_1} g, \partial^\alpha_{\beta - \beta_1} f), \partial^{\alpha}_{\beta} f w^2 (\alpha, \beta))_{L^2_{x, v}} |
&\le  \int_{\Omega}  C_k\Vert \partial^\alpha_{\beta - \beta_1} f \Vert_{L^2_{14}} \Vert \partial_{\beta_1} g w(\alpha, \beta)\Vert_{H^s_{ \gamma/2 
 }}\Vert \partial^{\alpha}_{\beta}  f  w(\alpha, \beta) \Vert_{H^s_{  \gamma/2}} 
\\
&
\quad+C_k\Vert \partial_{\beta_1} g \Vert_{L^2_{14}} \Vert \partial^\alpha_{\beta - \beta_1} f w(\alpha, \beta)\Vert_{H^s_{  \gamma/2 +2s }}\Vert \partial^{\alpha}_{\beta}  f w(\alpha, \beta) \Vert_{H^s_{ \gamma/2}} dx,
\end{aligned}
\end{equation*}
For $|\alpha_1|>0$, since $\partial^{\alpha_1} \mu =0$, we have
$( Q(\partial^{\alpha_1}_{\beta_1} G, \partial^{\alpha - \alpha_1}_{\beta -\beta_1} f), \partial^{\alpha}_\beta f \langle v \rangle^{2k})_{L^2_{x, v}} = ( Q(\partial^{\alpha_1}_{\beta_1} g, \partial^{\alpha - \alpha_1}_{\beta -\beta_1} f), \partial^{\alpha}_\beta f \langle v \rangle^{2k})_{L^2_{x, v}}$. By Lemma \ref{L26}, we have
\begin{align*}
\quad\,|( Q(\partial^{\alpha_1}_{\beta_1} g, \partial^{\alpha - \alpha_1}_{\beta -\beta_1} f), \partial^{\alpha}_\beta f \langle v \rangle^{2k})_{L^2_{x, v}}|
&\le  \int_{\Omega}  C_k\Vert \partial^{\alpha - \alpha_1}_{\beta - \beta_1} f \Vert_{L^2_{14}} \Vert \partial^{\alpha_1}_{\beta_1} g w(\alpha, \beta)\Vert_{H^s_{ \gamma/2 }}\Vert \partial^{\alpha}_{\beta}  f  w(\alpha, \beta) \Vert_{L^2_{  \gamma/2}} dx
\\
&\quad+C_k\int_{\Omega} \Vert \partial^{\alpha_1}_{\beta_1} g \Vert_{L^2_{14}} \Vert \partial^{\alpha - \alpha_1}_{\beta - \beta_1} f w(\alpha, \beta)\Vert_{H^s_{ \gamma/2+2s }}\Vert \partial^{\alpha}_{\beta}  f w(\alpha, \beta) \Vert_{H^s_{ \gamma/2}} dx. 
\end{align*}
Gathering  all the terms, we have
\begin{align}\label{210a}\notag
&\quad\,(\partial^\alpha_\beta Q ( \mu+g, f),  \partial^\alpha_\beta f w^2(\alpha, \beta) )_{L^2_{x, v}}\le  - \frac {1} {8} \Vert  b(\cos \theta) \sin^2 \frac \theta 2    \Vert_{L^1_\theta}\Vert \partial^\alpha_\beta f w(\alpha, \beta) \Vert_{L^2_xL^2_{\gamma/2, *}}^2\\& - \gamma_1\Vert \partial^\alpha_\beta f w(\alpha, \beta) \Vert_{L^2_xH^s_{\gamma/2}}^2 
\notag + C_{k}  \Vert \partial^\alpha_\beta f \Vert_{L^2_{x, v} }^2 +C_k \sum_{|\beta_1| < |\beta|} \Vert \partial^\alpha_{\beta_1}  f w(\alpha, \beta_1 ) \Vert_{L^2_x H^s_{\gamma/2}}\Vert \partial^\alpha_\beta f w(\alpha, \beta  ) \Vert_{L^2_xH^s_{\gamma/2}} 
\\
&\notag+C_k \int_{\Omega}  \sum_{\alpha_1 \le \alpha, \beta_1 \le \beta}\Vert \partial^{\alpha - \alpha_1}_{\beta - \beta_1} f \Vert_{L^2_{14}} \Vert \partial^{\alpha_1}_{\beta_1} g w(\alpha, \beta)\Vert_{H^s_{ \gamma/2 }}\Vert \partial^{\alpha}_{\beta}  f  w(\alpha, \beta) \Vert_{H^s_{  \gamma/2}} dx
\\
&\notag+C_k\int_{\Omega} \sum_{\alpha_1 \le \alpha, \beta_1 \le \beta, |\alpha_1| +|\beta_1| >0} \Vert \partial^{\alpha_1}_{\beta_1} g \Vert_{L^2_{14}} \Vert \partial^{\alpha - \alpha_1}_{\beta - \beta_1} f w(\alpha, \beta)\Vert_{H^s_{  \gamma/2+2s }}\Vert \partial^{\alpha}_{\beta}  f w(\alpha, \beta) \Vert_{H^s_{ \gamma/2}} dx
\\
&+C_k\int_{\Omega} \Vert  g \Vert_{L^2_{14}} \Vert \partial^{\alpha }_{\beta} f w(\alpha, \beta)\Vert_{H^s_{  \gamma/2 }}^2 dx.
\end{align}
Now we only need to prove that 
\begin{align}\label{210b}
\int_{\Omega} \Vert  g \Vert_{L^2_{14}} \Vert \partial^{\alpha }_{\beta} f w(\alpha, \beta)\Vert_{H^s_{  \gamma/2 }}^2
 dx \lesssim \Vert \langle v \rangle^{14} g \Vert_{H^2_{x, v} } \|f\|^2_{Y_k},
\end{align} 
and  for all  $\alpha_1 \le \alpha$, $\beta_1 \le \beta$, $|\alpha_1| +|\beta_1| >0$, 
\begin{align}\label{210c}
\int_{\Omega} \Vert \partial^{\alpha_1}_{\beta_1} g \Vert_{L^2_{14}} \Vert \partial^{\alpha - \alpha_1}_{\beta - \beta_1} f w(\alpha, \beta)\Vert_{H^s_{  \gamma/2+2s }}\Vert \partial^{\alpha}_{\beta}  f w(\alpha, \beta) \Vert_{H^s_{ \gamma/2}} dx  \lesssim \Vert \langle v \rangle^{14} g \Vert_{H^2_{x, v} }\|f\|^2_{Y_k}. 
\end{align}
The fifth term on the right-hand side of \eqref{210a} follows similarly by changing the order $f$ and $g$.
 First, for the case $|\alpha_1| = |\beta_1| =0 $,  we have
\begin{equation*}
\begin{aligned}
\int_{\Omega} \Vert  g \Vert_{L^2_{14}} \Vert \partial^{\alpha}_{\beta } f w(\alpha, \beta)\Vert_{H^s_{  \gamma/2 }}^2 dx 
&\lesssim \Vert  g \Vert_{L^\infty_xL^2_{14}} \Vert \partial^{\alpha  }_{\beta } f w(\alpha, \beta)\Vert_{L^2_x H^s_{  \gamma/2 }}^2\lesssim \Vert \langle v \rangle^{14} g \Vert_{H^2_{x, v} } \|f\|^2_{Y_k}.
\end{aligned}
\end{equation*}
This gives \eqref{210b}. 
For \eqref{210c}, 
we split it into two cases:  $|\alpha_1| + |\beta_1| =1 $ and  $|\alpha_1| + |\beta_1| = 2 $. For the case $|\alpha_1| + |\beta_1| =1 $,
by $\Vert fg\Vert_{L^2_x} \le \Vert  f \Vert_{L^6_x} \Vert g \Vert_{L^3_x}$,
we have
\begin{equation}
\label{long1}
\begin{aligned}
&\int_{\Omega} \Vert \partial^{\alpha_1}_{\beta_1} g \Vert_{L^2_{14}} \Vert \partial^{\alpha - \alpha_1}_{\beta - \beta_1} f w(\alpha, \beta)\Vert_{H^s_{  \gamma/2+2s }}\Vert \partial^{\alpha}_{\beta}  f w(\alpha, \beta) \Vert_{H^s_{ \gamma/2}} dx
\\\lesssim & \Vert \partial^{\alpha_1}_{\beta_1} g \Vert_{L^6_x L^2_{14}} \Vert \partial^{\alpha - \alpha_1}_{\beta - \beta_1} f w(\alpha, \beta)\Vert_{L^3_x H^s_{  \gamma/2+2s }}\Vert \partial^{\alpha}_{\beta}  f w(\alpha, \beta) \Vert_{L^2_x H^s_{ \gamma/2}} 
\\
\lesssim &\Vert \langle v \rangle^{14} g \Vert_{H^2_{x, v} } \Vert \partial^{\alpha - \alpha_1}_{\beta - \beta_1} f w(\alpha, \beta)\Vert_{L^3_x H^s_{  \gamma/2+2s }}  \|f \|_{Y_k}. 
\end{aligned}
\end{equation}
We again split it into two parts  $|\alpha|+|\beta| = 1$ and  $|\alpha|+|\beta| = 2$. For the case  $|\alpha_1| + |\beta_1| = |\alpha|+|\beta| =1$, we have $|\alpha - \alpha_1| =|\beta - \beta_1| =0 $. Then by \eqref{L3} and \eqref{l31}, we have
\begin{align*}
\Vert f w(\alpha, \beta)\Vert_{L^3_x H^s_{  \gamma/2+2s }}  \lesssim \Vert f w(0, 0)\Vert_{L^2_x H^s_{  \gamma/2 }}  + \Vert f w(1,0)\Vert_{H^1_x H^s_{  \gamma/2 }} \lesssim \|f \|_{Y_k}. 
\end{align*}
We then consider the case $|\alpha_1| +|\beta_1| =1$, $|\alpha| + |\beta| =2$.   This time we have $|\alpha - \alpha_1| =1$, $|\beta - \beta_1|=0$ or $|\alpha - \alpha_1| =0$, $|\beta - \beta_1| =1$. For the first case, by \eqref{l32}, 
\begin{align*}
\Vert \partial^{\alpha - \alpha_1} f w(\alpha, \beta)\Vert_{L^3_x H^s_{  \gamma/2+2s }}  \le  \Vert f w(1, 0)\Vert_{H^1_x H^s_{ \gamma/2 }} + \Vert f w( 2 , 0)\Vert_{H^2_x H^s_{ \gamma/2 }}  \lesssim \|f \|_{Y_k}.
\end{align*}
For the second case,  $|\beta| \ge 1$ and by \eqref{l32}, we obtain 
\begin{align*}
\Vert  \partial_{\beta - \beta_1}  f w(\alpha, \beta)\Vert_{L^3_x H^s_{ \gamma/2+2s }} \le \Vert f w(0, 1)\Vert_{L^2_x H^{1+s}_{ \gamma/2 }} + \Vert f  w( 1 , 1)\Vert_{H^1_x H^{1+s}_{ \gamma/2 }} \lesssim\|f \|_{Y_k}.
\end{align*}
For the case $|\alpha_1|+|\beta_1| = 2$, we have $|\alpha| +|\beta|=2$ and $|\alpha - \alpha_1| = |\beta - \beta_1|=0$. Thus 
\begin{equation}
\label{long2}
\begin{aligned}
&\quad\,\int_{\Omega} \Vert \partial^{\alpha_1}_{\beta_1} g \Vert_{L^2_{14}} \Vert \partial^{\alpha - \alpha_1}_{\beta - \beta_1} f w(\alpha, \beta)\Vert_{H^s_{  \gamma/2+2s }}\Vert \partial^{\alpha}_{\beta}  f w(\alpha, \beta) \Vert_{H^s_{ \gamma/2}} dx\\
&\lesssim  \Vert \partial^{\alpha}_{\beta} g \Vert_{L^2_x L^2_{14}} \Vert f w(\alpha, \beta)\Vert_{L^\infty_x H^s_{  \gamma/2+2s }}\Vert \partial^{\alpha}_{\beta}  f w(\alpha, \beta) \Vert_{L^2_x H^s_{ \gamma/2}} 
\lesssim \Vert \langle v \rangle^{14} g \Vert_{H^2_{x, v} } \Vert f w(\alpha, \beta)\Vert_{L^\infty_x H^s_{  \gamma/2+2s }}   \|f \|_{Y_k}.
\end{aligned}
\end{equation}
By  \eqref{Linfty} with suitable $k$ and \eqref{linfty2},  we have
\begin{align*}
\Vert  f w(\alpha, \beta)\Vert_{L^\infty_x H^s_{ \gamma/2  +2s}} 
&\le \Vert f w(1, 0)\Vert_{H^1_x H^s_{ \gamma/2 }} + \Vert f w(2 , 0)\Vert_{H^2_x H^s_{ \gamma/2 }}\lesssim \|f \|_{Y_k}.
\end{align*}
These estimates imply \eqref{210c}. Substituting \eqref{210b} and \eqref{210c} into \eqref{210a}, we obtain \eqref{236}. 
 For \eqref{236a}, by \eqref{213a} and \eqref{psi}, we have 
 \begin{align}\label{236b}\notag
 	&\quad\,\big| \big( \partial^\alpha_\beta Q (g,  f),\big(e^{\frac{\pm A_{\al,\beta}\psi}{\<v\>^2}}-1\big)\partial^{\alpha}_{\beta} h w^2 (\alpha, \beta)\big)_{L^2_{x, v}} \big|\\
 	&\notag\lesssim \sum_{\substack{\al_1\le\al\\ \beta_1\le\beta}}
 	\Big(\min\Big\{\big\|\|\pa^{\al_1}_{\beta_1}g\|_{L^2_{14}}\|\pa^{\al-\al_1}_{\beta-\beta_1}fw (\alpha, \beta)\|_{H^s_{\gamma/2}}\big\|_{L^2_x}
 	\|\big(e^{\frac{\pm A_{\al,\beta}\psi}{\<v\>^2}}-1\big)\<v\>^{2s}\partial^{\alpha}_{\beta} hw(\alpha, \beta)\|_{L^2_xH^s_{\gamma/2}},\\
 	&\notag\qquad\qquad\qquad+\big\|\|\pa^{\al_1}_{\beta_1}g\|_{L^2_{14}}\|\<v\>^{2s}\pa^{\al-\al_1}_{\beta-\beta_1}fw (\alpha, \beta)\|_{H^s_{\gamma/2}}\big\|_{L^2_x}
 	\|\big(e^{\frac{\pm A_{\al,\beta}\psi}{\<v\>^2}}-1\big)\partial^{\alpha}_{\beta} hw(\alpha, \beta)\|_{L^2_xH^s_{\gamma/2}}\Big\}\\
 	&\qquad\qquad\qquad
 	+\big\|\|\pa^{\al_1}_{\beta_1}f\|_{L^2_{14}}\|\pa^{\al-\al_1}_{\beta-\beta_1}gw (\alpha, \beta)\|_{H^s_{\gamma/2}}\big\|_{L^2_x}\|\big(e^{\frac{\pm A_{\al,\beta}\psi}{\<v\>^2}}-1\big)\partial^{\alpha}_{\beta} hw(\alpha, \beta)\|_{L^2_xH^s_{\gamma/2}}\Big).
 	\end{align}
 Applying $L^\infty-L^2$ and $L^3-L^6$ H\"{o}lder's inequality and using the first term in the minimum, the first right-hand term of \eqref{236b} can be estimated as 
 \begin{align}\label{al1}\notag
 	&\quad\, \Big(\sum_{|\al_1|+|\beta_1|=0}\|\pa^{\al_1}_{\beta_1}g\|_{L^\infty_xL^2_{14}}\|\pa^{\al-\al_1}_{\beta-\beta_1}fw (\alpha, \beta)\|_{L^2H^s_{\gamma/2}}
 	+\sum_{|\al_1|+|\beta_1|=1}\|\pa^{\al_1}_{\beta_1}g\|_{L^3_xL^2_{14}}\|\pa^{\al-\al_1}_{\beta-\beta_1}fw (\alpha, \beta)\|_{L^6_xH^s_{\gamma/2}}
 	\\&\notag\qquad +\sum_{|\al_1|+|\beta_1|=2}\|\pa^{\al_1}_{\beta_1}g\|_{L^2_xL^2_{14}}\|\pa^{\al-\al_1}_{\beta-\beta_1}fw (\alpha, \beta)\|_{L^\infty_xH^s_{\gamma/2}} \Big)
 	A_{\al,\beta}\|\na_x\psi\|_{H^1_x}\|\partial^{\alpha}_{\beta} hw(\alpha, \beta)\|_{L^2_xH^s_{\gamma/2}}\\
 	&\lesssim \|\<v\>^{14}g\|_{H^2_{x,v}}\|\na_x\psi\|_{H^1_x}\|f\|_{Y_k}\|\partial^{\alpha}_{\beta} hw(\alpha, \beta)\|_{L^2_xH^s_{\gamma/2}}. 
 \end{align}
 The second right-hand term of \eqref{236b} can be estimated similarly and then we deduce that 
 \begin{align*}
 \quad\,\big| \big( \partial^\alpha_\beta Q (g,  f),\big(e^{\frac{\pm A_{\al,\beta}\psi}{\<v\>^2}}-1\big)\partial^{\alpha}_{\beta} h w^2 (\alpha, \beta)\big)_{L^2_{x, v}} \big|
 &\lesssim \|\<v\>^{14}g\|_{H^2_{x,v}}\|\na_x\psi\|_{H^1_x}\|f\|_{Y_k}\|\partial^{\alpha}_{\beta} hw(\alpha, \beta)\|_{L^2_xH^s_{\gamma/2}}\\
 	&\quad+\|\<v\>^{14}f\|_{H^2_{x,v}}\|\na_x\psi\|_{H^1_x}\|g\|_{Y_k}\|\partial^{\alpha}_{\beta} hw(\alpha, \beta)\|_{L^2_xH^s_{\gamma/2}}.
 \end{align*}
 The proof of \eqref{236c} is similar by replacing the term $\big(e^{\frac{\pm A_{\al,\beta}\psi}{\<v\>^2}}-1\big)$ by $1$ and $e^{\frac{\pm A_{\al,\beta}\psi}{\<v\>^2}}$. Note that we keep the minimum in \eqref{236b} this time, and the proof is omitted for brevity. 
Thus Lemma \ref{L210} is proved. 
\end{proof}
We can also prove a similar result for the Landau case.

\begin{lem}\label{L58}
Let $-3\le\gamma\le 1$, $Q$ be the Landau collision operator and $k \ge 7$. There exists a constant $C_k>0$ such that for any function $f, g, h$, and any function $\psi$ satisfies \eqref{214c}, we have the following.

(1) For any $|\alpha| + |\beta| \le 2$, 
\begin{equation}\label{211a}
\begin{aligned} 
(\partial^\alpha_\beta Q ( \mu, f),  \partial^\alpha_\beta f w^2(\alpha, \beta) ) &\le  - \gamma_1 \Vert \partial^\alpha_\beta  f w(\alpha, \beta)\Vert_{L^2_x L^2_{D}}^2 + C_{k}  \Vert \partial^\alpha_\beta f \Vert_{L^2_x L^2_v}^2
\\
&\quad + C_k \sum_{\beta_1 < \beta} \Vert \partial^\alpha_{\beta_1}  f w(\alpha, \beta_1) \Vert_{L^2_x L^2_{D}}\Vert \partial^\alpha_\beta g w( \alpha, \beta)\Vert_{L^2_x L^2_{D}} .
\end{aligned}
\end{equation}

(2) Remind the norm $\|\cdot\|_{Y_k}$ is given in \eqref{D}. For any $|\alpha| + |\beta| \le 2$,  we have  
\begin{align}\label{211b}
\big| \big( \partial^\alpha_\beta Q (g,  f), \partial^\alpha_{\beta} f w^2 (\alpha, \beta)\big)_{L^2_{x, v}} \big| \le C_k \Vert g \langle v \rangle^7 \Vert_{H^2_{x, v}} \|f\|_{Y_k}^2,
\end{align}
\begin{align}\label{211c}
	\big| \big( \partial^\alpha_\beta Q (g,  f),\big(e^{\frac{\pm A_{\al,\beta}\psi}{\<v\>^2}}-1\big)\partial^{\alpha}_{\beta} h w^2 (\alpha, \beta)\big)_{L^2_{x, v}} \big| \le C_k \|g\<v\>^7\|_{H^2_{x,v}}\|\na_x\psi\|_{H^1_x}\|f\|_{Y_k}\|h\|_{Y_k},
\end{align}
and 
\begin{multline}
	\label{211d}
	\big| \big( \partial^\alpha_\beta Q (g,  f),\partial^{\alpha}_{\beta} h w^2 (\alpha, \beta)\big)_{L^2_{x, v}} \big|+\big| \big( \partial^\alpha_\beta Q (g,  f),e^{\frac{\pm A_{\al,\beta}\psi}{\<v\>^2}}\partial^{\alpha}_{\beta} h w^2 (\alpha, \beta)\big)_{L^2_{x, v}} \big|
	\\\le C_k \|g\<v\>^7\|_{H^2_{x,v}}\min\{\|f\|_{Y_k}\|\<v\>h\|_{Y_k},\|\<v\>f\|_{Y_k}\|h\|_{Y_k}\}. 
\end{multline}

\end{lem}
\begin{proof}
Notice that $(\partial^\alpha_\beta Q(\mu, f), \partial^\alpha f w^2(\alpha, \beta))_{L^2_{x, v}}= \sum_{\alpha_1 \le \alpha, \beta_1 \le \beta} ( Q(\partial_{\beta_1} \mu, \partial^{\alpha }_{\beta -\beta_1} f), \partial^{\alpha}_\beta f w^2( \alpha, \beta))_{L^2_{x, v}}$. We split it into two cases: $\beta_1 = 0$ and $|\beta_1|>0$. If $\beta_1  =0$, after integrating about $x$ in \eqref{26b}, we have
\begin{align*}
( Q( \mu, \partial^{\alpha }_\beta f), \partial^{\alpha}_\beta f w^2(\alpha, \beta))_{L^2_{x, v}}  \le  -\gamma_1 \Vert \partial^\alpha_\beta f w(\alpha, \beta)\Vert_{L^2_x L^2_{D}}^2 + C_{k}  \Vert \partial^\alpha_\beta f  \Vert_{L^2_x L^2_v }^2.
\end{align*}
For the case $|\beta_1| >0$, by \eqref{210},
$
w( \alpha, \beta) \langle v \rangle \le w( \alpha, \beta - \beta_1).
$ Then by Lemma  \ref{L52} we have
\begin{equation*}
\begin{aligned} 
|( Q(\partial_{\beta_1} \mu, \partial^\alpha_{\beta - \beta_1} f), \partial^{\alpha}_{\beta} f w^2 (\alpha, \beta))_{L^2_{x, v}} |&\le  C_k \Vert \partial^\alpha_{\beta - \beta_1}  f w(\alpha, \beta) \Vert_{L^2_x L^2_{D,1}}\Vert \partial^\alpha_\beta g w(\alpha, \beta) \Vert_{L^2_x L^2_{D}} 
\\
&\le   C_k \Vert \partial^\alpha_{\beta - \beta_1}  f w(\alpha,\beta-\beta_1 ) \Vert_{L^2_x L^2_{D}}  \Vert \partial^\alpha_\beta g w(\alpha, \beta) \Vert_{L^2_x L^2_{D}},
\end{aligned}
\end{equation*}
So \eqref{211a} is thus finished by gathering the above two terms together. Next, for \eqref{211b}, notice that
\begin{align*}
( \partial^\alpha_{\beta } Q( g,  f), \partial^\alpha_{\beta} f w^2 ( \alpha, \beta))_{L^2_{x, v}}   =\sum_{\alpha_1 \le \alpha, \beta_1 \le \beta} ( Q(  \partial^{\alpha_1}_{\beta_1 }  g,  \partial^{\alpha -\alpha_1}_{\beta  -\beta_1}  f), \partial^\alpha_{\beta} f w^2 (\alpha, \beta))_{L^2_{x, v}}.
\end{align*}
We again split it into two cases: $|\alpha_1| = |\beta_1| =0 $ and $|\alpha_1| +|\beta_1| >0$. For the case $|\alpha_1| = |\beta_1| =0 $, by \eqref{27} and \eqref{psi}, we have
\begin{equation*}
\begin{aligned}
|( Q( g, \partial^\alpha_{\beta } f), \partial^{\alpha}_{\beta} f w^2 (\alpha, \beta))_{L^2_{x, v}} |  
&\le   \int_{\Omega} C_k\Vert g \Vert_{L^2_7 }\Vert \big(e^{\frac{\pm A_{\al,\beta}\psi}{\<v\>^2}}-1\big)\<v\> \partial^\alpha_{\beta} f w(\alpha, \beta) \Vert_{L^2_{D}}^2 dx 
\\
&\le  C_k \Vert g \Vert_{L^\infty_x L^2_7 }\|\na_x\psi\|_{H^1_x} \Vert  \partial^\alpha_{\beta } f w( \alpha, \beta) \Vert_{L^2_x L^2_{D}}^2 
\le  C_k \Vert g \langle v \rangle^7 \Vert_{H^2_{x, v}}\|\na_x\psi\|_{H^1_x} \|f\|^2_{Y_k}.
\end{aligned}
\end{equation*}
For  $\alpha_1 \le \alpha$, $\beta_1 \le \beta$, $|\alpha_1| +|\beta_1| >0$, by \eqref{27}, we have
\begin{multline*}
|( Q(\partial^{\alpha_1}_{\beta_1} g, \partial^{\alpha -\alpha_1}_{\beta - \beta_1} f), \partial^{\alpha}_{\beta} f w^2 (\alpha, \beta))_{L^2_{x, v}} |\le  \int_{\Omega}  C_k\Vert \partial^{\alpha_1}_{\beta_1} g \Vert_{L^2_{7}} \Vert \partial^{\alpha -\alpha_1}_{\beta - \beta_1} f w( \alpha, \beta)\Vert_{L^2_{D, 1}}\Vert \partial^{\alpha}_{\beta}  f w(\alpha, \beta) \Vert_{L^2_{D }} dx.
\end{multline*}
We split it into two cases,  $|\alpha_1| + |\beta_1| =1 $ and  $|\alpha_1| + |\beta_1| = 2 $. For the case $|\alpha_1| + |\beta_1| =1 $,
by $\Vert fg\Vert_{L^2_x} \le \Vert  f \Vert_{L^6_x} \Vert g \Vert_{L^3_x}$ and similarly as \eqref{long1}, we have
\begin{multline*}
\int_{\Omega} \Vert \partial^{\alpha_1}_{\beta_1} g \Vert_{L^2_{7}} \Vert \partial^{\alpha - \alpha_1}_{\beta - \beta_1} f w(\alpha, \beta)\Vert_{L^2_{D,1}}\Vert \partial^{\alpha}_{\beta}  f w( \alpha, \beta) \Vert_{L^2_{D}} dx\lesssim \Vert \langle v \rangle^{7} g \Vert_{H^2_{x, v} } \Vert \partial^{\alpha - \alpha_1}_{\beta - \beta_1} f w( \alpha, \beta)\Vert_{L^3_x L^2_{D,1}}   \|f \|_{Y_k}.
\end{multline*}
 We again split it into two parts  $|\alpha|+|\beta| = 1$ and  $|\alpha|+|\beta| = 2$. For the case  $|\alpha_1| + |\beta_1| = |\alpha|+|\beta| =1$, we have $|\alpha - \alpha_1| =|\beta - \beta_1| =0 $. By \eqref{l31} and \eqref{L3} we obtain  that
\begin{align*}
\Vert f w( \alpha, \beta)\Vert_{L^3_x L^2_{D,1}}  \lesssim \Vert f w(0, 0)\Vert_{L^2_x L^2_{D}}  + \Vert f w(1,0)\Vert_{H^1_x L^2_{D}} \lesssim  \|f \|_{Y_k}.
\end{align*}
For the case $|\alpha| + |\beta| =2$,  either $|\alpha - \alpha_1| =1, |\beta - \beta_1|=0$ or $|\alpha - \alpha_1| =0, |\beta - \beta_1| =1$. For the first case, by \eqref{l32} and \eqref{L3}, we have
\begin{align*}
\Vert \partial^{\alpha - \alpha_1} f w( \alpha, \beta)\Vert_{L^3_x L^2_{D,1}}  \le  \Vert f w( 1, 0)\Vert_{H^1_x L^2_{D}} + \Vert f w( 2 , 0)\Vert_{H^2_x L^2_{D}}  \lesssim \|f \|_{Y_k}.
\end{align*}
For the second case, this time $|\beta| \ge 1$. By \eqref{l32} and \eqref{L3}, we have
\begin{align*}
\Vert  \partial_{\beta - \beta_1}  f w(\alpha, \beta)\Vert_{L^3_x L^2_{D,1}} \le \Vert \nabla_v f w(0, 1)\Vert_{L^2_x L^2_{D}} + \Vert \nabla_v f  w( 1 , 1)\Vert_{H^1_x L^2_{D}} \lesssim  \|f \|_{Y_k}.
\end{align*}
For the case $|\alpha_1|+|\beta_1| = 2$, we have obviously $|\alpha| +|\beta|=2$ and $|\alpha - \alpha_1| = |\beta - \beta_1|=0$. Similar to \eqref{long2}, we have
\begin{equation*}
\int_{\Omega} \Vert \partial^{\alpha_1}_{\beta_1} g \Vert_{L^2_{7}} \Vert \partial^{\alpha - \alpha_1}_{\beta - \beta_1} f w(\alpha, \beta)\Vert_{L^2_{D,1}}\Vert \partial^{\alpha}_{\beta}  f w(\alpha, \beta) \Vert_{L^2_{D}} dx
\lesssim \Vert \langle v \rangle^{7} g \Vert_{H^2_{x, v} } \Vert f w(\alpha, \beta)\Vert_{L^\infty_x L^2_{D,1}}    \|f \|_{Y_k}.
\end{equation*}
We obtain again from \eqref{Linfty} and \eqref{linfty2} that
\begin{align*}
\Vert  f w(\alpha, \beta)\Vert_{L^\infty_x L^2_{D,1}} \le \Vert f w(1, 0)\Vert_{H^1_x L^2_{D}} + \Vert f w(2 , 0)\Vert_{H^2_x L^2_{D}}\lesssim \Vert f \Vert_{Y_k}.
\end{align*}
The proof of \eqref{211b} is done after gathering the above terms.
For \eqref{211c}, 
by \eqref{27} and \eqref{psi}, we have 
\begin{align}\label{252}\notag
	&\quad\,\big| \big( \partial^\alpha_\beta Q (g,  f),\big(e^{\frac{\pm A_{\al,\beta}\psi}{\<v\>^2}}-1\big)\partial^{\alpha}_{\beta} h w^2 (\alpha, \beta)\big)_{L^2_{x, v}} \big|\\
	&\lesssim \sum_{\al_1\le\al,\,\beta\le\beta_1}\int_{\Omega}\|\pa^{\al_1}_{\beta_1}g\|_{L^2_7}\min\big\{\|\pa^{\al-\al_1}_{\beta-\beta_1}fw (\alpha, \beta)\|_{L^2_{D}}\|\big(e^{\frac{\pm A_{\al,\beta}\psi}{\<v\>^2}}-1\big)\<v\>\partial^{\alpha}_{\beta} hw(\alpha, \beta)\|_{L^2_{D}},
	\\&\notag\qquad\qquad\qquad\qquad\qquad\qquad\qquad\|\<v\>\pa^{\al-\al_1}_{\beta-\beta_1}fw (\alpha, \beta)\|_{L^2_{D}}\|\big(e^{\frac{\pm A_{\al,\beta}\psi}{\<v\>^2}}-1\big)\partial^{\alpha}_{\beta} hw(\alpha, \beta)\|_{L^2_{D}}\big\}\\
	&\notag\lesssim \Big(\sum_{|\al_1|+|\beta_1|=0}\|\pa^{\al_1}_{\beta_1}g\|_{L^\infty_xL^2_7}\|\pa^{\al-\al_1}_{\beta-\beta_1}fw (\alpha, \beta)\|_{L^2_xL^2_{D}}
		\notag\qquad +\sum_{|\al_1|+|\beta_1|=1}\|\pa^{\al_1}_{\beta_1}g\|_{L^3_xL^2_7}\|\pa^{\al-\al_1}_{\beta-\beta_1}fw (\alpha, \beta)\|_{L^6_xL^2_{D}}
			\\&\notag\qquad +\sum_{|\al_1|+|\beta_1|=2}\|\pa^{\al_1}_{\beta_1}g\|_{L^2_xL^2_7}\|\pa^{\al-\al_1}_{\beta-\beta_1}fw (\alpha, \beta)\|_{L^\infty_xL^2_{D}} \Big)
		A_{\al,\beta}\|\na_x\psi\|_{H^1_x}\|\partial^{\alpha}_{\beta} hw(\alpha, \beta)\|_{L^2_xL^2_{D}}\\
	&\notag\lesssim \|\<v\>^7g\|_{H^2_{x,v}}\|\na_x\psi\|_{H^1_x}\|f\|_{Y_k}\|h\|_{Y_k}, 
\end{align}
where we apply $L^\infty-L^2$ and $L^3-L^6$ H\"{o}lder's inequality. 
The proof of \eqref{211d} is similar by replacing the term $\big(e^{\frac{\pm A_{\al,\beta}\psi}{\<v\>^2}}-1\big)$ by $1$ and $e^{\frac{\pm A_{\al,\beta}\psi}{\<v\>^2}}$. Note that this time we keep the minimum in \eqref{252}, and we omit the proof for brevity. 
This completes the proof of Lemma \ref{L58}. 
\end{proof}

For the local existence, similar to \cite{AMSY} we have
\begin{thm}\label{T32}
Suppose that $-3<\gamma\le 1$ for Boltzmann case and $-3\le\gamma\le 1$ for Landau case. For any $k\ge k_0+2$, there exist small constants $\varepsilon_0,\tau_0, T_0>0$, such that if $f_0\in X_{k_0}$ satisfies $\mu+f_0\ge 0$ and
\begin{align}\label{312}
\E_{k_0}(0)\le\varepsilon_0,\quad \E_{k}(0)\le A,
\end{align}
for some $A>0$, 
Then the Cauchy problem
\begin{equation}
	\label{Cauchy1}
	\left\{
	\begin{aligned}
		&\partial_tf_\pm + v\cdot\nabla_xf_\pm \mp\nabla_x\phi\cdot\nabla_vf_\pm \pm \nabla_x\phi\cdot v\mu = Q(f_\pm+f_\mp,\mu)+ Q(\mu+f_\pm,f_\pm) +Q(\mu+f_\mp,f_\pm),\\
		&-\Delta_x\phi=\int_{\mathbb R^3}{(f_+-f_-)dv},\quad\int_{\Omega}{\phi(x) dx}=0, \quad f(0) = f_0,\quad E(0) = E_0,\\
		&f(t,x,R_xv) = f(t,x,v) \text{ on } \gamma_-\text{ and }
		\pa_n\phi = 0 \text{ on }\pa\Omega, \text{ if }\Omega\text{ is given by \eqref{Omega}},
	\end{aligned}\right.
\end{equation}
admits a unique weak solution $f\in L^{\infty}_{T_0}X_k$ satisfying $\mu+f  \ge 0$ and 
\begin{align}\label{315a}
	\sup_{0\le t\le T_0}\E_{k_0}(t)\le \tau_0,\quad \sup_{0\le t\le T_0}\E_{k}(t)\le A. 
\end{align}
\end{thm}

\section{Global regularity}\label{sec4}
In this section, we prove the Theorem \ref{globaldecay}. We first establish the macroscopic estimates. 
\subsection{Macroscopic estimate}
We recall the operator $\P$ defined in \eqref{P}. 
It is direct to obtain
$\Vert a \Vert_{L^2_x}^2 + \Vert b \Vert_{L^2_x}^2 + \Vert c \Vert_{L^2_x}^2 \sim \Vert \P f \Vert_{L^2_{x, v}}^2$. We also rewrite the equation \eqref{vplocal1} as
\begin{equation}
\label{rewrite}
\partial_tf_\pm + v\cdot\nabla_xf_\pm  \pm \nabla_x\phi\cdot v\mu -L_\pm f = N_\pm(f),\quad
-\Delta_x\phi=\int_{\mathbb R^3}{(f_+-f_-)dv},
\end{equation}
with initial data $f(0) = f_0$, $ E(0) = E_0$,
where $L=[L_+,L_-]$ and $N=[N_+,N_-]$ are given by 
\begin{equation}\label{nlp}
	\begin{aligned}
		L_\pm f :&= Q(f_\pm+f_\mp,\mu) + 2Q(\mu,f_\pm),\quad
		N_\pm(f):=\pm\nabla_x\phi\cdot\nabla_vf_\pm + Q(f_\pm,f_\pm)+Q(f_\mp,f_\pm). 
	\end{aligned}
\end{equation}
Our next goal is to estimate $a(t, x), b(t, x),c(t, x)$ in terms of $\{\I-\P\} f$.
In contrast to  \cite{G4}, our $P$ is not symmetric and we can not compare the $v_i, v_j$  terms on both sides.
The following Lemma gives the macroscopic estimates.

\begin{lem}\label{L420}
	For both Landau and Boltzmann case, 
suppose $f$ solves \eqref{rewrite}. For any integer $m\ge 0$, there exists function $G=G(t)$ satisfying
\begin{equation}
	\label{ing}
	G(t)\lesssim\sum_{|\alpha|\le 1}\|\partial^{\alpha}f\|_{L^2_xL^2_{10}}  \|\partial^{\alpha}\nabla_x \P f\|_{L^2_{x, v}},
\end{equation}
such that 
\begin{multline}
\label{uni4}
	\pa_tG(t)+\lambda\|[a_+,a_-,b,c]\|^2_{H^2_x} +\lambda \|\na_x\phi\|^2_{H^2_x}
	\lesssim 
	\|\{\I-\P\}f\|^2_{H^2_xL^2_{10}} +   \|N_{\|}\|_{H^1_xL^2_v}^2 +\|\na_x\phi\|_{L^2_x}^4,
\end{multline}
Here $N_{\|}$ is the inner product of $N(f)$ with some linear combination of $(1,v_i,  v_i v_j, v_i^2, v_i|v|^2)$ over $v\in\R^3$. 
\end{lem}

\begin{proof}
Here we only consider the non-cutoff Boltzmann case and the Landau case is similar. The proof follows the idea in \cite[Section 6]{G4}.  Splitting $f=\P f+\{\I-\P\}f$, we rewrite \eqref{rewrite} to be 
\begin{multline}\label{47}
\big(\partial_t a_\pm + \partial_t b \cdot v + \partial_t c (|v|^2-3)\big)\mu+v\cdot\nabla_x\big( a_\pm +
b \cdot v+  c (|v|^2-3)\big)\mu\pm\nabla_x\phi \cdot v\mu
\\
=-(\partial_t+v\cdot\nabla_x)\{\II-\PP\}f- L_\pm\{\I-\P\}f+N_\pm(f). 
\end{multline}
Taking the  inner product of \eqref{47} with $1, v_i^2, |v|^2, v_iv_j$, $i, j =1, 2, 3$, $i \neq j$ over $v\in\mathbb R^3$, we have
\begin{equation}
\label{uni55}
\begin{aligned}
\partial_t a_\pm+\nabla_x \cdot  b&=(-(\partial_t+v\cdot\nabla_x)\{\II-\PP\}f- L_\pm\{\I-\P\}f+N_\pm(f), 1)_{L^2_v},
\\
\partial_t a_\pm +2 \partial_t c + 2\partial_{x_i} b_i  +\nabla_x \cdot b&=(-(\partial_t+v\cdot\nabla_x)\{\II-\PP\}f- L_\pm\{\I-\P\}f+N_\pm(f), v_i^2)_{L^2_v},
\\
3\partial_t a_\pm + 6\partial_t c + 5\nabla_x \cdot  b&=(-(\partial_t+v\cdot\nabla_x)\{\II-\PP\}f- L_\pm\{\I-\P\}f+N_\pm(f), |v|^2)_{L^2_v}.
\end{aligned}
\end{equation}
and 
\begin{equation}\label{49}
	\partial_{x_j} b_i  +  \partial_{x_i} b_j=(-(\partial_t+v\cdot\nabla_x)\{\II-\PP\}f- L_\pm\{\I-\P\}f, v_iv_j)_{L^2_v}+(N_\pm(f), v_iv_j)_{L^2_v}
	=:\gamma_{1ij}+\gamma_{2ij}, \quad i\ne j.
\end{equation}
By taking the subtraction of \eqref{uni55}$_2$ and \eqref{uni55}$_1$, we have 
\begin{align}\notag
\label{abc0}
\partial_t  c+  \partial_{x_i} b_i&=\big(-(\partial_t+v\cdot\nabla_x)\{\II-\PP\}f- L_\pm\{\I-\P\}f, \frac{ |v_i|^2-1}  {  2  }\big)_{L^2_v}+\big(N_\pm(f), \frac{ |v_i|^2-1}  {  2  }\big)_{L^2_v}\\&=:\gamma_{1i}+\gamma_{2i}.
\end{align}
Taking the inner product of \eqref{47} with $v$, $v|v|^2$ over $v\in\mathbb R^3$, we have
\begin{equation}
\label{uni5new}
\begin{aligned}
\partial_t b+ \nabla_x (a_\pm\pm\phi)+2\nabla_x c&=(-(\partial_t+v\cdot\nabla_x)\{\II-\PP\}f- L_\pm\{\I-\P\}f+N_\pm(f), v)_{L^2_v},
\\
5\partial_t b+ 5\nabla_x (a_\pm\pm\phi)+20\nabla_x c&= (-(\partial_t+v\cdot\nabla_x)\{\II-\PP\}f- L_\pm\{\I-\P\}f+N_\pm(f), v|v|^2)_{L^2_v}.
\end{aligned}
\end{equation}
The two identities in \eqref{uni5new} implies 
\begin{equation}\label{abc1}
\nabla_x c=  \big(-(\partial_t+v\cdot\nabla_x)\{\II-\PP\}f- L_\pm\{\I-\P\}f+N_\pm(f), \frac{v(|v|^2-5)}{10}\big)_{L^2_v},
\end{equation}
and
\begin{equation}\label{abc2}
\nabla_x a_\pm=- \partial_t b\mp\nabla_x\phi+ \big(-(\partial_t+v\cdot\nabla_x)\{\II-\PP\}f- L_\pm\{\I-\P\}f+N_\pm(f), \frac{v(10-|v|^2)}{5}\big)_{L^2_v}. 
\end{equation}
Here we use the fact that
$\int_{\mathbb R^3}{|v|^2\mu dv}=3, \quad\int_{\mathbb R^3}{|v|^4\mu dv}=15, \quad\int_{\mathbb R^3}{|v|^6\mu dv}=105$. By the definitions \eqref{P} of $\P f$ and \eqref{nlp} of $L$, we have for $\psi(v) =1, v_i, |v|^2$  that 
\begin{align*}
(\{\II-\PP\}f, \psi(v) ) =0,\,\, (L_\pm\{\I-\P\}f, \psi(v)) =0, \,\,(\partial_t \{\II-\PP\} f, \psi(v) ) =0,\,\, (v \cdot \nabla_x \{\II-\PP\}f, 1) =0. 
\end{align*}
Thus the inequalities \eqref{uni55}$_1$, \eqref{uni55}$_3$ and \eqref{uni5new}$_1$ 
(i.e. inner product with $1, v_i, |v|^2$) become 
\begin{equation}
\label{uni5}
\begin{aligned}
\partial_t a_\pm+\nabla_x \cdot  b&= (N_\pm(f), 1)_{L^2_v},
\\
 \partial_t b+ \nabla_x a_\pm+2\nabla_x c\pm\nabla_x\phi&=-( v\cdot\nabla_x\{\II-\PP\}f, v  )_{L^2_v}+ (N_\pm(f), v)_{L^2_v},
\\
3 \partial_t a_\pm + 6  \partial_t c+ 5\nabla_x \cdot b&= -( v\cdot \nabla_x\{\II-\PP\} f, |v|^2 )_{L^2_v}+(N_\pm(f), |v|^2)_{L^2_v}. 
\end{aligned}
\end{equation}
Combining \eqref{uni5}$_1$ and \eqref{uni5}$_3$, we have 
\begin{equation}
\label{uni8}
\partial_t c+\frac{1}{3}\nabla_x \cdot  b=-\frac{1}{6}(v\cdot\nabla_x\{\II-\PP\}f,|v|^2)_{L^2_v}+\Big( N_\pm(f),\frac{|v|^2-3}{6} \Big)_{L^2_v}. 
\end{equation}
For brevity, we define 
\begin{align}\label{xi}
\xi_{a}=\frac{v(10-|v|^2)}{5}, \quad  \xi_{bi}=\frac{2v_i^2-5}{2}, \quad \xi_{c}= \frac{v(|v|^2-5)}{10}.
\end{align}
 It follows that
\begin{multline*}
|(v \cdot \nabla_x \{\II-\PP\}f,  \xi_{a})| +|(v \cdot \nabla_x \{\II-\PP\}f,  \xi_{bi})| +|(v \cdot \nabla_x \{\II-\PP\}f,  v_iv_j)|\\ +|(v \cdot \nabla_x \{\II-\PP\}f,  \xi_{c})| \le \Vert  \nabla_x \{\II-\PP\} f\Vert_{L^2_6}. 
\end{multline*}
For the Boltzmann case, by Lemma \ref{L21}, we have
\begin{align*}
|(Q(\mu, f), \xi_{a})| 
&\le  \Vert \mu \Vert_{L^2_{10}} \Vert f \Vert_{L^2_{10}} \Vert \xi_{a}\Vert_{H^{2s}_{-7}} \le C\Vert f \Vert_{L^2_{10}} ,\quad
 |(Q(f, \mu), \xi_{a})| &\le  \Vert f \Vert_{L^2_{10}} \Vert \mu \Vert_{L^2_{10}} \Vert \xi_{a}\Vert_{H^{2s}_{-7}} \le C\Vert f \Vert_{L^2_{10}}.
\end{align*}
For the Landau case, by Lemma \ref{L55}, we have
\begin{align*}
|(Q(\mu, f), \xi_{a})| &\le  \Vert \mu \Vert_{L^2_{10}} \Vert f \Vert_{L^2_{10}} \Vert \xi_{a}\Vert_{H^{2}_{-5}} \le C\Vert f \Vert_{L^2_{10}} , \quad |(Q(f, \mu), \xi_{a})| \le  \Vert f \Vert_{L^2_{10}} \Vert \mu \Vert_{L^2_{10}} \Vert \xi_{a}\Vert_{H^{2}_{-5}} \le C\Vert f \Vert_{L^2_{10}}. 
\end{align*}
Similar arguments can be carried on inner product with $\xi_{bi}$, $v_iv_j$, $\xi_c$. 
Recalling the definition \eqref{nlp} of $L=[L_+,L_-]$, we have 
\begin{equation}
\label{ineqi}
|( Lf,\xi_{a})_{L^2_v}|+ |( Lf,\xi_{bi})_{L^2_v}| + |( Lf, v_iv_j)_{L^2_v}| + |( Lf,\xi_{c})_{L^2_v}| \lesssim \Vert f \Vert_{L^2_{10}}. 
\end{equation}
\textbf{Step 1: Computation of $\nabla_x\partial^{\alpha}c$.}
Let $|\al|\le 2$. Applying $\partial^\alpha$ on both side of \eqref{abc1} and taking inner product with $\nabla_x\partial^\alpha c$ over $x\in\Omega$, we deduce that
\begin{align}
	\notag
\label{uni11}
\|\nabla_x\partial^\alpha c\|_{L^2_x}^2
&\notag\le - \frac{d}{dt}\int_{\Omega}{(\{\II-\PP\}\partial^\alpha f, \xi_{c})_{L^2_v}\cdot\nabla_x\partial^\alpha c\,dx}\notag-\int_{\Omega}{( \{\II-\PP\}\nabla_x\partial^\alpha f, \xi_{c})_{L^2_v}\cdot\partial_t\partial^\alpha c\,dx}
\\
&+C\|\nabla_x\partial^\alpha c\|_{L^2_x}\big(\|\partial^{\alpha}\{\II-\PP\}f\|_{H^1_xL^2_{10}}+\|\partial^{\alpha}N_{\|}\|_{L^2_{x}} \big).
\end{align}
Here, for the torus case, one can take integration by parts directly. For the union of cubes case, we have from boundary \eqref{abc} that $\pa_{x_i}c=0$ on $\Gamma_i$. On the other hand, when $\al_i=0,2$, we have from Lemma \ref{highspecular} and definition \eqref{xi} of $\xi_c$ that for $x\in\Gamma_i$, 
\begin{align*}
	(\{\II-\PP\}\partial^\alpha f, \xi_{c})_{L^2_v}
	&= (\{\II-\PP\}\partial^\alpha f(R_xv), \xi_{c}(R_xv))_{L^2_v}= -(\{\II-\PP\}\partial^\alpha f(v), \xi_{c}(v))_{L^2_v}=0. 
\end{align*}
Note that $\partial^\alpha\PP f(R_xv)=\partial^\alpha\PP f(v)$ by using \eqref{boundaryabc}. This implies the integration by parts about $\na_x$ in \eqref{uni11}. 
For the term $\pa_t\pa^\al c$ in \eqref{uni11}, it follows from \eqref{uni8} that 
\begin{align*}
\partial_t\partial^\alpha c=\frac{1}{3}\nabla_x\partial^{\alpha}b-\frac{1}{6}( v\cdot\nabla_x\{\II-\PP\}\partial^{\alpha}f,|v|^2)_{L^2_v}+\Big( \partial^{\alpha}N_\pm(f),\frac{|v|^2-3}{6}\Big)_{L^2_v}.
\end{align*}
By Cauchy-Schwarz inequality, for any small $\varepsilon>0$,  we have
\[
-\int_{\Omega}{( \{\II-\PP\}\nabla_x\partial^\alpha f, \xi_{c})_{L^2_v}\cdot\partial_t\partial^\alpha c\,dx}
\le {\varepsilon}\|\nabla_x\partial^\alpha b\|^2+\frac{C}{\varepsilon}\|\nabla_x\partial^\alpha\{\II-\PP\}f\|^2_{L^2_xL^2_{10}}+\varepsilon\|\partial^\alpha N_{\|}\|_{L^2_{x}}^2.
\]
Plugging this into \eqref{uni11}, we deduce that
\begin{multline}\label{uni11a}
\|\nabla_x\partial^\alpha c\|^2\le -\frac{d}{dt}\int_{\Omega}{(\{\II-\PP\}\partial^\alpha f, \xi_{c})_{L^2_v}\cdot\nabla_x\partial^\alpha c\,dx}+ \varepsilon\|\nabla_x\partial^\alpha b\|^2+ \varepsilon\|\nabla_x\partial^\alpha c\|^2\\
+\frac{C}{\varepsilon}\big(
\|\partial^{\alpha}\{\II-\PP\}f\|^2_{H^1_xL^2_{10}}+\|\partial^{\alpha}N_{\|}\|_{L^2_{x}} ^2\big).
\end{multline}
\textbf{Step 2: Computation of $\nabla_x\partial^{\alpha}b$.} For fixed $i$, we use \eqref{49}, \eqref{abc0} to compute
\begin{align}
\label{uni14}\notag
\Delta_x  b_i&=\sum_{j\ne i}\partial_{x_j x_j}  b_i+\partial_{x_ix_i} b_i
\notag=\sum_{j\ne i}\left(-\partial_{x_ix_j} b_j+\partial_{x_j}( \gamma_{1ij}+ \gamma_{2ij})\right)+\partial_{x_i} (\gamma_{1i}+\gamma_{2i})-\partial_t\partial_{x_i} c
\\
&\notag=\sum_{j\ne i}\left(\partial_t\partial_{x_i} c-\partial_{x_i} ( \gamma_{1j}+ \gamma_{2j})\right)+\sum_{j\ne i}\partial_{x_j}  ( \gamma_{1ij}+ \gamma_{2ij})+\partial_{x_i} (\gamma_{1i}+\gamma_{2i})-\partial_t\partial_{x_i} c
\\
&=\sum_{j\ne i}\left(\partial_{x_j} ( \gamma_{1ij}+ \gamma_{2ij})-\partial_{x_i} (\gamma_{1j}+\gamma_{2j})  \right)-\partial_{x_ix_i}   b_i+2\partial_{x_i}(\gamma_{1i}+\gamma_{2i}).
\end{align}
We can rewrite the linear terms $\sum_{j\ne i}\left(\partial_{x_j}\gamma_{1ij}-\partial_{x_i} \gamma_{1j}  \right)+2\partial_{x_i}\gamma_{1i}$ including $\gamma_{1i}$, $\gamma_{1j}$, $\gamma_{1ij}$ as the linear combination of $\partial_{x_j}\left(-((\partial_t+v\cdot\nabla_x)\{\II-\PP\}f,\xi_{ij})_{L^2_v}-(L_\pm\{\I-\P\}f,\xi_{ij})_{L^2_v}\right)$, where $\xi_{ij}$ are certain linear combinations of the basis $\{1, v_iv_j, |v_i|^2\}_{i,j=1}^3$. Similar to \eqref{ineqi}, we have
$
|(Lf, \xi_{ij})_{L^2_v}|\lesssim \|f\|_{L^2_{10}}.
$
Note  that  $-(\Delta_x\partial^\alpha  b_i,\partial^\alpha b_i)_{L^2_x} =\|\nabla_x\partial^\alpha b_i\|_{L^2_x}^2$ and $(\pa_{x_ix_i}\pa^\al b_i,\pa^\al b_i)_{L^2_x} = -\|\pa_{x_i}\pa^\al b_i\|_{L^2_x}^2, $ 
which is trivial for the torus case and follows from \eqref{boundaryabc} for the union of cubes case. 
Applying $\partial^\alpha$ on both side of \eqref{uni14}, taking inner product with $\partial^\alpha b_i$ over $x\in\Omega$ and summation over $i=1,2,3$, we deduce that
\begin{align}\label{420}\notag
\|\nabla_x\partial^\alpha b\|_{L^2_x}^2&=\sum_{i}    \Big(-\int_{\Omega}{\sum_{j\ne i}\left(\partial_{x_j}\partial^\alpha(\gamma_{1ij}+\gamma_{2ij})-\partial_{x_i}\partial^\alpha(\gamma_{1j}+\gamma_{2j})\right)\cdot\partial^\alpha b_i\,dx}\Big)\\
&\notag\quad- \sum_{i}\Big(\int_{\Omega}{|\partial_{x_i}\partial^\alpha b_{i}|^2dx}-2\int_{\Omega}{\partial_{x_i}\partial^\alpha( \gamma_{1i}+ \gamma_{2i})\cdot\partial^\alpha b_i\,dx}\Big)
\\&\notag\le\sum_{i,j}\frac{d}{dt}\int_{\Omega}{( \{\II-\PP\}\partial_{x_j}\partial^\alpha f, \xi_{ij})_{L^2_v}\cdot\partial^\alpha b_i\,dx}\notag\quad-\sum_{i,j}\int_{\Omega}{( \{\II-\PP\}\partial_{x_j}\partial^\alpha f, \xi_{ij})_{L^2_v}\cdot\partial_t\partial^\alpha b_i\,dx}
\\
&\quad+C\|\nabla_x\partial^\alpha b\|\big(
\|\partial^{\alpha}\{\II-\PP\}f\|_{H^1_xL^2_{10 }}+\|\partial^{\alpha}N_{\|}\|_{L^2_{x}}\big).
\end{align}
Here we used integration by parts on $\pa_{x_i}$ and $\pa_{x_j}$. For the case of the union of cubes, we need the following boundary values to complete the integration by parts. 
For $i=1,2,3$ and $\al=(\al_1,\al_2,\al_3)$. If $\al_i=0$ or $2$, then it follows from \eqref{boundaryabc} that $\pa^\al b_i=0$ on $\Gamma_i$. 
If $\al_i=1$, then applying Lemma \ref{highspecular} and change of variable $v\mapsto R_xv$, we can obtain from \eqref{abc0} that $\pa^\al\gamma_{1j} = 0,\ \text{ on }\Gamma_i\text{ for }j=1,2,3$. 
Similarly, for $j\neq i$, if $\al_j=1$, then $\pa^\al b_i=0$ on $\Gamma_j$. If $\al_j=0$ or $2$, then it follows from \eqref{49} that $\pa^\al \gamma_{1ij} = 0,\ \text{ on }\Gamma_j\text{ for }j\neq i$.

Next we calculate the term $\partial_t\partial^\alpha b_i$. 
Applying $\pa^\al$ to \eqref{uni5}$_2$ yields  
\begin{align*}
\partial_t\partial^\alpha b_i=-\nabla_{x_i}\partial^{\alpha}a_\pm-2\nabla_{x_i}\partial^{\alpha}c
\mp\nabla_{x_i}\partial^{\alpha}\phi
-(v\cdot\nabla_x\{\II-\PP\}\partial^{\alpha}f +\partial^{\alpha}N_\pm(f),v_i)_{L^2_v}.
\end{align*}
Then by Cauchy-Schwarz inequality, we have 
\begin{multline*}
-\int_{\Omega}{( \{\II-\PP\}\partial_{x_j}\partial^\alpha f, \xi_{ij})_{L^2_v}\cdot\partial_t\partial^\alpha b_i\,dx}\le {\varepsilon}\big(\|\nabla_x\partial^\alpha a_\pm\|^2_{L^2_x}+\|\nabla_x\partial^\alpha c\|^2_{L^2_x}+\|\nabla_x\partial^{\alpha}\phi\|^2_{L^2_x}\big)
\\+{\varepsilon}\|\partial^\alpha N_{\|}\|_{L^2_{x}}^2
+\frac{C}{\varepsilon}\|\nabla_x\partial^\alpha\{\II-\PP\}f\|^2_{L^2_xL^2_{10}},
\end{multline*}
and \eqref{420} becomes 
\begin{multline}\label{uni15}
\|\nabla_x\partial^\alpha b\|^2\le- \sum_{i,j}\frac{d}{dt}\int_{\Omega}{(\{\II-\PP\}   \partial^\alpha f, \xi_{ij})_{L^2_v}\cdot\partial_{x_j}\partial^\alpha b_i\,dx} +\varepsilon\|\nabla_x \partial^{\alpha}\phi\|^2_{L^2_x}\\
+ \varepsilon\|\nabla_x\partial^\alpha a_\pm\|^2_{L^2_x}+\ve\|\nabla_x\partial^\alpha c\|^2_{L^2_x}+\frac{C}{\varepsilon}\big(
\|\partial^{\alpha}\{\II-\PP\}f\|^2_{H^1_xL^2_{10}}+\|\partial^{\alpha}N_{\|}\|_{L^2_{x}}^2\big),
\end{multline}
where we choose $\ve$ small enough. \\

\textbf{Step 3: Computation of $\nabla_x\partial^{\alpha}a$.} Taking the addition and difference of \eqref{abc2} over $\pm$, it yields that 
\begin{align}\label{422}
	\nabla_x (a_++a_-)=- 2\partial_t b+\sum_\pm \big(-(\partial_t+v\cdot\nabla_x)\{\II-\PP\}f- L_\pm\{\I-\P\}f+N_\pm(f), \xi_a\big)_{L^2_v},
\end{align}
and 
\begin{align}\label{423}
	\nabla_x (a_+-a_-)=-2\nabla_x\phi+ \sum_\pm\pm\big(-(\partial_t+v\cdot\nabla_x)\{\II-\PP\}f- L_\pm\{\I-\P\}f+N_\pm(f), \xi_a\big)_{L^2_v}.
\end{align}
Applying $\partial^\alpha$ derivative on both side of   \eqref{422} and taking inner product with $\nabla_x \partial^\alpha (a_++a_-)$ over $\Omega$, we deduce that
\begin{multline}
\label{uni16}
\|\nabla_x\partial^\alpha (a_++a_-)\|_{L^2_x}^2= 
-\int_{\Omega}\partial_t\Big(\partial^\alpha b+
(\{\II-\PP\}\partial^\alpha f,\xi_{a})_{L^2_v}\Big)\cdot  \nabla_x  \partial^\alpha   (a_++a_-)\,dx
\\
- \int_{\Omega}{\big( v \cdot \nabla_x\{\II-\PP\}\partial^\alpha f+L_\pm\{\I-\P\}\partial^\alpha f+\partial^\alpha N_\pm(f),\xi_{a}\big)_{L^2_v}\cdot   \nabla_x  \partial^\alpha (a_++a_-)\,dx}.
\end{multline}
By integration by parts about $\pa_t$ and $\na_x$, we have 
\begin{align}\label{424}\notag
&\quad\,-\int_{\Omega}\partial_t\Big(\partial^\alpha b
	 +(\{\II-\PP\}\partial^\alpha f,\xi_{a})_{L^2_v}\Big)\cdot \nabla_x  \partial^\alpha   (a_++a_-)\,dx\notag\\&\notag=-\frac{d}{dt}\int_{\Omega}\Big(\partial^\alpha b
	+(\{\II-\PP\}\partial^\alpha f,\xi_{a})_{L^2_v}\Big)\cdot \nabla_x  \partial^\alpha  (a_++a_-)\,dx
\\
&\quad\,-\int_{\Omega}{    \nabla_x\cdot\Big(  \partial^\alpha b+
(\{\II-\PP\}\partial^\alpha f,\xi_{a})_{L^2_v}\Big)\,\partial_t \partial^\alpha (a_++a_-)\,dx}. 
\end{align}
To apply the integration by parts for the union of cubes case, we write the following boundary values by using \eqref{boundaryabc}. 
If $\al_i=1$, then $\pa^\al(a_++a_-)=0$ on $\Gamma_i$. If $\al_i=0,2$, then $\pa^\al b_i=0$ on $\Gamma_i$ and by change of variable $v\mapsto R_xv$ and the definition \eqref{xi} of $\xi_a$, we have 
\begin{align*}
	(\{\II-\PP\}\partial^\alpha f,\xi_{a})_{L^2_v}&=(\{\II-\PP\}\partial^\alpha f(R_xv),\xi_{a}(R_xv))_{L^2_v}=-(\{\II-\PP\}\partial^\alpha f,\xi_{a})_{L^2_v}=0. 
\end{align*}
The time derivative and tangent derivatives don't affect the zero boundary values. This completes the integration by parts in \eqref{424}. 
In view of \eqref{424} and \eqref{uni16}, applying \eqref{uni5}$_1$, \eqref{ineqi} and Cauchy-Schwarz inequality, we have 
\begin{multline}
\label{uni18}
\|\nabla_x\partial^\alpha (a_++a_-)\|^2
+\frac{d}{dt}\sum_\pm\int_{\Omega}\Big(\partial^\alpha b
+(\{\II-\PP\}\partial^\alpha f,\xi_{a})_{L^2_v}\Big)\cdot \nabla_x  \partial^\alpha  (a_++a_-)\,dx
\\
\lesssim 
 \|\partial^{\alpha}\{\II-\PP\}f\|_{H^1_xL^2_{10 }} 
+ 
 \|\partial^\alpha N_{\|}\|_{L^2_{x}}^2  +   \|\nabla_x\partial^\alpha b\|_{L^2_x}^2.  
\end{multline}
Similarly, noticing that $\big(\na_x\pa^\al\phi,\na_x\pa^\al(a_+-a_-)\big)_{L^2_x}
	= -\big(\Delta_x\pa^\al\phi,\pa^\al(a_+-a_-)\big)_{L^2_x} = \|\Delta_x\pa^\al\phi\|_{L^2_x}^2$,
we apply $\pa^\al$ to \eqref{423} and take inner product with $\na_x\pa^\al(a_+-a_-)$ over $\Omega$ to deduce that 
\begin{multline}\label{426}
	\|\na_x\pa^\al(a_+-a_-)\|_{L^2_x}^2 + 2\|\Delta_x\pa^\al\phi\|_{L^2_x}^2 
	+ \frac{d}{dt}\int_{\Omega}\sum_\pm\mp(\{\II-\PP\}\partial^\alpha f,\xi_{a})_{L^2_v}\cdot \nabla_x  \partial^\alpha  (a_+-a_-)\,dx\\
	\lesssim 
	\|\partial^{\alpha}\{\II-\PP\}f\|_{H^1_xL^2_{10 }} 
	+ 
	\|\partial^\alpha N_{\|}\|_{L^2_{x}}^2  +   \|\nabla_x\partial^\alpha b\|_{L^2_x}^2.  
\end{multline}
Notice that $|a_++a_-|^2 +  |a_+-a_-|^2=2|a_+|^2+2|a_-|^2.$ For the torus case, by Sobolev inequality, we have 
\begin{align}\label{phibound}
\|\nabla_x \phi\|^2\lesssim \|\na_x^2 \phi\|^2 =\|\Delta_x\phi\|_{L^2_x}^2. 
\end{align}
For the zeroth order, we apply the Poincar\'e inequality to $a_\pm, b, c$ in $x$ and use the conservation law \eqref{conse2.torus} to deduce that 
\begin{align}\label{abcc}\notag
	\Vert a_\pm \Vert_{L^2_{x}}^2 &\le C \left(\int_{\Omega} a_\pm(t, x) dx \right)^2 + C\Vert \nabla_x  a_\pm \Vert_{L^2_{x}}^2\lesssim \|\na_xa_\pm\|_{L^2_x}^2,\\
	\Vert b \Vert_{L^2_{x}}^2 &\le C\left( \int_{\Omega} b(t, x) dx \right)^2+ C \Vert \nabla_x  b \Vert_{L^2_{x}}^2\lesssim \|\na_xb\|_{L^2_x}^2 ,\\
	\Vert c \Vert_{L^2_{x}}^2 &\notag\le C \left ( \int_{\Omega} c(t, x) dx \right)^2   +C \Vert \nabla_x  c \Vert_{L^2_{x}}^2\lesssim \|\na_x\phi\|_{L^2_x}^4+\|\na_xc\|_{L^2_x}^2. 
\end{align}
For the union of cubes case, it follows from \eqref{Neumann} that $\pa_{x_i}\phi=0$ on $\Gamma_i$. Then by Sobolev embedding (cf. \cite[Theorem 6.7-5]{Ciarlet2013}), we also have \eqref{phibound}. The second equality in \eqref{phibound} can be obtained by using integration by parts. 
For the zeroth order, the conservation laws \eqref{conservation_bounded} give the conservation on mass and energy and hence, the estimates for $a_\pm,c$ in \eqref{abcc} still hold. For the estimate on $b$, we have from \eqref{boundaryabc} that $b_i=0$ on $\Gamma_i$ for $i=1,2,3$. Then by Sobolev embedding (cf. \cite[Theorem 6.7-5]{Ciarlet2013}), we obtain $\|b\|_{L^2_x}\lesssim \|\na_xb\|_{L^2_x}$. 
In view of these facts, we take linear combination $\lam\times\eqref{426}+\lam\times\eqref{uni18}+\eqref{uni11a}+\eqref{uni15}$ and summation over $|\al|\le 1$ and $\pm$ with sufficiently small $\lam,\ve>0$ to deduce that 
\begin{align*}
\pa_tG(t)+\lambda\sum_{|\al|\le 2}\|\pa^\al[a_+,a_-,b,c]\|^2_{L^2_x} +\lambda \|\na_x\phi\|^2_{H^2_x}
\lesssim 
  \|\{\I-\P\}f\|^2_{H^2_xL^2_{10}} +   \|N_{\|}\|_{H^1_x}^2 +\|\na_x\phi\|_{L^2_x}^4,
\end{align*}
where 
\begin{align*}
	G &= \sum_\pm\sum_{|\al|\le 1}\Big(\int_{\Omega}{(\{\II-\PP\}\partial^\alpha f, \xi_{c})_{L^2_v}\cdot\nabla_x\partial^\alpha c\,dx}
	 + \sum_{i,j}\int_{\Omega}{(\{\II-\PP\}  \partial^\alpha f, \xi_{ij})_{L^2_v}\cdot\partial_{x_j}\partial^\alpha b_i\,dx}\\
	&+\lam\int_{\Omega}\Big(\partial^\alpha b	+(\{\II-\PP\}\partial^\alpha f,\xi_{a})_{L^2_v}\Big)\cdot \nabla_x  \partial^\alpha  (a_++a_-)\,dx
	\\
	&\mp\lam\int_{\Omega}(\{\II-\PP\}\partial^\alpha f,\xi_{a})_{L^2_v}\cdot \nabla_x  \partial^\alpha  (a_+-a_-)\,dx\Big). 
\end{align*}
This implies \eqref{uni4} and the estimate \eqref{ing} can be directly deduced by Cauchy-Schwarz inequality.
\end{proof}

\subsection{Estimates on the energy} 
In this subsection, we give the energy estimates on \eqref{vplocal1}. 
\begin{lem}\label{L43}
	Let $T>0$, $\eta>0$ and $k\ge k_0$. 
	Suppose $[f,\phi]$ is the solution to equation \eqref{vplocal1} satisfying $\mu+f\ge 0$ and 
	$\E(t)\le C$.Then there exists functional $\E_1(t)$ satisfying 
	\begin{align}
		\label{E111}\E_1(t) \approx \Vert   f  \Vert_{X_k }^2+\|\na_x\phi\|^2_{H^3_x}, 
	\end{align}
such that, for any $T, \eta>0$ and  $k \ge k_0$, we have
	\begin{multline}\label{eq1}
\pa_t\E_1(t)+\lam\eta\|f\|^2_{Y_k} +\lam\sum_{|\al|\le2} \Vert  w(\alpha,0) \partial^\alpha\{\I-\P\}  f\Vert_{L^2_xL^2_{D}}^2
\lesssim \|f\|_{X_{k_0}}\|f\|_{X_k}^2\\+C_{\eta,k}  \Vert  \{\I-\P\} f \Vert_{H^2_xL^2_v}^2+ \eta\|\na_x\phi\|^2_{H^2_x}+\eta\|\P f\|_{H^2_{x}L^2_v}^2+C(\sqrt{\E_{k_0}(t)}+\E_{k_0}(t))\D_k(t). 
	\end{multline}
for some generic constant $\lam>0$ and any sufficiently small $\eta>0$. 
\end{lem}

\begin{proof}
%
We proceed by considering spatial derivatives and mixed derivatives. Notice that 
$\|\phi\|_{L^\infty_x}\lesssim \|\na_x\phi\|_{H^1_x}\lesssim \E(t)\le C$,
which verify \eqref{214c} and we can apply estimates in Section \ref{secpri}. 

\smallskip
\noindent{\bf Step 1. Estimates with mixed derivatives.} 
 Recall that $A_{\al,\beta}$ is given in \eqref{Cab1}. 
For any $|\al|+|\beta|\le 2$, we apply $\pa^\al_\beta$ to \eqref{vplocal1},  and obtain 
\begin{multline}\label{35}
	\quad\,\partial_t\pa^\al_\beta f_\pm + \pa^\al_\beta( v\cdot\nabla_x f_\pm) \mp\pa^\al_\beta(\nabla_x\phi\cdot\nabla_vf_\pm) \pm \pa^\al_\beta(\nabla_x\phi\cdot v\mu) \\
	= \pa^\al_\beta Q(f_\pm+f_\mp,\mu)+\pa^\al_\beta Q(\mu+f_\pm, f_\pm) +\pa^\al_\beta Q(\mu+f_\mp, f_\pm). 
\end{multline}
Thus, taking inner product of \eqref{35} with $e^{\frac{\pm A_{\al,\beta}\phi}{\<v\>^2}}w^2(\al,\beta)\pa^\al_\beta  f_\pm$ over $\Omega\times\R^3$, one has 
\begin{align}\label{36}\notag
	&\frac{1}{2}\partial_t\|e^{\frac{\pm A_{\al,\beta}\phi}{2 \<v\>^2}}\pa^\al_\beta f_\pm\|^2_{L^2_{x,v}}+ \big(A_{\al,\beta}\pa_t\phi\<v\>^{-2}\pa^\al_\beta f_\pm,e^{\frac{\pm A_{\al,\beta}\phi}{\<v\>^2}}w^2(\al,\beta)\pa^\al_\beta f_\pm\big)_{L^2_{x,v}}\\
	&\notag+ \big(v\cdot\nabla_x\pa^\al_\beta f_\pm,e^{\frac{\pm A_{\al,\beta}\phi}{\<v\>^2}}w^2(\al,\beta)\pa^\al_\beta f_\pm\big)_{L^2_{x,v}}+\sum_{|\beta_1|=1}( \pa_{\beta_1}v \cdot \nabla_x\partial^\alpha_{\beta-\beta_1} f_\pm ,e^{\frac{\pm A_{\al,\beta}\phi}{\<v\>^2}}w^2(\al,\beta)\partial^\alpha_{\beta}f_\pm    )_{L^2_{x, v}}\\
	&\notag\notag\mp\big(\nabla_x\phi\cdot\nabla_v\pa^\al_\beta f_\pm, e^{\frac{\pm A_{\al,\beta}\phi}{\<v\>^2}}w^2(\al,\beta)\pa^\al_\beta f_\pm\big)_{L^2_{x,v}}\mp\sum_{\substack{\al_1\le \al\\|\al_1|\ge 1}}\big(\pa^{\al_1}\nabla_x\phi\cdot\nabla_v\pa^{\al-\al_1}_\beta f_\pm,e^{\frac{\pm A_{\al,\beta}\phi}{\<v\>^2}}w^2(\al,\beta)\pa^\al_\beta f_\pm\big)_{L^2_{x,v}}\notag\\
	&\notag\pm \big(\pa^\al_\beta(\nabla_x\phi\cdot v\mu),e^{\frac{\pm A_{\al,\beta}\phi}{\<v\>^2}}w^2(\al,\beta)\pa^\al_\beta f_\pm\big)_{L^2_{x,v}}= \big(\pa^\al_\beta Q(f_\pm+f_\mp,\mu),e^{\frac{\pm A_{\al,\beta}\phi}{\<v\>^2}}w^2(\al,\beta)\pa^\al_\beta f_\pm\big)_{L^2_{x,v}}\notag\\
	&\notag+\big(\pa^\al_\beta Q(\mu+f_\pm,f_\pm) + \pa^\al_\beta Q(\mu+f_\mp,f_\pm),w^2(\al,\beta)\pa^\al_\beta f_\pm\big)_{L^2_{x,v}}
	  \\&+\big(\pa^\al_\beta Q(\mu+f_\pm,f_\pm) + \pa^\al_\beta Q(\mu+f_\mp,f_\pm),\big(e^{\frac{\pm A_{\al,\beta}\phi}{\<v\>^2}}-1\big)w^2(\al,\beta)\pa^\al_\beta f_\pm\big)_{L^2_{x,v}}.
\end{align}

	We denote the second to tenth terms in \eqref{36} by $I_1$ to $I_9$ and estimate them one by one. 
	For $I_1$, 
	we make a rough estimate: 
	\begin{align}\label{I1}
		|I_1|\lesssim \|\pa_t\phi\|_{L^\infty_x}\|f\|_{X_k}^2. 
	\end{align}
Notice that $\|\phi\|_{L^\infty_x}\lesssim \|\na_x\phi\|_{H^1_x}\le C$, which verify \eqref{214c} and we can apply estimates in Section \ref{secpri}. By \eqref{214a}, we deduce that 
$|I_2+I_4|\lesssim \|\na_x\phi\|_{H^2_x}\|f\|^2_{Y_k}$. Next, after applying \eqref{estiv3}, when $|\beta|\ge 1$, we have 
\begin{align*}
	|I_3|\lesssim C_\eta\sum_{\substack{|\beta_2|=|\beta|-1 \\ |\alpha_2| =|\alpha|+1}} \Vert \partial^{\alpha_2}_{\beta_2}  f w(\alpha_2, \beta_2)\Vert_{L^2_xL^2_{\gamma/2}}^2+\eta\Vert \partial^\alpha_{\beta} f w(\alpha, \beta) \Vert_{L^2_{x}L^2_{\gamma/2} }^2,
\end{align*}
for any $\eta>0$. By \eqref{estiv5}, we have $|I_5|\lesssim \|\na_x\phi\|_{H^3_x}\|f\|^2_{Y_k}$. 
When $|\beta|=0$, using \eqref{231}, we have 
\begin{align*}
	\sum_\pm I_6&\ge C_\al\pa_t\|\partial^\alpha\na_x\phi\|_{L^2_x}^2- C\|\pa^\al\na_x\phi\|_{L^2_x}\|\pa^\al\{\I-\P\}f\|_{L^2_xL^2_5}-C\|\partial^\alpha \nabla_x \phi\|_{L^2_x}\|\na_x\psi\|_{H^1_x}\|f\|_{Y_k}\\
	&\ge C_\al\pa_t\|\partial^\alpha\na_x\phi\|_{L^2_x}^2- \eta\|\pa^\al\na_x\phi\|^2_{L^2_x}-\eta\|w(\al,0)\pa^\al f\|^2_{L^2_xL^2_{\gamma/2}}\\&\quad -C_\eta\|\pa^\al f\|^2_{L^2_{x,v}}-C\|\partial^\alpha \nabla_x \phi\|_{L^2_x}\|\na_x\phi\|_{H^1_x}\|f\|_{Y_k}, 
\end{align*}
for any $\eta>0$. 
When $|\beta|\ge 1$, we have from \eqref{estiv4} that 
\begin{align}\label{I6}
	\sum_\pm |I_6|\le C_k\Vert \na_x\phi \Vert_{H^{2}_x }   \Vert f \Vert_{H^{2}_x L^2_{v}  }.
\end{align}
For the Boltzmann case and for convenience, we define  
\begin{align}\label{Ybark}
\Vert f \Vert_{\bar{Y}_k} := \sum_{|\alpha| + |\beta| \le  2} C_{|\alpha|, |\beta|} 
\Vert \partial^{\alpha}_{\beta} f w (\alpha, \beta) \Vert_{L^2_xL^2_{\gamma/2, *}}^2.
\end{align}
For $I_7$, since $\theta \in [0, \frac {\pi } 2]$,  we have $\sin^2 \frac \theta 2 \le \frac 1 2 $ and hence if $k \ge 14$ we have
\begin{align}\label{39a}
       \Vert b(\cos\theta) \sin^{k-2 } \frac \theta 2 \Vert_{L^1_\theta} \|f\|^2_{\bar Y_k}\le \frac {1} {16} \Vert  b(\cos \theta) \sin^2 \frac \theta 2  \Vert_{L^1_\theta}\|f\|^2_{\bar Y_k}. 
\end{align}
 For Boltzmann case, applying Lemma \ref{L29}, we have for any $\eta>0$ that 
\begin{align*}
	\sum_\pm |I_7|&\le \Vert b(\cos\theta) \sin^{k-2 } \frac \theta 2 \Vert_{L^1_\theta}   \Vert \partial^\alpha_\beta f w(\alpha, \beta) \Vert_{L^2_x L^2_{\gamma/2, *}}\Vert \partial^\alpha_\beta f  w(\alpha, \beta) \Vert_{L^2_x L^2_{\gamma/2, *}}
	\\
	&\quad + \eta \|f\|^2_{\bar Y_k }+C_\eta\|f\|_{H^2_{x, v}}^2
+ C_k \sum_{\beta_1 < \beta} \Vert \partial^{\alpha}_{\beta_1} f w(\alpha, \beta_1) \Vert_{L^2_x L^2_{\gamma/2}}\Vert \partial^\alpha_\beta f w(\alpha, \beta) \Vert_{L^2_x L^2_{\gamma/2  }}\\
	&\le\frac {1} {16} \Vert  b(\cos \theta) \sin^2 \frac \theta 2  \Vert_{L^1_\theta} \Vert \partial^\alpha_\beta f w(\alpha, \beta) \Vert_{L^2_x L^2_{\gamma/2, *}}^2+ \eta \|f\|^2_{\bar Y_k }  +C_\eta\|f\|_{H^2_{x, v}}^2
	\\
	&\quad+ C_k \sum_{\beta_1 < \beta} \Vert \partial^{\alpha}_{\beta_1} f w(\alpha, \beta_1) \Vert_{L^2_x L^2_{\gamma/2}}\Vert \partial^\alpha_\beta f w(\alpha, \beta) \Vert_{L^2_x L^2_{\gamma/2  }}
.
\end{align*}
For the Landau case, it follows from Lemma \ref{L29} that 
\begin{align*}
	\sum_\pm |I_7|&\le C_k \sum_{|\alpha| +|\beta| \le 2} \Vert \partial^\alpha_\beta   f \Vert_{L^2_{x} L^2_5}^2   \le \eta\|f\|^2_{Y_k}   +  C_\eta\|f\|^2_{H^2_{x,v}}, 
\end{align*}
For $I_8$, for Boltzmann case, we deduce from \eqref{236} that 
\begin{align*}
\sum_\pm I_8  &\le  - \frac {1} {8} \Vert  b(\cos \theta) \sin^2 \frac \theta 2  \Vert_{L^1_\theta}\Vert \partial^\alpha_\beta  f w(\alpha, \beta)    \Vert_{L^2_x L^2_{\gamma/2, * }}^2  -\gamma_1 \Vert  \partial^\alpha_\beta  f w(\alpha, \beta) \Vert_{L^2_xH^s_{\gamma/2}}^2
+C_k\Vert \langle v \rangle^{14} f \Vert_{H^2_{x, v }} \|f\|^2_{Y_k}
\\
&\quad + C_{k}  \Vert \partial^\alpha_\beta f \Vert_{L^2_{x, v}}^2+ C_k \sum_{\beta_1 < \beta} \Vert \partial^\alpha_{\beta_1}  f w(\alpha, \beta_1) \Vert_{H^s_{\gamma/2}}\Vert \partial^\alpha_\beta f w(\alpha, \beta)\Vert_{H^s_{\gamma/2}},
\end{align*}
For Landau case, note that  and from \eqref{211a} and \eqref{211b}, 
\begin{align*}
\sum_\pm I_8 &\le  - \gamma_1 \Vert \partial^\alpha_\beta  f w(\alpha, \beta)\Vert_{L^2_x L^2_{D}}^2 + C_{k}  \Vert \partial^\alpha_\beta f \Vert_{L^2_x L^2_v}^2
\\
&\quad + C_k \sum_{\beta_1 < \beta} \Vert \partial^\alpha_{\beta_1}  f w(\alpha, \beta_1) \Vert_{L^2_x L^2_{D}}\Vert \partial^\alpha_\beta f w( \alpha, \beta)\Vert_{L^2_x L^2_{D}} + C_k \Vert f \langle v \rangle^7 \Vert_{H^2_{x, v}} \|f\|_{Y_k}^2,
\end{align*}
for Landau case, with some constants $\gamma_1,C_k>0$. 
For $I_9$, for Boltzmann case, we have from \eqref{236a} that 
\begin{align*}
	|I_9|\le C_k \big(\|\<v\>^{14}f\|_{H^2_{x, v}}+C\big)\|\na_x\phi\|_{H^1_x}\|f\|_{Y_k}^2
	+\|\<v\>^{14}f\|_{H^2_{x,v}}\|\na_x\phi\|_{H^1_x}\big(\|f\|_{Y_k}+C\big)\|f\|_{Y_k},
\end{align*}
 and for Landau case, from \eqref{211c} we deduce
\begin{align*}
	|I_9|\le C_k \big(\|\<v\>^7f\|_{H^2_{x,v}}+C\big)\|\na_x\phi\|_{H^1_x}\|f\|_{Y_k}^2.
\end{align*}

Combining the above estimates, we take summation $\sum_{|\al|+|\beta|\le 2,\,\pm}C_{|\al|,|\beta|}\times\eqref{36}$ with $0<\eta\ll 1$ and constants $C_{|\al|,|\beta|}$ satisfying \eqref{Cal} to deduce that for the Boltzmann case: 
\begin{align}\label{37}\notag
	&\quad\,\frac{1}{2}\partial_t\|f\|_{X_k}^2 +\pa_t\sum_{|\al|\le 2}C_\al C_{|\al|,0}\|\partial^\alpha\na_x\phi\|_{L^2_x}^2 
	+\frac {1} {8} \Vert  b(\cos \theta) \sin^2 \frac \theta 2  \Vert_{L^1_\theta}\|f\|^2_{\bar Y_k}
		+\gamma_1 \|f\|_{Y_k}^2
	\\
	&\notag\le \frac {1} {16} \Vert  b(\cos \theta) \sin^2 \frac \theta 2  \Vert_{L^1_\theta} \|f\|^2_{\bar Y_k} 
	+\eta\|   f   \|_{Y_k}^2+C_{\eta, k}\|f\|^2_{H^2_{x, v}}\\
	&\notag\quad+C\|\pa_t\phi\|_{L^\infty_x}\|f\|_{X_k}^2 +
	C_k\Big(\|\na_x\phi\|_{H^3_x}\|f\|^2_{Y_k} 
	+\|\na_x\phi\|^2_{H^2_x}
	+\Vert f \Vert_{X_{14}} \|f\|^2_{Y_k}
	\\
	&\quad+ \big(\|  f \|_{X_{14}}+C\big)\|\na_x\phi\|_{H^1_x}\|f\|_{Y_k}^2
	+\|f\|_{X_{14}}\|\na_x\phi\|_{H^1_x}\big(\|  f   \|_{Y_k}+C\big)\|f\|_{Y_k}\Big),
\end{align}
and for the Landau case:
\begin{align}\notag\label{37a}
	&\quad\,\frac{1}{2}\partial_t\|f\|_{X_k}^2 +\pa_t\sum_{|\al|\le 2}C_\al C_{|\al|,0}\|\partial^\alpha\na_x\phi\|_{L^2_x}^2 
	+\gamma_1 \|f\|_{Y_k}^2
	\\
	&\notag\le \eta\|f\|_{Y_k}^2+C_{\eta,k}\|f\|^2_{H^2_{x,v}}+C\|\pa_t\phi\|_{L^\infty_x}\|f\|_{X_k}^2+
	C_k\Big(\|\na_x\phi\|_{H^3_x}\|f\|^2_{Y_k} +\|\na_x\phi\|^2_{H^2_x}\\
	&
	+\Vert f \Vert_{X_7} \|f\|^2_{Y_k}
	+ \big(\|f\|_{X_7}+C\big)\|\na_x\phi\|_{H^1_x}\|f\|_{Y_k}^2
	+\|f\|_{X_7}\|\na_x\phi\|_{H^1_x}\big(\|f\|_{Y_k}+C\big)\|f\|_{Y_k}\Big),
\end{align}
for some generic constant $\gamma_1$ independent of $k$, 
where we let $k\ge 14$.
Next notice that 
\begin{align}\label{39}
\|\na_x^3\phi\|_{L^2_x}^2 &\lesssim \|\na_x\Delta_x\phi\|^2_{L^2_x} \lesssim \|\na_xf\|_{L^2_xL^2_4}^2 \lesssim \|f\|_{X_{k_0}}^2.  
\end{align}
For union of cubes case, we apply embedding \cite[Theorem 6.7-5]{Ciarlet2013} with boundary values $\pa_{x_i}\phi=0$ on $\Gamma_i$ to obtain \eqref{39}. 
Taking difference of \eqref{19} over $\pm$, we have $\pa_t(a_+-a_-) = -\na_x\cdot\int_{\R^3}v(f_+-f_-)\,dv$. So
\begin{multline}\label{39b}
	\|\pa_t\phi\|_{L^\infty_x}\lesssim \|\na_x\pa_t\phi\|_{L^2_x}^{1/2}\|\na_x\pa_t\phi\|_{H^1_x}^{1/2}\lesssim \|\na_x^2\pa_t\phi\|_{L^2_x}
	\lesssim \|\Delta_x\pa_t\phi\|_{L^2_x}\\\lesssim \|\pa_t(a_+-a_-)\|_{L^2_x}
	\lesssim \|\na_xf\|_{L^2_xL^2_3} \lesssim \|f\|_{X_{k_0}}. 
\end{multline}
For the Boltzmann case, applying \eqref{39a}, \eqref{39}, \eqref{39b} and choosing $\eta<\frac{\gamma_1}{4}$, we have 
\begin{align}\notag\label{426a}
		&\quad\,\pa_t\Big(\frac{1}{2}\|f\|_{X_k}^2+\sum_{|\al|\le 2}C_\al C_{|\al|,0}\|\partial^\alpha\na_x\phi\|_{L^2_x}^2\Big)
		 +\frac{\gamma_1}{2} \|f\|_{Y_k}^2
	\\
	&\notag\le 
	C_{k}\|f\|^2_{H^2_{x, v}}+C    \|f\|_{X_{k_0}}   \|f\|_{X_k}^2+
	C_k\Big(\|\na_x\phi\|_{H^3_x}\|f\|^2_{Y_k}
	+\|\na_x\phi\|^2_{H^2_x}
	+\Vert f \Vert_{X_{k_0}} \|f\|^2_{Y_k}
	\\
	&\quad+ \big(\|f\|_{X_{k_0}}+C\big)\|\na_x\phi\|_{H^1_x}\|f\|_{Y_k}^2
	+\|f\|_{X_{k_0}}\|\na_x\phi\|_{H^1_x}\big(\|f\|_{Y_k}+C\big)\|f\|_{Y_k}\Big),
\end{align}
for some generic constant $\gamma_1$ independent of $k$. 
Here $C_\al$ and $C_{|\al|,0}$ are given in \eqref{231} and \eqref{Cal} respectively.
If $|\beta| \ge 1$, by interpolation we have
\begin{align}\label{430}
	C_k\Vert  \partial^\alpha_\beta  f \Vert_{L^2_{x, v}}^2 \le C \Vert \partial^\alpha f\Vert_{L^2_xH^{|\beta|}_v}^2 \le \frac {\gamma_1} {24} \Vert \partial^\alpha f \Vert_{L^2_xH^{{|\beta|}+s}_v}^2  + C_k \Vert \partial^\alpha f \Vert_{L^{2}_{x, v}}^2  \le  \frac {\gamma_1} {24} \Vert  f \Vert_{Y_k}^2  + C_k \Vert \partial^\alpha f \Vert_{L^{2}_{x, v}}^2. 
\end{align}
Also, by Sobolev embedding or Poincar\'{e}'s inequality, 
\begin{align*}
	\|\na_x\phi\|_{H^2_x}\lesssim \|\na_x^2\phi\|_{H^1_x} = \|\Delta_x\phi\|_{H^1_x}\lesssim \|f\|_{H^1_xL^2_4}\lesssim \eta\|f\|_{Y_k}+C_\eta\|f\|_{H^1_xL^2_v}. 
\end{align*}
In view of the above two estimates, choosing $\eta>0$ small enough, \eqref{426a} implies that 
\begin{multline}\label{44}
	\quad\,\pa_t\Big(\frac{1}{2}\|f\|_{X_k}^2+\sum_{|\al|\le 2}C_\al C_{|\al|,0}\|\partial^\alpha\na_x\phi\|_{L^2_x}^2\Big)
	+\frac{\gamma_1}{4} \|f\|_{Y_k}^2
	\\\lesssim 
	\|f\|^2_{H^2_{x}L^2_v}+C\|f\|_{X_{k_0}}     \|f\|_{X_k}^2+\|\na_x\phi\|^2_{H^2_x}+\big(\sqrt{\E_{k_0}(t)}+\E_{k_0}(t)\big)\D_k(t). 
\end{multline}
%
%
For the Landau case, one can apply a similar calculation as \eqref{37a} instead of \eqref{37} to derive \eqref{44} by using the same technique and we omit the details for brevity.  

 \smallskip\noindent{\bf Step 2. Estimates with spatial derivatives. }
For any $|\al|\le 2$, we apply $\pa^\al$ to \eqref{vplocal1}, take inner product with $e^{\frac{\pm A_{\al,0}\phi}{\<v\>^2}}w^2(\al,0)\pa^\al  f_\pm$ over $\Omega\times\R^3$ to deduce that 
\begin{multline}\label{36a}
	\frac{1}{2}\partial_t\|e^{\frac{\pm A_{\al,0}\phi}{2 \<v\>^2}}\pa^\al f_\pm\|^2_{L^2_{x,v}}+ \big(A_{\al,0}\pa_t\phi\<v\>^{-2}\pa^\al f_\pm,e^{\frac{\pm A_{\al,0}\phi}{\<v\>^2}}w^2(\al,0)\pa^\al f_\pm\big)_{L^2_{x,v}}\\
	 + \big(v\cdot\nabla_x\pa^\al f_\pm,e^{\frac{\pm A_{\al,0}\phi}{\<v\>^2}}w^2(\al,0)\pa^\al f_\pm\big)_{L^2_{x,v}}\mp\big(\nabla_x\phi\cdot\nabla_v\pa^\al f_\pm,e^{\frac{\pm A_{\al,0}\phi}{\<v\>^2}}w^2(\al,0)\pa^\al f_\pm\big)_{L^2_{x,v}}\\
	\mp\sum_{\substack{\al_1\le \al\\|\al_1|\ge 1}}\big(\pa^{\al_1}\nabla_x\phi\cdot\nabla_v\pa^{\al-\al_1} f_\pm,e^{\frac{\pm A_{\al,0}\phi}{\<v\>^2}}w^2(\al,0)\pa^\al f_\pm\big)_{L^2_{x,v}}\pm \big(\pa^\al(\nabla_x\phi\cdot v\mu),e^{\frac{\pm A_{\al,0}\phi}{\<v\>^2}}w^2(\al,0)\pa^\al f_\pm\big)_{L^2_{x,v}}\\ 
	= 
	\big(\pa^\al L_\pm f+\pa^\al Q(f_\pm+f_\mp,f_\pm) ,e^{\frac{\pm A_{\al,0}\phi}{\<v\>^2}}w^2(\al,0)\pa^\al f_\pm\big)_{L^2_{x,v}},
\end{multline}
where $L_\pm$ is given in \eqref{nlp}. 
We denote the second to seventh terms in \eqref{36a} by $J_1$ to $J_6$ and estimate them one by one.
Similar to the calculation of $I_1$ to $I_6$ from \eqref{I1} to \eqref{I6}, one can obtain the following estimates by using \eqref{214a} and \eqref{estiv5}:
\begin{equation*}
\begin{aligned}
	|J_1+J_2+J_3+J_4|&\lesssim \big(\|\pa_t\phi\|_{L^\infty_x}+\|\na_x\phi\|_{H^3_x}\big)\|f\|_{Y_k}^2
	\\&\lesssim C\|\pa_t\phi\|_{L^\infty_x}\|f\|_{X_k}^2+\|\na_x\phi\|_{H^2_x}\|f\|_{Y_k}^2
	\lesssim \|f\|_{X_{k_0}}\|f\|_{X_k}^2+\sqrt{\E_{k_0}(t)}\D_k(t),
\end{aligned}
\end{equation*}
where we apply \eqref{39} and \eqref{39b} for the terms $\|\na_x^3\phi\|_{L^2_x}$ and $\|\pa_t\phi\|_{L^\infty_x}$ respectively. 
For $J_5$, using \eqref{231}, we have 
\begin{align*}
	J_5&\ge C_\al\pa_t\|\partial^\alpha\na_x\phi\|_{L^2_x}^2- C\|\pa^\al\na_x\phi\|_{L^2_x}\|\pa^\al\{\I-\P\}f\|_{L^2_xL^2_5}-C\|\partial^\alpha \nabla_x \phi\|_{L^2_x}\|\na_x\phi\|_{H^1_x}\|f\|_{Y_k}\\
	&\ge C_\al\pa_t\|\partial^\alpha\na_x\phi\|_{L^2_x}^2- \eta\|\pa^\al\na_x\phi\|^2_{L^2_x}-\eta\|w(\al,0)\pa^\al\{\I-\P\}f\|^2_{L^2_xL^2_{\gamma/2}}
	\\&\quad-C_\eta\|\pa^\al\{\I-\P\}f\|^2_{L^2_xL^2_v}-C\sqrt{\E_{k_0}(t)}\D_k(t),
\end{align*}
for any $\eta>0$. Here $C_\al$ is given in \eqref{231}. 
For the term $J_6$, we need some different estimates. 
Noticing 
	$L_\pm ( \partial^\alpha \PP f) = 0$ ,
we split $J_6$ into
\begin{align*}
	J_6&=\big(\pa^\al Q(2\mu+f_\pm+f_\mp,\{\II-\PP\}f),e^{\frac{\pm A_{\al,0}\phi}{\<v\>^2}}w^2(\al,0)\pa^\al f_\pm\big)_{L^2_{x,v}}\\
	&\quad\notag+\big(\pa^\al Q(f_\pm+f_\mp,\PP f),e^{\frac{\pm A_{\al,0}\phi}{\<v\>^2}}w^2(\al,0)\pa^\al f_\pm\big)_{L^2_{x,v}}\\
	&\quad\notag+\big(\pa^\al Q(\{\II-\PP\}f+\{\I_\mp-\P_\mp\}f,\mu),e^{\frac{\pm A_{\al,0}\phi}{\<v\>^2}}w^2(\al,0)\pa^\al f_\pm\big)_{L^2_{x,v}}\\
	&=: J_7+J_8+J_{9}. 
\end{align*}
We further split $J_7$ as 
\begin{align*}
	J_7 &= \big(\pa^\al Q(2\mu+f_\pm+f_\mp,\{\II-\PP\}f),w^2(\al,0)\pa^\al \{\II-\PP\}f\big)_{L^2_{x,v}}\\
	&\quad + \big(\pa^\al Q(2\mu+f_\pm+f_\mp,\{\II-\PP\}f),w^2(\al,0)\pa^\al \PP f\big)_{L^2_{x,v}}\\
	&\quad + \big(\pa^\al Q(2\mu+f_\pm+f_\mp,\{\II-\PP\}f),\big(e^{\frac{\pm A_{\al,0}\phi}{\<v\>^2}}-1\big)w^2(\al,0)\pa^\al f_\pm\big)_{L^2_{x,v}}\\
	&=: J_{7,1}+J_{7,2}+J_{7,3}. 
\end{align*}
For Boltzmann case, we deduce from \eqref{236} that  $J_{7,1}$ satisfies 
\begin{align*}
	\sum_\pm J_{7,1}  
	&\le  - \frac {1} {8} \Vert  b(\cos \theta) \sin^2 \frac \theta 2  \Vert_{L^1_\theta}     \Vert w(\alpha,0)  \partial^\alpha \{\I-\P\} f   \Vert_{L^2_x L^2_{\gamma/2, * }}^2  \\&\quad
	-\frac{\gamma_1}{2} \Vert  w(\alpha,0) \partial^\alpha \{\I-\P\} f \Vert_{L^2_xH^s_{\gamma/2}}^2+C_k\sqrt{\E_{k_0}(t)}\D_k(t) + C_{k}  \Vert \partial^\alpha \{\I-\P\} f \Vert_{L^2_{x, v}}^2. 
\end{align*}
For $J_{7,2}$ and $J_{7,3}$, for the Boltzmann case,  we deduce from Lemma \ref{L21} that   
\begin{align*}
	|J_{7,2}|&\lesssim  (C+\|\<v\>^{14}f\|_{H^2_xL^2_v})\|\{\I-\P\}f\|_{H^2_xL^2_v}\|w^2(\al,0)\P f\|_{H^2_xH^{2s}_{\gamma+2s}}\\
	&\lesssim \eta\|\P f\|^2_{H^2_xL^2_v}
	+ C_\eta\|\{\I-\P\}f\|^2_{H^2_xL^2_v}+C_k\sqrt{\E_{k_0}(t)}\D_k(t).
\end{align*}
 It follows from \eqref{236a} that 
\begin{align*}
	|J_{7,3}|&\lesssim \|\na_x\phi\|_{H^1_x}\Big((C+\|\<v\>^{14}f\|_{H^2_{x,v}})\|\{\I-\P\}f\|_{Y_k}\|\P f\|_{Y_k}+\|\<v\>^{14}\{\I-\P\}f\|_{H^2_{x,v}}(C+\|f\|_{Y_k})\|\P f\|_{Y_k}\Big)\\
	&\lesssim \big(\sqrt{\E_{k_0}(t)}+\E_{k_0}(t)\big)\D_k(t). 
\end{align*}
For the Landau case, it follows from \eqref{211a} and \eqref{211b} that 
\begin{align*}
	\sum_\pm J_{7,1} &\le  - \frac{\gamma_1}{2} \Vert  w(\alpha,0) \partial^\alpha\{\I-\P\}  f\Vert_{L^2_x L^2_{D}}^2 + C_{k}  \Vert \partial^\alpha\{\I-\P\} f \Vert_{L^2_x L^2_v}^2+C_k\sqrt{\E_{k_0}(t)}\D_k(t),
\end{align*}
with some constants $\gamma_1,C_k>0$. 
Applying Lemma \ref{L55} and \eqref{211c}, we have 
\begin{align*}
	|J_{7,2}|&\lesssim(C+\Vert f \Vert_{L^2_5}) \Vert \{\I-\P\}f \Vert_{H^2_xL^2_v}\Vert \P f \Vert_{H^2_xH^2_5}
	\lesssim \eta\|\P f\|^2_{H^2_xL^2_v} + C_\eta\|\{\I-\P\}f\|^2_{H^2_xL^2_v}+C_k\sqrt{\E_{k_0}(t)}\D_k(t),
\end{align*}
and 
\begin{align*}
	|J_{7,3}|&\lesssim(C+\|\<v\>^7f\|_{H^2_{x,v}})\|\na_x\phi\|_{H^1_x}\|\{\I-\P\}f\|_{Y_k}\|f\|_{Y_k}
	\lesssim \big(\sqrt{\E_{k_0}(t)}+\E_{k_0}(t)\big)\D_k(t). 
\end{align*}
For $J_8$, we deduce from \eqref{236c} that for the Boltzmann case:
\begin{align*}
|J_8|\lesssim  \|\<v\>^{14}f\|_{H^2_{x,v}}\|\<v\>^{2s}\P f\|_{Y_k}\|f\|_{Y_k}
+\|\<v\>^{14}\P f\|_{H^2_{x,v}}\|f\|_{Y_k}\|f\|_{Y_k}\lesssim \sqrt{\E_{k_0}(t)}\D_k(t), 
\end{align*}
and from \eqref{211d} that for the Landau case:
\begin{align*}
|J_8|\lesssim\|f\<v\>^7\|_{H^2_{x,v}}\|\<v\>\P f\|_{Y_k}\|f\|_{Y_k} \lesssim \sqrt{\E_{k_0}(t)}\D_k(t). 
\end{align*}
Here we let $k\ge k_0$. 
For $J_{9}$, we split it as  
\begin{align*}
J_9 =& \big(\pa^\al Q(\{\II-\PP\}f+\{\I_\mp-\P_\mp\}f,\mu),e^{\frac{\pm A_{\al,0}\phi}{\<v\>^2}}w^2(\al,0)\pa^\al \{\II-\PP\}f f_\pm\big)_{L^2_{x, v}}
\\
&+\big(\pa^\al Q(\{\II-\PP\}f+\{\I_\mp-\P_\mp\}f,\mu),e^{\frac{\pm A_{\al,0}\phi}{\<v\>^2}}w^2(\al,0)\partial^\alpha \PP f_\pm\big)_{L^2_{x, v}}
\\
=: &J_{9,1} +  J_{9,2}. 
\end{align*}
Then for Boltzmann case, applying Lemma \ref{L29}, we have for any $\eta>0$ that 
\begin{align*}
	|J_{9, 1}|&\le \Vert b(\cos\theta) \sin^{k-2 } \frac \theta 2 \Vert_{L^1_\theta}   \Vert w(\alpha,0)  \partial^\alpha \{\I-\P\}f\Vert_{L^2_x L^2_{\gamma/2, *}}\Vert \partial^\alpha \{\I-\P\}f    w(\alpha,0) \Vert_{L^2_x L^2_{\gamma/2, *}}
	\\
		&\quad+ C_k \Vert w(\alpha,0) \partial^\alpha \{\I-\P\}f \Vert_{L^2_{x} L^2_{\gamma/2-1/2}}   \Vert \partial^\alpha \{\I-\P\}f  w(\alpha,0)  \Vert_{L^2_{x}L^2_{\gamma/2-1/2}} 
	\\
	&\le \frac {1} {16} \Vert  b(\cos \theta) \sin^2 \frac \theta 2  \Vert_{L^1_\theta}   \Vert w(\alpha,0) \partial^\alpha \{\I-\P\}f \Vert_{L^2_x L^2_{\gamma/2, *}}^2 
	\\
	&\quad+ \eta\Vert w(\alpha,0) \partial^\alpha \{\I-\P\}f  \Vert_{L^2_{x} L^2_{\gamma/2}}^2 + C_{k,\eta}\Vert  \{\I-\P\}f \Vert_{H^2_{x}L^2_v}^2
	,
\end{align*}
For the term $J_{9,2}$, by Lemma \ref{L21} we have
\begin{align*}
	|J_{9, 2}|&\le C  \Vert w(\alpha,0)  \partial^\alpha \{\I-\P\}f\Vert_{L^2_x L^2_{7}}  \Vert \mu \Vert_{H^{2s}} \Vert \partial^\alpha \P f    w(\alpha,0) \Vert_{L^2_x L^2_{\gamma+2s}}
	\\
	&\le \eta\|\pa^\al \P f\|_{L^2_{x, v}}^2+ \eta\Vert w(\alpha,0) \partial^\alpha \{\I-\P\}f  \Vert_{L^2_{x} L^2_{\gamma/2}}^2 + C_{k,\eta}\Vert  \{\I-\P\}f \Vert_{H^2_{x}L^2_v}^2	,
\end{align*}
For the Landau case, we have from Lemma \ref{L29} that 
\begin{multline*}
	|J_{9}|\le C_k\Vert \partial^\alpha \{\I-\P\}f \Vert_{L^2_{x} L^2_5}\Vert \partial^\alpha  f \Vert_{L^2_{x} L^2_5} \\\le \eta\|w(\al,0)\pa^\al\{\I-\P\}f\|^2_{L^2_xL^2_{D}}+\eta\|\pa^\al \P f\|_{L^2_{x,v}}^2+C_\eta\|\{\I-\P\}f\|^2_{H^2_{x}L^2_v}. 
\end{multline*}
Taking summation $\sum_\pm\sum_{|\al|\le 2}\eqref{36a}$,
combining the above estimates for $J_1$ to $J_{9}$, applying \eqref{39a} and \eqref{430}, we deduce that for small $\eta>0$, 
\begin{multline}\label{44a}
	\quad\,\partial_t\sum_{|\al|\le 2}\Big(\frac{1}{2}\|e^{\frac{\pm A_{\al,0}\phi}{2 \<v\>^2}}\pa^\al f_\pm\|^2_{L^2_{x,v}}
	+C_\al\|\na_x\phi\|_{H^2_x}^2\Big)+\frac{\gamma_1}{3}\sum_{|\al|\le2} \Vert w(\alpha,0) \partial^\alpha \{\I-\P\} f  \Vert_{L^2_xH^s_{\gamma/2}}^2\\
	\lesssim \|f\|_{X_{k_0}}\|f\|_{X_k}^2+C_{\eta,k}  \Vert  \{\I-\P\} f \Vert_{H^2_xL^2_v}^2+ \eta\|\na_x\phi\|^2_{H^2_x}+\eta\|\P f\|_{H^2_{x}L^2_v}^2+C(\sqrt{\E_{k_0}(t)}+\E_{k_0}(t))\D_k(t). 
\end{multline}
Here, 
for the Boltzmann case, we use the fact that $ \Vert b(\cos\theta) \sin^{k-2 } \frac \theta 2 \Vert_{L^1_\theta} \le \frac {1} {16} \Vert  b(\cos \theta) \sin^2 \frac \theta 2  \Vert_{L^1_\theta}$,   
which follows from $k\ge k_0$. Taking linear combination $\eta \times\eqref{44}+\eqref{44a}$ and applying \eqref{39b}, we obtain 
\begin{multline}\label{434a}
	\pa_t\E_1(t)+\lam\eta \|f\|^2_{Y_k} +\frac{\gamma_1}{4}\sum_{|\al|\le2} \Vert  w(\alpha,0) \partial^\alpha\{\I-\P\}  f\Vert_{L^2_x L^2_{D}}^2\\
	\lesssim \|f\|_{X_{k_0}}\|f\|_{X_k}^2+\eta\|f\|^2_{H^2_xL^2_v}+ C_{\eta,k}  \Vert  \{\I-\P\} f \Vert_{H^2_xL^2_v}^2+ \eta\|\na_x\phi\|^2_{H^2_x}+\eta\|\P f\|_{H^2_{x}L^2_v}^2+C(\sqrt{\E_{k_0}(t)}+\E_{k_0}(t))\D_k(t).
\end{multline}
where $\E_1(t)$ is given by
\begin{multline}\label{EE}
	\E_1(t) = \frac{\eta}{2}\|f\|_{X_k}^2+\eta \sum_{|\al|\le 2}C_\al C_{|\al|,0}\|\partial^\alpha\na_x\phi\|_{L^2_x}^2+ \sum_\pm\sum_{|\al|\le 2}\Big(\frac{1}{2}\|e^{\frac{\pm A_{\al,0}\phi}{2 \<v\>^2}}\pa^\al f_\pm\|^2_{L^2_{x,v}}
	+C_\al\pa_t\|\na_x\phi\|_{H^2_x}^2\Big).  
\end{multline}
It's direct to verify that $\E_1(t)$ satisfies \eqref{E111}. 
For the second right-hand term of \eqref{434a}, we split it as 
\begin{align}\label{434b}
	\eta\|f\|^2_{H^2_xL^2_v} \lesssim \eta\|\P f\|^2_{H^2_xL^2_v} + \eta\sum_{|\al|\le2} \Vert  w(\alpha,0) \partial^\alpha\{\I-\P\}  f\Vert_{L^2_x L^2_{D}}^2.
\end{align}
Choosing $\eta$ sufficiently small,  
the second right hand term of \eqref{434b} can be absorbed by the left hand side of \eqref{434a} and we obtain \eqref{eq1}. 
This completes the proof of Lemma \ref{L43}.

\end{proof}

\subsection{Recover the energy from semigroup method}
According to Lemma \ref{L43}, we only need to deal with term $\|\{\I-\P\}f\|_{H^2_xL^2_v}$ without velocity derivative on the right hand side of \eqref{eq1}. In order to eliminate this term, we define the semigroup generated by $\L$ given in \eqref{L} to be $S_\L(t)$. Then we first give some estimate on $S_\L(t)$, which is the solution operator to the equation 
\begin{equation}
	\begin{aligned}\label{vpl}
			\pa_tf + v\cdot\na_xf \pm\na_x\phi\cdot v\mu = L_\pm f,\quad f(0)=f_0,\quad
			-\Delta_x\phi = \int_{\R^3}(f_+-f_-)\,dv, \quad\int_{\Omega}\phi(t, x) dx =0.
		\end{aligned}
\end{equation}
If the domain $\Omega$ is union of cubes given by \eqref{Omega}, then we further assume 
\begin{align}\label{boundaryphi1}
	\pa_n\phi=0\ \ \text{ on }\pa\Omega. 
\end{align}
To obtain the estimate of \eqref{vpl}, we denote $-\Delta_x^{-1}\int_{\R^3}(\cdot)_+-(\cdot)_-\,dv$ to be the solution operator to the second equation of \eqref{vpl}. Then we define linear operators $A=[A_+,A_-]$, $B=[B_+,B_-]$ and $K_1,K_2$ as 
\begin{align*}
	&A_\pm = -v\cdot\na_x+L_\pm  -M\chi_R, \quad B_\pm = -v\cdot\na_x+L_\pm, \quad   K_1  = M\chi_R, \\
	 &K_2 = \pm\mu v\cdot\na_x\Delta_x^{-1}\int_{\R^3}\big((\cdot)_+-(\cdot)_-\big)\,dv,\quad
	\L = -v\cdot\na_x\pm\mu v\cdot\na_x\Delta_x^{-1}\int_{\R^3}\big((\cdot)_+-(\cdot)_-\big)\,dv+L_\pm,
\end{align*}
where $M>0$ is a large constant and $\chi \in D(\R)$ is the truncation function satisfying $1_{[-1,1]} \le \chi \le 1_{[-2,2]}$ and we denote $\chi_R(\cdot) := \chi(\cdot/R)$ for $R > 0$.
For the case of the union of cubes, we consider the domain of these operators, i.e. $A,B,\L$, with restriction of specular-reflection boundary condition for $f$ and Neumann boundary condition for $\phi$:
\begin{align*}
	f(R_xv) = f(v) \text{ on }\gamma_-,\quad \pa_n\Delta_x^{-1}\int_{\R^3}\big(f_+-f_-\big)\,dv=0\text{ on }\pa\Omega.
\end{align*}
Then we have the following lemma. 

\begin{lem}\label{Lem41}
	Consider both Boltzmann and Landau cases. For $k\ge k_0$, there exists $\ve_0>0$ such that if 
	\begin{align}
		\label{fphi}
		\|f\|_{H^2_xL^2_{k_0+4}}\le \ve_0,
	\end{align}
	then we have 
	\begin{align}\label{SL1}
		\|S_\L(t)f\|_{H^2_xL^2_{k}} 
		&\lesssim\<t\>^{-\frac{7}{6}}\|f\|_{H^2_xL^2_{k+4}}, 
	\end{align}
and (suppose $s=1$ for Landau case)
\begin{align}\label{SL3}
	\Big(\int^\infty_0\|\<v\>^{k}S_\L(t)f\|_{H^2_xL^2_v}^2\,dt\Big)^{\frac{1}{2}}
	\lesssim \|f\|_{H^2_xH^{-s}_{k-\gamma/2}}.
\end{align}
\end{lem}
\begin{proof}
	We will prove \eqref{SL1} and \eqref{SL3} in two steps. We only prove the Boltzmann case and the Landau case can be proved similarly. 
	
	\smallskip\noindent{\bf Step 1.}
 By Duhamel's principle, we have 
\begin{align}\label{Du1}
	S_B(t) = S_A(t) + \int^t_0S_B(t-s)K_1S_A(s)\,ds,
\end{align}
and 
\begin{align}\label{Du2}
	S_\L(t) = S_B(t) + \int^t_0S_\L(t-s)K_2S_B(s)\,ds.
\end{align}
Using Lemma \ref{L23} and Theorem \ref{T25} with $g=0$ for the Boltzmann case and Lemma \ref{L51} for the Landau case, we deduce that for $k\ge k_0$, 
\begin{align}\label{48}
	(-Af,\<v\>^{2k}f)_{L^2_{x,v}}\ge \frac{\gamma_1}{2}\|\<v\>^kf\|_{L^2_xL^2_{D}}^2\gtrsim \|\<v\>^{k+\frac{\gamma}{2}}f\|_{L^2_{x,v}}^2, 
\end{align}
here $(v\cdot\na_x\pa^\al f,\pa^\al f)_{L^2_{x,v}}=0$ by using change of variable $v\mapsto R_xv$ for the case of union of cubes. Moreover, we use Lemma \ref{highspecular} to obtain 
\begin{align}\label{444}
	(v\cdot\na_xf,\<v\>^{2k}f)_{L^2_{x,v}}=0.
\end{align} 
For the case of the union of cubes,  and by using change of variable $v\mapsto R_xv$. For the case of torus, \eqref{444} also holds.  Then $A$ generates a semigroup on $L^2_k$ such that 
\begin{align}\label{48a}
	\|S_A(t)f\|_{L^2_xL^2_k}\le \|f\|_{L^2_xL^2_k}.
\end{align}
Also, by definition of semigroup, we have the equation 
\begin{align}\label{patSA}
	\pa_tS_A(t)f-AS_A(t)f=0.
\end{align}
Taking inner product of \eqref{patSA} with $\<v\>^{2k}S_A(t)f$ over $\Omega\times\R^3$, we have from \eqref{48} that
\begin{align}\label{patSA1}
	\frac{1}{2}\pa_t\|\<v\>^{k}S_A(t)f\|^2_{L^2_{x,v}} + \lam \|\<v\>^{k}S_A(t)f\|_{L^2_xL^2_D}^2 \le 0. 
\end{align} 
For the hard potential case, i.e. $\gamma\ge 0$, we have \[\|\<v\>^{k}S_A(t)f\|_{L^2_xL^2_D}\gtrsim \|\<v\>^{k}S_A(t)f\|_{L^2_xL^2_v},\] and hence, 
\begin{align}\label{SAex}
	\|S_A(t)f\|_{L^2_xL^2_k}\le e^{-\lam t}\|f\|_{L^2_xL^2_k},
\end{align}for some $\lam>0$. 
For the soft potential case and $k\ge k_1\ge k_0$, it follows from \eqref{patSA1} that 
\begin{align*}
	\pa_t\|S_A(t)f\|_{L^2_xL^2_{k_1}}^2&\lesssim -\|S_A(t)f\|_{L^2_xL^2_{k_1+\frac{\gamma}{2}}}^2\lesssim -\lam\<R\>^{\gamma}\|S_A(t)f\|_{L^2_xL^2_{k_1}}^2 + \<R\>^{2(k-k_1)+\gamma}\|S_A(t)f\|_{L^2_xL^2_k}^2,
\end{align*}
for some $\lam>0$. 
Solving this ODE and using \eqref{48a}, we have 
\begin{align*}
	\|S_A(t)f\|_{L^2_xL^2_{k_1}}^2 &\lesssim e^{-\lam\<R\>^{\gamma}t}\big(\|f\|_{L^2_xL^2_{k_1}}^2 + \<R\>^{2(k-k_1)+\gamma}\int^t_0e^{\lam\<R\>^\gamma s}\,ds\|f\|_{L^2_xL^2_k}^2\big)\\
	&\lesssim e^{-\lam\<R\>^{\gamma}t}\|f\|_{L^2_xL^2_{k_1}}^2+\<R\>^{2(k-k_1)}\|f\|_{L^2_xL^2_k}^2. 
\end{align*}
Choosing $\<R\> = \big(\<t\>[\log\<t\>]^{-\frac {2(k-k_0)} {\lambda}}\big)^{-\frac{1}{\gamma}}$, we have 
\begin{align}\label{SA1}
	\|S_A(t)f\|_{L^2_xL^2_{k_1}}\lesssim \<t\>^{-\frac{k-k_1-1/2}{|\gamma|}}\|f\|_{L^2_xL^2_k},
\end{align}
For $k\ge k_1+4$, since $|\gamma|\le 3$, we know that $\<t\>^{-\frac{k-k_1-1/2}{|\gamma|}}\le \<t\>^{-\frac{7}{6}}$. 
Together with \eqref{SAex}, we deduce that 
\begin{align}\label{SA}
	\|S_A(t)f\|_{L^2_xL^2_{k_1}}\lesssim \<t\>^{-\frac{7}{6}}\|f\|_{L^2_xL^2_k}, 
\end{align}
For the semigroups $S_B(t)$ and $S_\L(t)$ in exponential weighted space, we have 
\begin{equation}\label{498}
	\begin{aligned}
		\pa_tS_B(t)f - BS_B(t)f& =0,\quad
		\pa_tS_\L(t)f - \L S_\L(t)f=0.
	\end{aligned}
\end{equation}
Taking $H^2_xL^2_v$ inner product with $\mu^{-1}S_B(t)f$ and $\mu^{-1}S_\L(t)f$ over $\Omega\times\R^3$ respectively, we have 
\begin{equation}\label{499}
	\begin{aligned}
		&\pa_t\|\mu^{-\frac{1}{2}}S_B(t)f\|_{H^2_xL^2_v}^2 + \lam\|\mu^{-\frac{1}{2}}\{\I-\P\mu^{\frac{1}{2}}\}(S_B(t)f)\|^2_{H^2_xH^s_{\gamma/2}}\le 0,\\
		&\pa_t\|\mu^{-\frac{1}{2}}S_\L(t)f\|_{H^2_xL^2_v}^2 + \pa_t\|\na_x\phi\|_{H^2_x}^2 + \lam\|\mu^{-\frac{1}{2}}\{\I-\P\mu^{\frac{1}{2}}\}(S_\L(t)f)\|^2_{H^2_xH^s_{\gamma/2}}\le 0.
	\end{aligned}
\end{equation}
where the dissipation rate of $L_\pm$ for exponential perturbation can be found in \cite{GS,G} and $\phi$ is given by 
\begin{align}\label{phid}
	\phi = \Delta_x^{-1}\int_{\R^3}(S_\L(t)f)_+-(S_\L(t)f)_-\,dv. 
\end{align} 
In the second estimate of \eqref{499}, we used the fact that 
$\pa_t\int_{\R^3}S_\L(t)f\,dv + \int_{\R^3}v\cdot\na_x S_\L(t)f\,dv=0$. Thus
\begin{align*}
	&\quad\,\sum_\pm\big(\pm\mu v\cdot\na_x\pa^\al\Delta_x^{-1}\int_{\R^3}(S_\L(t)f)_+-(S_\L(t)f)_-\,dv,\mu^{-1}\pa^\al(S_\L(t)f)_\pm\big)_{L^2_{x,v}}\\
	&= \big(\pa^\al\phi,\int_{\R^3}v\cdot\na_x\pa^\al((S_\L(t)f)_+-(S_\L(t)f)_-)\,dv\big)_{L^2_{x}}\\
	&= \big(\pa^\al\phi,-\pa_t\pa^\al\int_{\R^3}((S_\L(t)f)_+-(S_\L(t)f)_-)\,dv\big)_{L^2_{x}}  = \frac{1}{2}\pa_t\|\pa^\al\na_x\phi\|_{L^2_x}^2.
\end{align*}
Here we take integration by parts on $\na_x$ with \eqref{boundaryphi1} and use boundary condition \eqref{specular} to deduce $\int_{\R^3}v((S_\L(t)f)_+-(S_\L(t)f)_-)\,dv=0$ on $\pa\Omega$ for the case of union of cubes.

In order to derive the weighted version of \eqref{499}, 
taking $H^2_x L^2_v$ inner product of \eqref{498} with $\<v\>^{2k}\mu^{-1}S_B(t)f$ and $\<v\>^{2k}\mu^{-1}S_\L(t)f$ over $\Omega\times\R^3$ respectively, we have 
\begin{equation}\label{4991}
	\begin{aligned}
		&\pa_t\|\<v\>^k\mu^{-\frac{1}{2}}S_B(t)f\|_{H^2_xL^2_v}^2 + \lam\|\<v\>^k(\mu^{-\frac{1}{2}}S_B(t)f)\|^2_{H^2_xH^s_{\gamma/2}}\\&\quad\le C_R \|\chi_R\mu^{-\frac{1}{2}}S_B(t)f\|^2_{H^2_xL^2_v},\\
		&\pa_t\|\<v\>^k\mu^{-\frac{1}{2}}S_\L(t)f\|_{H^2_xL^2_v}^2 + \pa_t\|\na_x\phi\|_{H^2_x}^2 + \lam\|\<v\>^k(\mu^{-\frac{1}{2}}S_\L(t)f)\|^2_{H^2_xH^s_{\gamma/2}}\\&\quad\le C \|\na_x\phi\|_{H^2_x}^2  
		+C_{R}\|\chi_R\mu^{-\frac{1}{2}}S_\L(t)f\|^2_{H^2_xL^2_v}. 
	\end{aligned}
\end{equation}
By macroscopic estimates from \cite[Theorem 3.4]{Deng2021e} for the case of the union of cubes and exponential perturbation, we have
\begin{equation}\label{499a}
	\begin{aligned}
		&\pa_t\E_{int,B}(t) + \lam\|\mu^{-\frac{1}{2}}\P\mu^{\frac{1}{2}}S_B(t)f\|^2_{H^2_xH^s_{\gamma/2}}\le C\|\mu^{-\frac{1}{2}}\{\I-\P\mu^{\frac{1}{2}}\}S_B(f)\|_{H^2_xL^2_{\gamma/2}}^2,\\
		&\pa_t\E_{int,\L}(t) + \pa_t\|\na_x\phi\|_{H^2_x}^2 + \lam\|\mu^{-\frac{1}{2}}\P\mu^{\frac{1}{2}}S_\L(t)f\|^2_{H^2_xH^s_{\gamma/2}} + \lam\|\na_x\phi\|^2_{H^2_x}\\
		\le &C\|\mu^{-\frac{1}{2}}\{\I-\P\mu^{\frac{1}{2}}\}S_\L(f)\|_{H^2_xL^2_{\gamma/2}}^2 + \|\na_x\phi\|^4_{L^2_x},
	\end{aligned}
\end{equation}
where $\phi$ is given by \eqref{phid} and $\E_{int,B},\E_{int,\L}$ are functional satisfying 
\begin{align}\label{4993}
	\E_{int,B}(t)\lesssim \|\mu^{-\frac{1}{2}}S_B(t)f\|_{H^2_xL^2_v}^2,\quad
	\E_{int,\L}(t)\lesssim \|\mu^{-\frac{1}{2}}S_\L(t)f\|_{H^2_xL^2_v}^2. 
\end{align}respectively. 
Although the macroscopic estimates from \cite{Deng2021e} are for the union of cubes, the case of the torus can be similarly derived with simpler calculations; see also \cite{G8}.  
 Taking linear combinations $\eqref{499}+\kappa^2\times \eqref{4991}+\kappa\times\eqref{499a}$ and assuming the {\it a priori} assumption 
 \begin{align}\label{prioriphi}
 	\|\na_x\phi\|^2_{H^2_x}\le \ve, 
 \end{align}
 with $\ve>0$ sufficiently small and $\phi$ given by \eqref{phid}, we have 
\begin{equation}\label{499b}
	\begin{aligned}
		&\pa_t\E_{B}(t) + \lam\|\<v\>^k\mu^{-\frac{1}{2}}S_B(t)f\|^2_{H^2_xH^s_{\gamma/2}}\le 0,
		&\pa_t\E_{\L}(t) + \lam\|\<v\>^k\mu^{-\frac{1}{2}}S_\L(t)f\|^2_{H^2_xH^s_{\gamma/2}} + \lam\|\na_x\phi\|^2_{H^2_x}\le 0,
	\end{aligned}
\end{equation}
where 
\begin{align*}
	\E_{B}(t) &= \|\mu^{-\frac{1}{2}} S_B(t)f\|_{H^2_xL^2_v}^2 + \kappa^2\|\<v\>^k\mu^{-\frac{1}{2}}S_B(t)f\|_{H^2_xL^2_v}^2 + \kappa\E_{int,B},\\
	\E_{\L}(t) &= \|\mu^{-\frac{1}{2}} S_\L(t)f\|_{H^2_xL^2_v}^2 + \kappa^2\|\<v\>^k\mu^{-\frac{1}{2}}S_\L(t)f\|_{H^2_xL^2_v}^2 + \kappa\E_{int,\L} + \|\na_x\phi\|_{H^2_x}^2.
\end{align*}
With \eqref{4993}, it's direct to check that 
\begin{align*}
	\E_{B}(t) &\sim \|\mu^{-\frac{1}{2}} S_B(t)f\|_{H^2_xL^2_v}^2 + \|\<v\>^k\mu^{-\frac{1}{2}}S_B(t)f\|_{H^2_xL^2_v}^2,\\
	\E_{\L}(t) &\sim \|\mu^{-\frac{1}{2}} S_\L(t)f\|_{H^2_xL^2_v}^2 + \|\<v\>^k\mu^{-\frac{1}{2}}S_\L(t)f\|_{H^2_xL^2_v}^2  + \|\na_x\phi\|_{H^2_x}^2.
\end{align*}
We calculate the first inequality of \eqref{499b} and the second one is similar. When $\gamma\ge 0$, we have 
\[
\pa_t\E_{B}(t) + \lam\|\<v\>^k\mu^{-\frac{1}{2}}S_B(t)f\|^2_{H^2_xL^2_v}\le 0,
\]
thus
\[
\|\<v\>^k\mu^{-\frac{1}{2}}S_B(t)f\|_{H^2_xL^2_v}^2\le e^{-\lam t}\|\<v\>^k\mu^{-\frac{1}{2}}f\|_{H^2_xL^2_v}^2. 
\]
When $\gamma<0$, we have 
\[
\pa_t\E_{B}(t) + \lam\|\<v\>^{k+\frac{\gamma}{2}}\mu^{-\frac{1}{2}}S_B(t)f\|^2_{H^2_xL^2_v}\le 0.
\]
Then we apply similar arguments for obtaining \eqref{SA1} to deduce that 
\begin{align}\label{SB}
	\|\mu^{-\frac{1}{2}}S_B(t)f\|^2_{H^2_xL^2_{k-4}} \lesssim \<t\>^{-\frac{2(3-\frac{1}{2})}{|\gamma|}}\|\mu^{-\frac{1}{2}}f\|_{H^2_xL^2_k}^2
	\lesssim \<t\>^{-\frac{7}{6}}\|\mu^{-\frac{1}{2}}f\|_{H^2_xL^2_k}^2. 
\end{align}
The estimate for hard potential case is also included in \eqref{SB}. 
Similarly, it follows from the second estimate of \eqref{499b} that 
\begin{align}\label{SL}
	\|\mu^{-\frac{1}{2}}S_\L(t)f\|^2_{H^2_xL^2_{k-4}} +\|\na_x\phi\|_{H^2_x}^2
\lesssim \<t\>^{-\frac{7}{6}}\big(\|\mu^{-\frac{1}{2}}f\|_{H^2_xL^2_k}^2 + \|\na_x\phi|_{t=0}\|_{H^2_x}^2\big). 
\end{align}
Now we turn back to Duhamel's principle \eqref{Du1} and \eqref{Du2}. For any $f$, 
we have from \eqref{Du1}, \eqref{SA} and \eqref{SB} that for $k_1\ge k_0$, 
\begin{align}\label{SBB}\notag
	\|S_B(t)f\|_{H^2_xL^2_{k_1}} &\lesssim  \|S_A(t)f\|_{H^2_xL^2_{k_1}} + \int^t_0\|\mu^{-\frac{1}{2}}S_B(t-s)K_1S_A(s)f\|_{H^2_xL^2_v}\,ds\\
	&\notag\lesssim \<t\>^{-\frac{7}{6}}\|f\|_{H^2_xL^2_{k_1+4}} + 
	\int^t_0\<t-s\>^{-\frac{7}{6}}\|\mu^{-\frac{1}{2}}K_1S_A(s)f\|_{H^2_xL^2_{4}}\,ds\\
	&\lesssim \<t\>^{-\frac{7}{6}}\|f\|_{H^2_xL^2_{k_1+4}} + 
	\int^t_0\<t-s\>^{-\frac{7}{6}}\|S_A(s)f\|_{H^2_xL^2_v}\,ds\lesssim \<t\>^{-\frac{7}{6}}\|f\|_{H^2_xL^2_{k_1+4}}. 
\end{align}
Thus, recalling $\phi$ is given by \eqref{phid}, it follows from \eqref{Du2}, \eqref{SL} and \eqref{SBB} that for $k_1\ge k_0$, 
\begin{align}\label{455a}\notag
	\|S_\L(t)f\|_{H^2_xL^2_{k_1}} &\lesssim  \|S_B(t)f\|_{H^2_xL^2_{k_1}} + \int^t_0\|\mu^{-\frac{1}{2}}S_\L(t-s)K_2S_B(s)f\|_{H^2_xL^2_v}\,ds\\
	&\notag\lesssim \<t\>^{-\frac{7}{6}}\|f\|_{H^2_xL^2_{k_1+4}} + 
	\int^t_0\<t-s\>^{-\frac{7}{6}}\big(\|\mu^{-\frac{1}{2}}K_2S_B(s)f\|_{H^2_xL^2_{4}}
	\\&\qquad\qquad+ \| \nabla_x \Delta_x^{-1}\int_{\R^3}(K_2S_B(s)f)_+-(K_2S_B(s)f)_-\,dv\|_{H^2_x}\big)\,ds.
\end{align}
Note that by using odd property on $\mu v$, we have 
\begin{align*}
\int_{\R^3}&(K_2S_B(s)f)_+-(K_2S_B(s)f)_-\,dv = 2\int_{\R^3}\mu(v) v\cdot\na_x\Delta_x^{-1}\int_{\R^3}\big((S_B(s)f)_+(u)-(S_B(s)f)_-(u)\big)\,du\,dv=0.
\end{align*}
Also, letting $$\phi_1=\Delta_x^{-1}\int_{\R^3}\big((S_B(s)f)_+(u)-(S_B(s)f)_-(u)\big)\,du,$$ when $\Omega$ is torus, we know that 
\begin{align}\label{440}
	\|\na_x\phi_1\|_{L^2_x}\lesssim \|\Delta_x\phi_1\|_{L^2_x}. 
\end{align}
If $\Omega$ is union of cubes given by \eqref{Omega}, noticing from \eqref{boundaryphi1} that $\pa_{x_i}\phi_1=0$ on $\Gamma_i$, by Sobolev inequality \cite[Theorem 6.7-5]{Ciarlet2013}, we also have \eqref{440}. 
 Then we obtain from \eqref{455a} and \eqref{SA} that for $k_1\ge k_0$, 
\begin{align}\label{SLa}\notag
	\|S_\L(t)f\|_{H^2_xL^2_{k_1}} 
	&\lesssim \<t\>^{-\frac{7}{6}}\|f\|_{H^2_xL^2_{k_1+4}} + 
	\int^t_0\<t-s\>^{-\frac{7}{6}}\|\<v\>^4\mu^{\frac{1}{2}}(v) v\cdot\na_x\phi_1\|_{H^2_xL^2_v}\,ds\\
	&\notag\lesssim\<t\>^{-\frac{7}{6}}\|f\|_{H^2_xL^2_{k_1+4}} +  \int^t_0\<t-s\>^{-\frac{7}{6}}\|\Delta_x\phi_1\|_{H^2_xL^2_{k_1+4}}\,ds\\
	&\lesssim\<t\>^{-\frac{7}{6}}\|f\|_{H^2_xL^2_{k_1+4}} +  \int^t_0\<t-s\>^{-\frac{7}{6}}\|S_B(s)f\|_{H^2_xL^2_{2}}\,ds \lesssim\<t\>^{-\frac{7}{6}}\|f\|_{H^2_xL^2_{k_1+4}}. 
\end{align}
Next, we check that \eqref{prioriphi} is fulfilled if $\ve_0>0$ in \eqref{fphi} is small enough. Similar to \eqref{440}, we have 
\begin{align}\label{phi2}
	\|\na_x\phi\|_{H^2_x}\lesssim \|\Delta_x\phi\|_{H^1_x} = \|\int_{\R^3}(S_\L(t)f)_+-(S_\L(t)f)_-\,dv\|_{H^1_x}. 
\end{align}
Then we can obtain that $\|\na_x\phi\|_{H^2_x} \lesssim \|S_\L(t)f\|_{H^1_xL^2_{2}} \lesssim \|f\|_{H^2_xL^2_{k_0+4}}$. This closes the {\it a priori} assumption \eqref{prioriphi} and completes the proof of \eqref{SL1}.

\smallskip\noindent{\bf Step 2.}
For brevity of notations, we let $s=1$ for the Landau case. 
In order to obtain \eqref{SL3}, we consider the dual $A^*_{l}$ of $A_l=\<v\>^lA(\<v\>^{-l}f)$ for any $l\in\R$.
Notice that 
\begin{align*}
	\big(A^*_{l}f,\<v\>^{2k}f\big)_{L^2_xL^2_v} = 
	\big(\<v\>^lf,A(\<v\>^{2k-l}f)\big)_{L^2_xL^2_v}=
	\big(\<v\>^{2l-2k}g,A g\big)_{L^2_xL^2_v}, 
\end{align*}
where we let $g=\<v\>^{2k-l}f$. 
Then the taking $L^2_xL^2_v$ inner product of $\pa_tS_{A^*_{l}}(t)f-A^*_lS_{A^*_{l}}(t)f=0$ with $\<v\>^{2k}S_{A^*_{l}}(t)f$, we have 
\[
\pa_t\|\<v\>^{k}S_{A^*_{l}}(t)f\|^2_{L^2_xL^2_v} + \big(A^*_lS_{A^*_{l}}(t)f,\<v\>^{2k}S_{A^*_{l}}(t)f\big)_{L^2_xL^2_v} = 0.
\]
Hence for $l-k\ge k_0$, 
\[
\pa_t\|\<v\>^{l-k}g\|^2_{L^2_xL^2_v} + \big(\<v\>^{2l-2k}g, Ag\big)_{L^2_xL^2_v} = 0,
\]
where $g=\<v\>^{2k-l}S_{A^*_{l}}(t)f$.
Integrate over $t$, we have 
\begin{align*}
	\lam \int^\infty_0\|\<v\>^{l-k}g\|_{L^2_xH^s_{\gamma/2}}^2\,dt\le \|\<v\>^{l-k}g|_{t=0}\|^2_{L^2_xL^2_v}.
\end{align*}
That is, for $l-k\ge k_0$, 
\begin{align}\label{intSA}
	\lam \int^\infty_0\|\<v\>^{k}S_{A^*_{l}}(t)f\|_{L^2_xH^s_{\gamma/2}}^2\,dt\le \|\<v\>^{k}f\|^2_{L^2_xL^2_v}.
\end{align}
Observe that if $f$ satisfies equation $\pa_tf=Af$ then $g=\<v\>^lf$ satisfies $\pa_tg=A_lg$, and also that $(Af,\<v\>^{2l}f)_{L^2_{x,v}}=(A_lg,g)_{L^2_{x,v}}$. 
Thus, $\<v\>^lS_A(t)f=S_{A_l}(\<v\>^lf)$. 
Moreover, by duality, we have 
\begin{align*}
	(S_{A_l}f,g)_{L^2_{x,v}} = (f,S_{A_l^*}g)_{L^2_{x,v}}. 
\end{align*}
Therefore, for some sequence $\{\varphi_n\}$ in Schwartz space such that $\|\varphi_n\|_{L^2_xL^2_{k}}\le 1$, we have from \eqref{intSA} that 
\begin{align}\label{intSA1}\notag
	\int^\infty_0\|\<v\>^{k}S_A(t)f\|_{L^2_xL^2_v}^2\,dt
	&= \int^\infty_0\lim_{n\to\infty}\big|\big(\<v\>^{2k}S_{A}(t)f,\varphi_n\big)_{L^2_{x,v}}\big|^2\,dt\notag= \liminf_{n\to\infty}\int^\infty_0\big|\big(\<v\>^{2k}f,S_{A^*_{2k}}(t)\varphi_n\big)_{L^2_{x,v}}\big|^2\,dt\\
	&\notag\le\liminf_{n\to\infty}\int^\infty_0\|\<v\>^{k}f\|^2_{L^2_xH^{-s}_{-\gamma/2}}\|\<v\>^kS_{A^*_{2k}}(t)\varphi_n\|_{L^2_xH^s_{\gamma/2}}^2\,dt\lesssim \|\<v\>^{k-\gamma/2}f\|^2_{L^2_xH^{-s}_v}, 
\end{align}
for $k\ge k_0$. This is the estimate for $S_A(t)$. To obtain the estimates for $S_B(t)$, 
it follows from Duhamel's principle \eqref{Du1} that 
\begin{equation}\label{465}
\begin{aligned}
	\Big(\int^\infty_0\|\<v\>^{k}S_B(t)f\|_{H^2_xL^2_v}^2\,dt\Big)^{\frac{1}{2}}
	&\lesssim \Big(\int^\infty_0\|\<v\>^{k}S_A(t)f\|_{H^2_xL^2_v}^2\,dt\Big)^{\frac{1}{2}}\\
	&+ \Big(\int^\infty_0\Big\|\<v\>^{k}\int^t_0S_B(s)K_1S_A(t-s)f\,ds\Big\|_{H^2_xL^2_v}^2\,dt\Big)^{\frac{1}{2}}.
	\end{aligned}
\end{equation}
Using \eqref{SB}, the second right-hand term of \eqref{465} can be estimated by 
\begin{align*}
	\quad\, \int^\infty_0\Big(\int^\infty_s\|\mu^{-\frac{1}{2}}S_B(s)K_1S_A(t-s)f\|_{H^2_xL^2_v}^2\,dt\Big)^{\frac{1}{2}}\,ds&\lesssim\int^\infty_0\Big(\int^\infty_0\<s\>^{-\frac{7}{6}}\|\<v\>^k\mu^{-\frac{1}{2}}K_1S_A(t)f\|^2_{H^2_xL^2_v}\,dt\Big)^{\frac{1}{2}}ds\\
	&\lesssim\Big(\int^\infty_0\|S_A(s)f\|^2_{H^2_xL^2_v}\,ds\Big)^{\frac{1}{2}}
	\lesssim \|f\|_{H^2_xH^{-s}_{k_0-\gamma/2}}.
\end{align*}
Note that $\na_x$ commutes with $S_A(t)$. The first term on the right-hand side of \eqref{465} can be estimated in the same way.
Thus, for $k\ge k_0$, we have from \eqref{465} that 
\begin{align}\label{466}
	\Big(\int^\infty_0\|\<v\>^{k}S_B(t)f\|_{H^2_xL^2_v}^2\,dt\Big)^{\frac{1}{2}}
	\lesssim \|f\|_{H^2_xH^{-s}_{k-\gamma/2}}.
\end{align}
Next we apply \eqref{Du2}, and \eqref{466} to derive the estimates on $S_\L(t)$: 
\begin{align}\label{465a}
        \nonumber
	&\Big(\int^\infty_0\|\<v\>^{k}S_\L(t)f\|_{H^2_xL^2_v}^2\,dt\Big)^{\frac{1}{2}}
	\\ \nonumber
	\lesssim& \Big(\int^\infty_0\|\<v\>^{k}S_B(t)f\|_{H^2_xL^2_v}^2\,dt\Big)^{\frac{1}{2}}
	+ \Big(\int^\infty_0\Big\|\<v\>^{k}\int^t_0S_\L(s)K_2S_B(t-s)f\,ds\Big\|_{H^2_xL^2_v}^2\,dt\Big)^{\frac{1}{2}}
	\\
	\lesssim &\|f\|_{H^2_xH^{-s}_{k-\gamma/2}} + \int^\infty_0\Big(\int^\infty_s\|\<v\>^{k}S_\L(s)K_2S_B(t-s)f\|_{H^2_xL^2_v}^2\,dt\Big)^{\frac{1}{2}}\,ds.
\end{align}
Applying \eqref{SL}, \eqref{440} and \eqref{466}, the second right-hand term of \eqref{465a} can be estimated by 
\begin{align*}
	&\quad\,\int^\infty_0\Big(\int^\infty_s\|\<v\>^{k}\mu^{-\frac{1}{2}}S_\L(s)K_2S_B(t-s)f\|_{H^2_xL^2_v}^2\,dt\Big)^{\frac{1}{2}}\,ds\\
	&\lesssim \int^\infty_0\<s\>^{-\frac{7}{6}}\Big(\int^\infty_0\|\<v\>^{k+4}\mu^{-\frac{1}{2}}K_2S_B(t)f\|^2_{H^2_xL^2_v}  \\&\qquad\qquad+ \|\na_x\Delta_x^{-1}\underbrace{\int_{\R^3}(K_2S_B(t)f)_+-(K_2S_B(t)f)_-\,dv}_{=0}\|_{H^2_x}^2\,dt\Big)^{\frac{1}{2}}\,ds\\
	&\lesssim \Big(\int^\infty_0\|\<v\>^{k+4}\mu^{-\frac{1}{2}}\mu v\cdot\na_x\Delta_x^{-1}\int_{\R^3}\big((S_B(t)f)_+-(S_B(t)f)_-\big)\,dv\|^2_{H^2_xL^2_v}\,dt\Big)^{\frac{1}{2}}\\
	&\lesssim \Big(\int^\infty_0\|S_B(t)f\|^2_{H^2_xL^2_2}\,dt\Big)^{\frac{1}{2}}
	\lesssim \|f\|_{H^2_xH^{-s}_{k-\gamma/2}}. 
\end{align*}
The term with brace is equal to zero because $\int_{\R^3}v\mu\,dv=0$. Inserting this into \eqref{465a}, we obtain \eqref{SL3}. 
This completes the proof of Lemma \ref{Lem41}. 
\end{proof}

Next, we can introduce a norm 
\begin{align}\label{tri}
	\vertiii{f}^2 :=
	\sum_{|\al|\le 2}\int_{0}^{+\infty} \Vert S_{\L}(\tau)\pa^\al f\Vert_{L^2_x L^2_v}^2 d\tau,
\end{align}
where the associated inner product is given by
\begin{align*}
	((f,g)) : = 
	\int_{0}^{+\infty}\sum_{|\al|\le 2} ( S_{\L}(\tau)\pa^\al f,  S_{\L}(\tau)\pa^\al g)_{L^2_xL^2_v} d\tau,
\end{align*}
Note that $\pa^\al$ commutes with $\L$ and hence, commutes with $S_\L(t)$. Then by \eqref{SL1} and \eqref{SL3}, we have 
\begin{align}\label{41a}
		\vertiii{f}\lesssim \Vert f \Vert_{H^2_xL^2_{k_0+4}}, \quad 
	\vertiii{f}\lesssim\Vert  f \Vert_{H^2_xH^{-s}_{k_0-\frac{\gamma}{2}}}\lesssim\Vert  f \Vert_{H^2_xH^{-s}_{k_0+\frac{3}{2}}},
\end{align}
where we let $s=1$ for the Landau case. 
Then we can recover the loss of energy by using this norm. 

\begin{lem}
\label{L44}
For both the Boltzmann and Landau cases, suppose $(f,\phi)$ is the solution to the equation \eqref{vplocal1}. Then we have 
 \begin{align}\label{ttt}
 	\frac{1}{2}\pa_t {\vertiii{f}^2}+\|f\|_{H^2_xL^2_v}^2\lesssim \sqrt{\E_{k_0}(t)}\D_k(t),  
 \end{align}
 where $\vertiii{\cdot}$ is given in \eqref{tri}, $\E_k(t)$ and $\D_k(t)$ are given in \eqref{DefE} and \eqref{DefD}. 
\end{lem}

\begin{proof}
 Recall the first equation of \eqref{vplocal1}:
\begin{align}\label{433}
\partial_t f_\pm 
=  \underbrace{\pm \nabla_x \phi \cdot \nabla_ v f}_{\text{Term {1}}} 
+ \underbrace{ \L f}_{\text{Term 2}}+\underbrace{ Q(f_\pm+f_\mp,f_\pm) }_{\text{Term {3}}},  
\end{align}
where $\L$ is given in \eqref{L}. 
Then we take the $(( \cdot, \cdot))$ inner product of \eqref{433} with $f$ and compute every term separately. 
   Since $\partial_t$ commute with $L$, we know that $\pa_t$ commute with $S_{\L}(\tau)$ and hence, 
\begin{align*}
((\partial_t \partial^\alpha f , \partial^\alpha f )) = \frac{1}{2}\pa_t {\vertiii{ \partial^\alpha f}^2}.
\end{align*}
Next, we compute Term 1. 
Since  $s \ge \frac 1 2$ implies  $1-s \le s$, by Lemma \ref{L26} and \eqref{41a}, we have 
\begin{equation*}
\begin{aligned}
&\quad\,\int_{0}^{+\infty}\sum_{|\al|\le 2} ( S_{\L}(\tau)\partial^\alpha(  \nabla_x \phi \cdot  \nabla_ v f  ),  S_{\L}(\tau) \partial^\alpha f )_{L^2_x L^2_v} d\tau
\\
&\lesssim  \sum_{|\al|\le 2}\Vert  \partial^\alpha(\nabla_x \phi \cdot  \nabla_ v f)   \Vert_{L^2_{x}H^{-s}_{k_0+3/2} }   \sum_{|\al|\le 2} \Vert \partial^\alpha f\Vert_{L^2_x L^2_{k_0+3/2}} 
\\
&\lesssim  \Vert  \nabla_x \phi \Vert_{H^2_{x}}   \Vert  f   \Vert_{H^2_x H^{1-s}_{k_0+3/2}}   \Vert  f\Vert_{H^2_x L^2_{k_0+3/2  }}\lesssim \sqrt{\E_{k_0}(t)}\D_k(t), 
\end{aligned}
\end{equation*}
for the non-cutoff Boltzmann case, where we let $k\ge k_0+3/2$. 
Similarly, for the Landau case,  we have from \eqref{41a} that for $k\ge k_0+4$, 
\begin{equation*}
\begin{aligned}
\int_{0}^{+\infty}\sum_{|\al|\le 2} ( S_{\L}(\tau)\partial^\alpha(  \nabla_x \phi \cdot  \nabla_ v f  ),  S_{\L}(\tau) \partial^\alpha f )_{L^2_x L^2_v} d\tau
&\lesssim  \sum_{|\al|\le 2} \Vert  \partial^\alpha (\nabla_x \phi \cdot \nabla_vf )  \Vert_{L^2_{x}L^2_{k_0+4} }    \Vert   f\Vert_{H^2_x L^2_{k_0+4}} 
\\
&\lesssim  \Vert  \nabla_x \phi \Vert_{H^2_{x}}   \Vert  f   \Vert_{H^2_x H^{1}_{k_0+4}}   \Vert  f\Vert_{H^2_x L^2_{k_0+4}}\lesssim   \sqrt{\E_{k_0}(t)}\D_k(t). 
\end{aligned}
\end{equation*}
For the term 2 and for both Boltzmann and Landau case, since $\lim_{t\to\infty}\<t\>^{-\frac{7}{6}} = 0$ in \eqref{SL1}, we have
\begin{equation*}
\begin{aligned}
\quad\,\int_0^\infty\sum_{|\al|\le 2} ( S_{\L}(\tau)\L \partial^\alpha f, S_{\L}(\tau)  \partial^\alpha f )_{L^2_x L^2_v} d\tau 
&=\int_0^\infty \frac {d} {d\tau}\sum_{|\al|\le 2} \Vert S_{\L}(\tau)  \partial^\alpha f\Vert_{L^2_x L^2_v}^2 d\tau 
= -\sum_{|\al|\le 2}\Vert  \partial^\alpha f \Vert_{L^2_xL^2_v}^2.
\end{aligned}
\end{equation*}
Finally, we consider the nonlinear part, i.e. Term 3. For the Boltzmann case, by \eqref{2134}, \eqref{41a}, we have
\begin{align*}
&\quad\,\int_0^\infty \sum_{|\al|\le 2}( S_{\L}(\tau)\partial^\alpha Q(f_\pm+f_\mp, f_\pm), S_{\L}(\tau)  \partial^\alpha f )_{L^2_x L^2_v} d\tau 
\\
&\le  C_k\sum_{|\al|\le 2}\sum_{\alpha_1 \le \alpha }\Vert  Q(\partial^{\alpha_1} (f_\pm+f_\mp), \partial^{\alpha -\alpha_1} f_\pm)\Vert_{L^2_x H^{-s}_{k_0-\gamma/2}} \Vert  \partial^\alpha f \Vert_{L^2_x L^2_{k_0}}
\\
&\le C_k (\Vert  f \Vert_{H^{2}_x L^2_{14}} \Vert  f\Vert_{H^{2}_x H^s_{k_0+\gamma/2+2s}}  +  \Vert f \Vert_{H^{2}_x L^2_{14}} \Vert  f\Vert_{H^{2}_x H^s_{k_0+\gamma/2}}  )\Vert \partial^\alpha f \Vert_{L^2_x L^2_{k_0}}\le  C_k\sqrt{\E_{k_0}(t)}\D_k(t).
\end{align*}
Here we let $k_0\ge 14$, $k\ge k_0+\gamma/2+2s$ and apply similar discussion on $\al_1$ as in \eqref{al1}. 
For the Landau case, by \eqref{27a}, \eqref{41a}, we have
\begin{align*}
&\quad\,\int_0^\infty\sum_{|\al|\le 2} ( S_{\L}(\tau) \partial^\alpha Q(f_\pm+f_\mp, f_\pm), S_{\L}(\tau) \partial^\alpha f )_{L^2_x L^2_v} d\tau \\&
\le  C_k\sum_{|\al|\le 2}\sum_{\alpha_1 \le \alpha }\Vert   Q(\partial^{\alpha_1} f_\pm+f_\mp,  \partial^{\alpha -\alpha_1} f_\pm)\Vert_{L^2_x H^{-1}_{k_0+3/2}} \Vert \partial^\alpha f \Vert_{L^2_x L^2_{k_0}}\\&\le C_k \Vert  f \Vert_{H^{2}_x L^2_{k_0+5/2}} \Vert  f\Vert_{H^{2}_x L^2_{D,k_0+5/2}}\Vert \partial^\alpha f \Vert_{L^2_x L^2_{k_0+5/2}}\le  C_k\sqrt{\E_{k_0}(t)}\D_k(t),
\end{align*}
where $k\ge k_0+4$ and we also apply similar discussion on $\al_1$ as in \eqref{al1}. 
Combining the above estimates, the $((\cdot,\cdot))$ inner product of \eqref{433} with $\pa^\al f$ yields 
\begin{align*}
\frac{1}{2}\pa_t {\vertiii{  f}^2}+\sum_{|\al|\le 2}\| \partial^\alpha f\|_{L^2_xL^2_v}^2\lesssim \sqrt{\E_{k_0}(t)}\D_k(t). 
\end{align*}
Then we conclude Lemma \ref{L44}. 
\end{proof}

\subsection{Proof of the main theorem} We start this subsection by proving the main stability theorem as below. Theorem \ref{T46} and \ref{T47} together with local existence from Theorem \ref{T32} will imply Theorem \ref{globaldecay}. We give the details as the following. 
To prove Theorem \ref{globaldecay}, we assume the {\em a priori} assumption as 
\begin{align}\label{priass}
	\sup_{0\le t\le T}\E_{k_0}(t)\le 2M,
\end{align}
for the case $\gamma\in[0,1]$ and 
\begin{align}
	\label{priass2}
	\sup_{0\le t\le T}(1+t)^{\frac{2l}{|\gamma|}}\E_{k_0}(t)\le CM,
\end{align}
for the case $\gamma\in[-3,0),$ where $l>\frac{|\gamma|}{2}$ is a constant and we further assume $k\ge k_0+l$ in this case. 

\begin{thm}\label{T46} 
 Assume that $f_0$ satisfies the conservation laws \eqref{conse2.torus} and $F_0  = \mu + f_0\ge 0$. 
 There exists $k_1\ge k_0$ such that for 
 \begin{align*}
 	k\ge k_0 \text{ for Landau case and }k\ge k_1\text{ for Boltzmann case}, 
 \end{align*}
 there exists a small constant $M\ge 0$ such that if $\E_{k_0}(0) \le M, \E_k(0)<\infty$, then there exist a unique global solution $f(t, x, v )$ to the Vlasov-Poisson-Boltzmann/Landau system \eqref{vplocal1} with $F = \mu+  f \ge 0$ satisfying 
\begin{align}\label{45}
	\pa_t\E_{k}(t)+\lam\D_k(t)\lesssim \|f\|_{X_{k_0}}\E_{k}(t), 
\end{align}
for any $T>0$, for some generic constant $\lam>0$. 
\end{thm}
\begin{proof}
	We take linear combination $\kappa\times\eqref{uni4}+\eqref{eq1}+C_0\times\eqref{ttt}$ with $\kappa,C_0>0$ to deduce that 
\begin{multline*}
	\pa_t\E_{k}(t)+\lambda\kappa\|[a_+,a_-,b,c]\|^2_{H^2_x} +\lambda\kappa \|\na_x\phi\|^2_{H^2_x}
	+\lam\eta\|f\|^2_{Y_k}+C_0\|f\|_{H^2_xL^2_v}^2 +\lam\sum_{|\al|\le2} \Vert  w(\alpha,0) \partial^\alpha\{\I-\P\}  f\Vert_{L^2_xL^2_{D}}^2
	\\\lesssim  \|f\|_{X_{k_0}}\|f\|_{X_k}^2 + 
	\kappa\|\{\I-\P\}f\|^2_{H^2_xL^2_{10}} + \kappa\| N_{\|}\|_{H^1_xL^2_v}^2+
	C_{\eta,k}  \Vert  \{\I-\P\} f \Vert_{H^2_xL^2_v}^2\\+ \eta\|\na_x\phi\|^2_{H^2_x}+\eta\|\P f\|_{H^2_{x}L^2_v}^2+C(\sqrt{\E_{k_0}(t)}+\E_{k_0}(t))\D_k(t),
\end{multline*}
for some generic constant $\lam>0$, 
where $\E_{k}(t)$ is given by 
\begin{align}\label{realE}
	\E_{k}(t) := \kappa G(t) + \E_1(t) + \frac{C_0}{2}\sum_{|\al|\le2}{\vertiii{ \partial^\alpha f}^2}.
\end{align}
Here $G(t)$, $\E_1(t)$ and $\vertiii{\cdot}$ are given in Lemma \ref{L420}, Lemma \ref{L43} and \eqref{tri} respectively. 
Choosing $\kappa>0$ sufficiently small, and then $\eta>0$ sufficiently small, and finally $C_0>0$ sufficiently large, we have 
\begin{align}\label{447}
	&\quad\,\pa_t\E_{k}(t)+\lam\D_k(t)\lesssim \|f\|_{X_{k_0}}\E_{k}(t)+
	\kappa\| N_{\|}\|_{H^1_xL^2_v}^2
	+(\sqrt{\E_{k_0}(t)}+\E_{k_0}(t))\D_k(t),
\end{align}
for some generic constant $\lam>0$, where $\D_k(t)$ is given by \eqref{DefD}. 
Then one can check \eqref{DefE} by using \eqref{ing}, \eqref{EE}, \eqref{41a}  with sufficiently small $\kappa$. Note that $e^{\frac{\pm A_{\al,0}\phi}{2 \<v\>^2}}\approx 1$ as in \eqref{214d}.

\smallskip 
Next, recall from  Lemma \ref{L420} that 
\[
N_{\|}=(\pm\nabla_x\phi\cdot\nabla_vf_\pm)_{\|} + Q(f_\pm+f_\mp,f_\pm)_{\|} = (\pm\nabla_x\phi\cdot\nabla_vf_\pm, \xi) + (Q(f_\pm+f_\mp,f_\pm), \xi),
\]
where $\xi$ is some linear combination of $1,v_i,  v_i v_j, v_i^2,v_i|v|^2$. For $|\al|\le 1$, we have 
\begin{align}\label{450}\notag
	\Vert \partial^\alpha (\nabla_x \phi \cdot \nabla_v f, \xi)_{L^2_v} \Vert_{L^2_x}  &\lesssim \sum_{|\al_1|=0}\Vert \partial^{\alpha_1} \nabla_x \phi \Vert_{L^\infty_x} \Vert \partial^{\alpha-\alpha_1}\na_v f \Vert_{L^2_x L^2_5}\notag+\sum_{|\al_1|=1}\Vert \partial^{\alpha_1} \nabla_x \phi \Vert_{L^3_x} \Vert \partial^{\alpha-\alpha_1}\na_v f \Vert_{L^6_x L^2_5}\\
	&\lesssim \|\na_x\phi\|_{H^2_x}\|f\|_{H^1_xH^1_5} \lesssim \sqrt{\E_{k_0}(t)\D_k(t)}.
\end{align}
For the Boltzmann case, by Lemma \ref{L21}, we apply similar discussion on $\al_1$ as in \eqref{450} to deduce that 
\begin{align*}
	\Vert  \partial^\alpha (Q(f_\pm+f_\mp,f_\pm),\xi) \Vert_{L^2_{x, v}}^2 &\lesssim \| f\Vert_{H^{2}_xL^2_{10}} \Vert  f \Vert_{H^{2}_xL^2_{10}}\Vert \xi \Vert_{H^{2s}_{-7}  } \lesssim  \sqrt{\E_{k_0}(t)\D_k(t)}. 
\end{align*}
For the Landau case, by \eqref{27} we have 
\begin{align*}
	\Vert  \partial^\alpha (Q(f_\pm+f_\mp,f_\pm),\xi\<v\>^{-14}\<v\>^{14}) \Vert_{L^2_{x, v}}^2 &\lesssim \| f\Vert_{H^{2}_xL^2_{7}} \Vert  f \Vert_{H^{2}_xL^2_{D,8}}\Vert \xi \Vert_{L^2_{D,-7}  } \lesssim  \sqrt{\E_{k_0}(t)\D_k(t)}. 
\end{align*}
Therefore, we obtain $\sum_{|\alpha| \le 1}\Vert \partial^\alpha N_{\parallel}\Vert_{L^2_{x}}^2 \lesssim \E_{k_0}(t)\D_k(t)$.
Then \eqref{447} implies 
\begin{align}\label{448}
	\pa_t\E_{k}(t)+\lam\D_k(t)\lesssim \|f\|_{X_{k_0}}\E_{k}(t)+ (\sqrt{\E_{k_0}(t)}+\E_{k_0}(t))\D_k(t). 
\end{align}
%
%
%
%
	Choosing $M$ in the {\it a priori} assumption \eqref{priass} and \eqref{priass2} sufficiently small, we have from \eqref{448} that $\pa_t\E_{k}(t)+\lam\D_k(t)\lesssim \|f\|_{X_{k_0}}\E_{k}(t)$
for some $\lam>0$. 
This concludes Theorem \ref{T46}. 
\end{proof}

To conclude Theorem \ref{globaldecay}, we need to prove the large-time behavior as the following.
\begin{thm}\label{T47}
	Let $l=0$ for hard potential case and $l>\frac{|\gamma|}{2}$ for soft potential case. 
	Let $k\ge k_0+2+l$ (and let $k$ sufficiently large for Boltzmann case). 
If the solution $(f,\phi)$ of \eqref{vplocal1} satisfies $\E_{k_0+2+l}(0) \le M,\,\,\E_k(0)<\infty.$ Then there exists a constant $\lambda>0$, such that for any $t>0$, we have 
\begin{align}
	\label{453}
	\sup_{0\le t\le T}\E_k(t)\le 2\E(0). 
\end{align}
Moreover, we have $\E_{k}(t)\le e^{-\lam t}\E_k(0)\quad\text{if}\quad 0\le\gamma\le 1, \quad \E_{k-l}(t) \lesssim (1+t)^{-\frac{2l}{|\gamma|}}\E_{k}(0)\quad\text{if} \quad \gamma<0$.

\end{thm}

\begin{proof}
If $0\le \gamma\le 1$, noticing that $\|\cdot\|_{L^2_{x,v}}\lesssim \|\cdot\|_{L^2_D}$ and choosing $M$ in \eqref{priass} small enough, we have $\|f\|_{X_{k_0}}\lesssim\E_{k_0}\lesssim M$, and hence, by \eqref{45} we have $\pa_t\E_k(t) + \lam\E_k(t)\le 0$. Solving this ODE, we obtain $\E_k(t)\le e^{-\lam t}\E_{k}(0)$. This closes the {\it a priori} assumption \eqref{priass} and concludes the case of hard potential. 
	
	\smallskip
Next, we assume $\gamma<0$. For any $l>\frac{|\gamma|}{2}$, we have from \eqref{priass2} that 
\begin{align}\label{454}
	\|f\|_{X_{k_0}}\lesssim \E_{k_0}(t) \lesssim M(1+t)^{-\frac{2l}{|\gamma|}}. 
\end{align}
Let $p=\frac{-\gamma+2l}{2l}$ and $p'=\frac{-\gamma+2l}{-\gamma}$. Then by $L^p-L^{p'}$ H\"{o}lder's inequality, we have 
\begin{align*}
	\|w(\al,\beta)\pa^\al_\beta f\|^2_{L^2_{v}}&= \int_{\R^3}\<v\>^{\frac{2l\gamma}{-\gamma+2l}}\<v\>^{-\frac{2l\gamma}{-\gamma+2l}}|w(\al,\beta)\pa^\al_\beta f|^2\,dv\\
	&\le \|\<v\>^{\gamma/2}w(\al,\beta)\pa^\al_\beta f\|_{L^2_v}^{\frac{4l}{-\gamma+2l}}\|\<v\>^{l}w(\al,\beta)\pa^\al_\beta f\|_{L^2_v}^{\frac{-2\gamma}{-\gamma+2l}}\le \D_k^{\frac{2l}{-\gamma+2l}}\E_{k+l}^{\frac{-\gamma}{-\gamma+2l}}. 
\end{align*}
From definition \eqref{DefE} and \eqref{E}, we have $\E^{\frac{-\gamma+2l}{2l}}_k\E_{k+l}^{\frac{\gamma}{2l}}\le \D_k$. Then it follows from \eqref{45} and \eqref{454} that 
\begin{align}\label{455}
	\pa_t\E_{k}(t) + \lam\E^{\frac{-\gamma+2l}{2l}}_k(t)\Big(\sup_{0\le t\le T}\E_{k+l}\Big)^{\frac{\gamma}{2l}} \le CM(1+t)^{-\frac{2l}{|\gamma|}}\E_{k}(t). 
\end{align}
Neglecting the second left hand term of \eqref{455}, we have  $\pa_t\E_{k}(t) \le CM(1+t)^{-\frac{2l}{|\gamma|}}\E_{k}(t)$. Taking integration over $t\in[0,T]$ yields 
\begin{align*}
	\sup_{0\le t\le T}\E_k(t)\le \E_k(0)+ CM\int^T_0(1+t)^{-\frac{2l}{|\gamma|}}\,dt \sup_{0\le t\le T}\E_k(t).
\end{align*} 
Since $\frac{2l}{|\gamma|}>1$, choosing $M>0$ sufficiently small, we have 
\begin{align}\label{456}
	\sup_{0\le t\le T}\E_k(t) \le 2\E_k(0). 
\end{align}
This gives \eqref{453}. 
Next, we solve \eqref{455} directly.  It's direct to obtain that 
\begin{align*}
	\pa_t(\E_k^{\frac{\gamma}{2l}}(t)) = \frac{\gamma}{2l}\E_k^{\frac{\gamma-2l}{2l}}(t)\pa_t\E_k(t)
	\ge -\frac{\gamma\lam}{2l}\Big(\sup_{0\le t\le T}\E_{k+l}\Big)^{\frac{\gamma}{2l}} + \frac{CM\gamma}{2l}(1+t)^{-\frac{2l}{|\gamma|}}\E_k^{\frac{\gamma}{2l}}(t),
\end{align*}
and thus
\begin{multline*}
	\pa_t\big(\exp\big\{\frac{CM|\gamma|^2}{2l(-2l+|\gamma|)}(1+t)^{-\frac{2l}{|\gamma|}+1}\big\}\E_k^{\frac{\gamma}{2l}}(t)\big)
	\ge -\frac{\gamma\lam}{2l}\exp\big\{\frac{CM|\gamma|^2}{2l(-2l+|\gamma|)}(1+t)^{-\frac{2l}{|\gamma|}+1}\big\}\Big(\sup_{0\le t\le T}\E_{k+l}\Big)^{\frac{\gamma}{2l}}. 
\end{multline*}
Note that $-2l+|\gamma|<0$. 
Taking integration over $t\in[0,T]$, we have 
\begin{align*}
	\exp\big\{\frac{CM|\gamma|^2}{2l(-2l+|\gamma|)}(1+t)^{-\frac{2l}{|\gamma|}+1}\big\}\E_k^{\frac{\gamma}{2l}}(t)
	&\ge C_{\gamma,l}\E_k^{\frac{\gamma}{2l}}(0) + C_{\gamma,l}\int^t_01\,dt \Big(\sup_{0\le t\le T}\E_{k+l}\Big)^{\frac{\gamma}{2l}},
\end{align*}
and hence $\E_k(t) \le C_{\gamma,l}(1+t)^{\frac{2l}{\gamma}}\sup_{0\le t\le T}\E_{k+l}(t)$. Replacing $k$ by $k-l\ge k_0$, we apply \eqref{456} to deduce  
\begin{align*}
	\E_{k-l}(t) \lesssim (1+t)^{\frac{2l}{\gamma}}\sup_{0\le t\le T}\E_{k}(t)\lesssim (1+t)^{\frac{2l}{\gamma}}\E_k(0). 
\end{align*}
Noticing $\gamma<0$, choosing $k-l=k_0$ and applying \eqref{small}, we close the {\it a priori} assumption \eqref{priass2}. Using the standard continuity arguments, we complete the proof of Theorem \ref{T47}.
\end{proof}

\bigskip{\bf Acknowledgments.}  C.-Q. Cao has been supported by grants from Beijing Institute of Mathematical Sciences and Applications and Yau Mathematical Science Center, Tsinghua University. D.-Q. Deng was partially supported by the National Research Foundation of Korea (NRF) grant funded by the Korea government (MSIT) No. RS-2023-00210484 and No. RS-2023-00212304. X.-Y. Li has been supported by grants from project ANR-17-CE40-0030
of the French National Research Agency (ANR) and CEREMADE, Universit\'e Paris Dauphine.

\medskip

\noindent{\bf Conflict of Interest:} The authors declare that they have no conflict of interest.


\end{document}